\documentclass{amsart}
\RequirePackage[utf8]{inputenc}
\usepackage{amsmath,amsthm,amsfonts,amssymb,amscd,amsbsy,multirow,tikz,hyperref,microtype}
\usepackage[all]{xy}
\usepackage{graphicx,pgf,caption,subcaption}
\usepackage{enumerate}
\usepackage{mathdots}
\usepackage{dsfont}
\usepackage{cleveref}
\usepackage{stmaryrd}

\newcommand{\conv}{\operatorname{conv}}
\newcommand{\Vol}{\operatorname{Vol}}
\newcommand{\diam}{\operatorname{diam}}
\newcommand{\aff}{\operatorname{aff}}
\newcommand{\V}{{\rm V}}

\newcommand{\id}{\operatorname{Id}}
\newcommand{\vol}{\operatorname{vol}}

\newcommand{\Ric}{\operatorname{Ric}}
\newcommand{\scal}{\operatorname{scal}}
\newcommand{\Z}{\mathds Z}

\newcommand{\R}{\mathds R}
\newcommand{\C}{\mathds C}

\newcommand{\SO}{\mathsf{SO}}
\renewcommand{\O}{\mathsf O}

\newcommand{\U}{\mathsf{U}}

\newcommand{\Sym}{\operatorname{Sym}}

\newcommand{\g}{\mathrm g}
\newcommand{\tr}{\operatorname{tr}}
\newcommand{\Gr}{\operatorname{Gr}_2}

\newcommand{\diag}{\operatorname{diag}}

\allowdisplaybreaks

\hypersetup{
    pdftoolbar=true,
    pdfmenubar=true,
    pdffitwindow=false,
    pdfstartview={FitH},
    pdftitle={Geography of pinched four-manifolds},
    pdfauthor={Renato G. Bettiol, Mario Kummer, and Ricardo A. E. Mendes},
    pdfsubject={},
    pdfkeywords={}
    pdfnewwindow=true,
    colorlinks=true, %set this false for printable version
    linkcolor=blue,
    citecolor=blue,
    urlcolor=black,
}

\newtheorem{theorem}{Theorem}[]
\newtheorem{lemma}[theorem]{Lemma}

\newtheorem{proposition}[theorem]{Proposition}
\newtheorem{corollary}[theorem]{Corollary}

\newtheorem{mainthm}{\sc Theorem}

\newtheorem{maincor}[mainthm]{\sc Corollary}

\theoremstyle{definition}
\newtheorem{definition}[theorem]{Definition}

\theoremstyle{remark}
\newtheorem{remark}[theorem]{Remark}

\newtheorem{example}[theorem]{Example}

\title{Geography of pinched four-manifolds}

\author[R. G. Bettiol]{Renato G. Bettiol}
\address{\!\!\!\begin{tabular}{lll}
CUNY Lehman College & & CUNY Graduate Center \\
Department of Mathematics & & Department of Mathematics \\
250 Bedford Park~Blvd W & & 365 Fifth Avenue \\
Bronx, NY, 10468, USA & & New York, NY, 10016, USA
\end{tabular}
}
\email{r.bettiol@lehman.cuny.edu}

\author[M. Kummer]{Mario Kummer}
\address{Technische Universit\"at Dresden \newline
\indent Fakult\"at Mathematik \newline
\indent Institut f\"ur Geometrie \newline
\indent Zellescher Weg 12-14 \newline
\indent 01062 Dresden, Germany}
\email{mario.kummer@tu-dresden.de}

\author[R. A. E. Mendes]{Ricardo A. E. Mendes}
\address{University of Oklahoma\newline
\indent Department of Mathematics\newline
\indent 601 Elm Ave\newline
\indent Norman, OK, 73019-3103, USA}
\email{ricardo.mendes@ou.edu}

\newcommand{\dpinched}{\Omega_\delta}
\newcommand{\epinched}{\mathrm{E}_\delta}
\newcommand{\esimp}{\Delta_\delta}

\newcommand{\pr}{\mathrm{pr}}

\allowdisplaybreaks
\numberwithin{equation}{section}
\numberwithin{table}{section}
\numberwithin{theorem}{section}


\date{July 24, 2022}

\begin{document}
\begin{abstract}
We prove several new restrictions on the Euler characteristic and signature of oriented $4$-manifolds with (positively or negatively) pinched sectional curvature. In particular, we show that simply connected $4$-manifolds with $\delta\leq \sec\leq 1$, where $\delta=\frac{1}{1+3\sqrt3}\cong 0.161$, are homeomorphic to $S^4$ or $\C P^2$.
\end{abstract}

\maketitle

\section{Introduction}
The so-called \emph{Geography Problem} in $4$-manifold topology is to determine which pairs $(\sigma,\chi)\in\Z^2$ can be realized as signature $\sigma=\sigma(M)$ and Euler characteristic $\chi=\chi(M)$ of a certain class of $4$-manifolds $M$; e.g., complex surfaces~\cite{persson}, irreducible $4$-manifolds~\cite{fin-stern}, or those with a given fundamental group~\cite{kirk-geography}. In this paper, we systematically investigate a geometric version of this problem, where the condition imposed on $M$ is the existence of a Riemannian metric with pinched curvature. We say that a Riemannian manifold $(M,\g)$ is \emph{positively} or \emph{negatively $\delta$-pinched}, $\delta\in (0,1]$, if the sectional curvature of every tangent 2-plane $\Pi$ satisfies
\begin{equation*}
 \delta\leq\sec(\Pi)\leq1, \quad\text{ or }\quad  -1\leq \sec(\Pi) \leq -\delta,
\end{equation*}
respectively. Collectively, $(M,\g)$ is called \emph{$\delta$-pinched} if it is either positively or negatively $\delta$-pinched. 
It is well-known that $\delta$-pinched $4$-manifolds have $\chi(M)>0$; our first main result provides an explicit upper bound for the ratio $|\sigma(M)|/\chi(M)$:

\begin{mainthm}\label{mainthm:Lambda}
If $(M^4,\g)$ is a $\delta$-pinched oriented $4$-manifold with finite volume, then
\begin{equation}\label{eq:mainthm}
|\sigma(M)|\leq \lambda(\delta)\,\chi(M),    
\end{equation}
where $\lambda\colon (0,1]\to\R$ is an \emph{explicit} continuous function, given in \eqref{eq:bestlambda}, that is strictly decreasing and satisfies $\lim\limits_{\delta\searrow0}\lambda(\delta)=+\infty$, $\lambda\!\left(\tfrac{1}{1+3\sqrt3}\right)<\tfrac12$, $\lambda\!\left(\tfrac14\right)=\tfrac13$, and $\lambda(1)=0$.
\end{mainthm}

In the above statement, and throughout this paper, all manifolds are assumed complete and without boundary. If $(M^4,\g)$ is negatively $\delta$-pinched, then $M$ need not be closed, in which case $\sigma(M)$ is to be understood as the proper homotopy invariant given by the $L^2$-signature $\sigma_{(2)}(M^4,\g)$, see Section \ref{subsec:integrals} and~\cite{cheeger-gromov1,cheeger-gromov2}.

The fact that $\lambda$ is an \emph{explicit} function of $\delta$ is the crucial component of Theorem~\ref{mainthm:Lambda}, as the existence of \emph{some} function $\lambda\colon(0,1]\to\R$ satisfying \eqref{eq:mainthm} and $\lambda(1)=0$ can be shown with routine arguments.
Theorem~\ref{mainthm:Lambda} is a concoction of our main technical result (Theorem~\ref{thm:estimates-general})
and a thorough extension (Theorem~\ref{thm:ville}) of the seminal works of Ville~\cite{ville-negative,ville-positive}, carried out in Appendix~\ref{app:ville}. While Theorem~\ref{thm:estimates-general} gives a continuously differentiable function $\lambda^*\colon(0,1]\to\R$ satisfying \eqref{eq:mainthm} and all other conditions in Theorem~\ref{mainthm:Lambda} except $\tfrac14\mapsto \tfrac13$, Theorem~\ref{thm:ville} yields a continuous function $\lambda^\V\colon[\delta_0^\V,1]\to\R$, where $\delta_0^\V\cong 0.163$, satisfying \eqref{eq:mainthm} and $\tfrac14\mapsto \tfrac13$. Combining these, we have the function $\lambda(\delta)=\min\{\lambda^*(\delta),\lambda^\V(\delta)\}$ in Theorem~\ref{mainthm:Lambda}; see Section~\ref{subsec:explicit} for details. 
Theorem~\ref{mainthm:Lambda} also improves a similar result of Polombo~\cite{polombo} for $0<\delta\leq\frac{1}{4}$.

Since $\delta$-pinching becomes less stringent as $\delta\searrow0$, it is natural that a function $\lambda\colon(0,1]\to\R$ satisfying \eqref{eq:mainthm} be decreasing and $\lambda(0_+)=+\infty$. % every closed manifold with $\sec>0$ or $\sec<0$ is $\delta$-pinched for some $\delta>0$.
On the other hand, $\lambda(1)=0$ indicates a certain sharpness regarding \eqref{eq:mainthm}, since $1$-pinched manifolds are either spherical or hyperbolic spaceforms, which, being locally conformally flat, have vanishing signature. Also $\lambda\!\left(\tfrac14\right)=\tfrac13$ is sharp, as the complex projective plane $M=\C P^2$ and compact quotients $M=\C H^2/\Gamma$ of the complex hyperbolic plane are all $\frac{1}{4}$-pinched and have $\sigma(M)=\tfrac{1}{3}\chi(M)$.

Let us further analyze the positively and negatively pinched cases separately.

\subsection{Positively pinched 4-manifolds}\label{subsec:pos}
By the celebrated $1/4$-pinched Sphere Theorem and its several improvements~\cite{brendle-schoen,petersen-tao}, it is known that there exists $\varepsilon>0$ such that if an oriented $4$-manifold $(M^4,\g)$ is $\delta$-pinched with $\delta\geq \tfrac14-\varepsilon$, then $M$ is diffeomorphic to $S^4$ or $\C P^2$. Moreover, every positively $\delta$-pinched oriented $4$-manifold $M$ is closed and has $b_1(M)=0$, hence $|\sigma(M)| < \chi(M)$, see \eqref{eq:basic_geography}.

Thus, in this case, Theorem~\ref{mainthm:Lambda} only provides new information in the range 
\begin{equation}\label{eq:pos-range}
\lambda^{-1}(1)<\delta<\tfrac14-\varepsilon,
\end{equation}
where $\lambda^{-1}(1)=\frac{39-5\sqrt{57}}{24}\cong 0.052$, according to \eqref{eq:bestlambda}. A particularly interesting value in the above range is $\tfrac{1}{1+3\sqrt3}\cong 0.161$, since positively $\tfrac{1}{1+3\sqrt3}$-pinched oriented $4$-manifolds were recently shown by Di\'ogenes and Ribeiro~\cite[Thm.~1]{DR19} to have definite intersection form. Combined with Theorem~\ref{mainthm:Lambda}, that gives $|\sigma(M)|<\frac{1}{2}\chi(M)$, and the classical works of Donaldson~\cite{Donaldson83} and Freedman~\cite{Freedman82}, this yields: % our next main~result.

\begin{mainthm}\label{mainthm:classification}
If $(M^4,\g)$ is a positively $\delta$-pinched simply-connected $4$-mani\-fold, with
$\delta\geq \frac{1}{1+3\sqrt{3}}\cong0.161$, then $M$ is homeomorphic to $S^4$ or $\C P^2$.
\end{mainthm}

Theorem~\ref{mainthm:classification} improves on results of Ville~\cite{ville-positive}, Seaman~\cite{seaman}, and Ko~\cite{ko2}, where the same conclusion is obtained under stricter $\delta$-pinching:
$\delta\geq\frac{4}{19}\cong0.211$, $\delta\geq 0.188$, and $\delta\geq 0.171$, respectively. Although it is widely expected that closed simply-connected $4$-manifolds $(M^4,\g)$ with $\sec>0$ be diffeomorphic to $S^4$ or $\C P^2$ (see e.g.~\cite{ziller}), this remains a difficult open problem. Perhaps the most compelling evidence for this conjecture is that it holds if $(M^4,\g)$ has an isometric circle action~\cite{hk,grove-wilking}. From this point of view, Theorem~\ref{mainthm:classification} provides new evidence \emph{without any symmetry assumptions}, cf.~also~\cite[Thm.~C]{bm-iumj}.

Our next main result gives upper bounds on the region in the $(|\sigma|,\chi)$-plane reachable by positively $\delta$-pinched $4$-manifolds, refining an observation of Berger~\cite{bergerpinch} that such $4$-manifolds are either homology spheres or have $\chi(M)\leq\frac{1}{\delta^2}+\frac{8}{27}\left(\frac{1}{\delta}-1\right)^2$.

\begin{mainthm}\label{mainthm:bounds-pos}
If $(M^4,\g)$ is a positively $\delta$-pinched oriented $4$-manifold, 
then either $M$ is diffeomorphic to $S^4$, or 
$\chi(M)\leq \frac{8}{9}(\frac{1}{\delta}-1)^2$ and $|\sigma(M)|\leq \frac{8}{27}(\frac{1}{\delta}-1)^2$.
\end{mainthm}

To the best of our knowledge, no restrictions on $\sigma(M)$, aside those inherited from $|\sigma(M)|\leq\chi(M)-2$, were previously known for positively $\delta$-pinched manifolds.

Even though the upper bounds in Theorem~\ref{mainthm:bounds-pos} diverge to $+\infty$ as $\delta\searrow0$, $\chi(M)$ and $|\sigma(M)|$ are known to be bounded above by a universal constant $C(4)$ for any closed $4$-manifold $(M^4,\g)$ with $\sec>0$. This is a consequence of the celebrated total Betti number bound of Gromov~\cite{gromov-total}, see Remark~\ref{rem:gromov-abresch} for details. 
Gromov conjectured that $C(4)=2^4$, but the best available estimates (due to Abresch~\cite{abreschII}) only give $C(4)\lesssim 2.731\times 10^{232}$. The latter is a smaller upper bound on $\chi(M)$ than that in Theorem~\ref{mainthm:bounds-pos} if $\delta\lesssim 5.705\times 10^{-117}$, but it is \emph{hundreds} of orders of magnitude larger if, e.g., $\delta\gtrsim0.086$, in which case Theorem~\ref{mainthm:bounds-pos} gives $\chi(M)\leq 10^2$. Still, according to the conjectured classification of closed simply-connected $4$-manifolds with $\sec>0$, one ought to have $\chi(M)\leq 3$ and $|\sigma(M)|\leq 1$, no matter how small is $\delta>0$.

\begin{figure}
\begin{tikzpicture}[scale=0.75]
    \fill[opacity=0.35,green] (0,0.5) -- (0,9) -- (3,9) -- (3,5.43) -- (0.617,1.12);

    \draw[->, line width=1] (0,0) -- (10,0);
    \draw[->, line width=1] (0,0) -- (0,10);
    \draw[-, line width=1] (0,0.5) -- (9.25,9.75); %slope 1
    \draw[-, line width=1] (3,9.75) -- (3,0);
    \draw[-, line width=1] (0,9) -- (9.75,9);
    \draw[-, line width=1] (0,0) -- (5.386,9.75); % slope 1/lambda(d) = 1.81

    \node (chi) at (0,10.5) {$\chi$};
    \node (sigma) at (10.5,0) {$|\sigma|$};

    \foreach \x in {0,0.5,...,9.5} {
        \foreach \y in {0,0.5,...,9.5} {
            \fill[color=black] (\x,\y) circle (0.04);
        }
    }
    \foreach \x in {0.25,0.75,...,9.5} {
        \foreach \y in {0.25,0.75,...,9.5} {
            \fill[color=black] (\x,\y) circle (0.04);
        }
    }
        \node (m) at (-0.25,0.5) {$2$};
        \node (m) at (9.25,10.1) {$\chi\geq|\sigma|+2$};
     \node (M) at (-1.25,9) {$\frac{8}{9}\left(\frac{1}{\delta}-1\right)^2$};
    \node (M) at (3,-.5) {$\frac{8}{27}\left(\frac{1}{\delta}-1\right)^2$};
    \node (m) at (5.25,10.1) {$\frac{|\sigma(M)|}{\chi(M)}\leq \lambda(\delta)$};
\end{tikzpicture}
\caption{Admissible range (green) for $\big(|\sigma(M)|,\chi(M)\big)$ if $M$ is positively $\delta$-pinched by \eqref{eq:basic_geography} and Theorems~\ref{mainthm:Lambda} and \ref{mainthm:bounds-pos}; drawn with $\delta=\frac{9\sqrt2-2}{79}\cong 0.136$, so $\lambda(\delta)\cong 0.552$, $\chi(M)\leq 36$, and $\sigma(M)\leq 12$.}\label{fig:geography}
\end{figure}
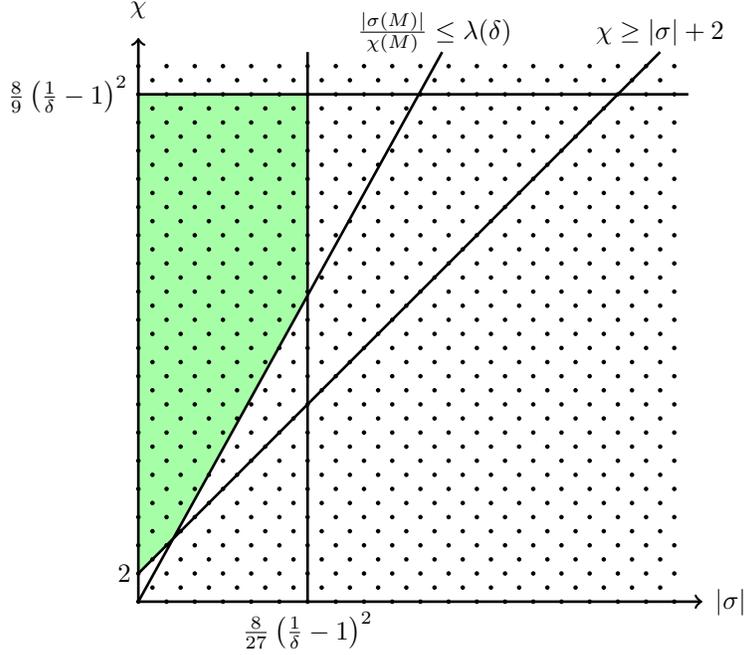

The admissible region in the $(|\sigma|,\chi)$-plane for positively $\delta$-pinched $4$-manifolds, as constrained by Theorems~\ref{mainthm:Lambda} and \ref{mainthm:bounds-pos}, is illustrated in Figure~\ref{fig:geography}.
By Synge's Theorem, such a $4$-manifold $M$ has $\pi_1(M)\cong\{1\}$ if it is oriented, and $\pi_1(M)\cong\Z_2$ otherwise. Thus, resorting again to the works of Donaldson~\cite{Donaldson83} and Freedman~\cite{Freedman82} for the simply-connected case, and to Hambleton, Kreck, and Teichner~\cite{hkt} for the non-simply-connected case, one derives the following from Theorems~\ref{mainthm:Lambda} and \ref{mainthm:bounds-pos}.

\begin{maincor}\label{cor:homeotypes}
For all $\delta>0$, there exists an \emph{explicit} finite list of possible homeomorphism types for positively $\delta$-pinched $4$-manifolds.
\end{maincor}

While Corollary~\ref{cor:homeotypes} also follows from the above mentioned bound on $\chi(M)$ for each $\delta>0$ by Berger~\cite{bergerpinch}, together with $|\sigma(M)|+2\leq \chi(M)$, the explicit list of homeomorphism types obtained using 
Theorems~\ref{mainthm:Lambda} and \ref{mainthm:bounds-pos} is substantially shorter.

\subsection{Negatively pinched 4-manifolds}\label{subsec:neg}
In light of constructions of Gromov and Thurston~\cite{gromov-thurston}, classification results similar to the $1/4$-pinched Sphere Theorem are impossible in the negatively pinched case. Indeed, for all $\varepsilon>0$, there are closed negatively $(1-\varepsilon)$-pinched $4$-manifolds that do not admit metrics of constant curvature. While the examples in~\cite{gromov-thurston} have zero signature, similar examples with \emph{nonzero signature} were recently found by Ontaneda~\cite[Cor.~4]{ontaneda}. Thus, in stark contrast to the positively pinched case~\eqref{eq:pos-range}, the conclusion of Theorem~\ref{mainthm:Lambda} in the negatively pinched case is nontrivial in the full range $0<\delta<1$.
As $\delta\nearrow1$, it follows from Theorem~\ref{mainthm:Lambda} that negatively $\delta$-pinched oriented $4$-manifolds $(M^4,\g)$ with $\sigma(M)\neq0$
must have $\chi(M)\nearrow+\infty$; namely, by~\eqref{eq:bestlambda}, we have that
\begin{equation*}
\chi(M)\geq \dfrac{24 \delta^2-12 \delta+15}{8(1-\delta)^2}
\end{equation*}
if $\delta$ is sufficiently close to $1$.
In particular, solving for $\delta$, one sees that closed oriented $4$-manifolds $M$ with $\sigma(M)\neq0$ and fixed Euler characteristic $\chi(M)=\chi$ can only admit metrics that are negatively $\delta$-pinched if $\delta\leq\delta_\chi<1$, where $\delta_\chi$ is explicit.
A similar gap $\delta\leq \delta_\pi$ holds~\cite[Cor.~1.3]{belegradek} weakening $\sigma(M)\neq0$ to $M$ not admitting hyperbolic metrics, and fixing $\pi_1(M)\cong\pi$ instead of $\chi(M)$, but $\delta_\pi$ is not explicit.

We also remark that although the lower bound \eqref{eq:mainthm} becomes arbitrarily weak as $\delta\searrow0$, it follows from Gromov~\cite[\S 1.7 (2)]{gromov-neg} that
there exists $C>0$ such that, for all $\delta>0$, negatively $\delta$-pinched closed $4$-manifolds with $\sigma(M)\neq0$ have $\chi(M)\geq C$. Combined with \eqref{eq:mainthm}, one has the lower bound
\begin{equation*}
\chi(M)\geq \max\left\{\frac{|\sigma(M)|}{\lambda(\delta)},C\right\},
\end{equation*}
which remains uniformly away from zero for all $\delta$, and diverges to $+\infty$ as $\delta\nearrow1$.

Another striking difference arises from negatively $\delta$-pinched closed $4$-manifolds $M$ not necessarily having $b_1(M)=0$. However, as these manifolds have $\chi(M)>0$,
\begin{equation}\label{eq:trivialb1}
b_1(M) \leq 1+ \tfrac12 b_+(M)+\tfrac12 b_-(M). 
\end{equation}
A consequence of Theorem~\ref{mainthm:Lambda} is that this upper bound can be improved to
\begin{equation}\label{eq:b1bound}
b_1(M)\leq 1+  \tfrac{\lambda-1}{2\lambda}\max\!\big\{b_{+}(M),b_-(M)\big\} +\tfrac{\lambda+1}{2\lambda}\min\!\big\{b_+(M),b_-(M)\big\},
\end{equation}
where $\lambda=\lambda(\delta)>0$. Note that the above weighted average of $b_\pm(M)$ is strictly smaller than the simple average in \eqref{eq:trivialb1} whenever $\sigma(M)\neq0$.

Similarly to Theorem~\ref{mainthm:bounds-pos}, upper bounds (depending on volume) can be given on the admissible region in the $(|\sigma|,\chi)$-plane for negatively $\delta$-pinched $4$-manifolds:

\begin{mainthm}\label{mainthm:bounds-neg}
If $(M^4,\g)$ is a negatively $\delta$-pinched oriented $4$-manifold with finite volume, then
\begin{equation*}
\chi(M)\leq \frac{3}{4\pi^2} \Vol(M,\g), \; \text{ and } \; |\sigma(M)|\leq \frac{2}{9\pi^2}(1-\delta)^2 \Vol(M,\g).
\end{equation*}
Equality in the upper bound for $\chi(M)$ is achieved if and only if $(M^4,\g)$ is hyperbolic.
\end{mainthm}

The above upper bound on $\chi(M)$ is likely known among experts (e.g., it follows from~\cite[Thm.~1]{ville-vol}), but we are unaware of any such prior results for $\sigma(M)$. As in Theorem~\ref{mainthm:Lambda}, $\sigma(M)$ is to be understood as the $L^2$-signature if $(M^4,\g)$ is noncompact.

Using Bishop Volume Comparison in the inequalities in Theorem~\ref{mainthm:bounds-neg}, we see that negatively $\delta$-pinched closed oriented $4$-manifolds 
$(M^4,\g)$ with $\diam(M,\g)\leq D$ have 
\begin{equation*}
\textstyle
\chi(M)\leq 2(2+\cosh D)\sinh^4\frac{D}{2}, \; \text{ and } \; |\sigma(M)|\leq \frac{16}{27}(1-\delta)^2(2+\cosh D)\sinh^4\frac{D}{2},
\end{equation*}
and, once again, equality in the upper bound for $\chi(M)$ holds if and only if $(M^4,\g)$ is hyperbolic.
It follows from Gromov~\cite[\S 1.7 (1)]{gromov-neg} that, for all $D>0$ and $V>0$, there exists $0<\delta(D,V)<1$ such that if a closed $4$-manifold $M$ is negatively $\delta(D,V)$-pinched and $\sigma(M)\neq0$, then $\diam(M,\g)\geq D$ and $\Vol(M,\g)\geq V$. Note that Theorem~\ref{mainthm:bounds-neg} \emph{quantifies} this result, yielding explicit estimates on $\delta(D,V)$, since it implies that closed $\delta$-pinched $4$-manifolds with $\sigma(M)\neq0$ have $\Vol(M,\g)\geq \frac{9\pi^2}{2(1-\delta)^2}$ and $\diam(M,\g)\geq D$ where $D>0$ satisfies $(2+\cosh D)\sinh^4\frac{D}{2}=\frac{27}{16(1-\delta)^2}$.

\subsection{Methods of proof}
Our results on pinched $4$-manifolds are proven with a blend of Differential Geometry and Convex Algebraic Geometry. By the Chern--Gauss--Bonnet formula and Hirzebruch signature formula, the Euler characteristic and signature of $(M^4,\g)$ can be computed as integrals \eqref{eq:integrals} of quadratic forms $\underline{\chi}(R)$ and $\underline{\sigma}(R)$ on its curvature operator $R$, respectively. (These formulas generalize to the case in which $M$ is noncompact, see Section~\ref{subsec:integrals}.) We use these to estimate $\chi(M)$ and $\sigma(M)$ combining pointwise bounds on such integrands, obtained through optimization methods; and global restrictions on the diameter and volume of such manifolds, obtained with standard Comparison Geometry techniques. Moreover, since $\underline{\chi}(R)$ and $\underline{\sigma}(R)$ only have degree 2 terms, 
any pointwise bounds on \emph{positively} $\delta$-pinched operators automatically hold for \emph{negatively} $\delta$-pinched operators.

The key link with Convex Algebraic Geometry is the Finsler--Thorpe Trick (Lemma~\ref{lem:finslerthorpe}), a distinctly $4$-dimensional phenomenon (see~\cite{cagco}) that characterizes the set of positively $\delta$-pinched (algebraic) curvature operators $R\colon\wedge^2\R^4\to\wedge^2\R^4$ as a \emph{spectrahedral shadow}. 
More precisely, it is the intersection of the projections onto the space $\Sym^2_b(\wedge^2\R^4)$ of the \emph{spectrahedra} of operators in $\Sym^2(\wedge^2\R^4)$ with all eigenvalues $\geq\delta$ and $\leq 1$, respectively. Recall that $\Sym^2_b(\wedge^2\R^4)\subset \Sym^2(\wedge^2\R^4)$ is the subspace of operators satisfying the first Bianchi identity, and its orthogonal complement is spanned by the Hodge star operator $*$. 
Moreover, a spectrahedral shadow is (by definition) a linear projection of a spectrahedron, and the intersection of two spectrahedral shadows is also a spectrahedral shadow.

As explained in Section~\ref{sec:simplex}, it also follows from Finsler--Thorpe's Trick that projecting away the traceless Ricci part $R_{\mathcal L}$ of a pinched curvature operator~$R$ produces an
Einstein curvature operator $R-R_{\mathcal L}$ which is \emph{at least} as pinched as~$R$.
The set of such operators is a far simpler convex set: it is the orbit under the $\SO(4)$-action on $\Sym^2_b(\wedge^2\R^4)$ of a set affinely equivalent to a simplex $\esimp^5\subset\R^5$, which we call the \emph{Einstein simplex}. Similarly, there is an \emph{augmented Einstein simplex} $\esimp^6\subset\R^6$ that parametrizes $\SO(4)$-orbits of $\delta$-pinched Einstein curvature operators $R$ and the corresponding $t_1\in\R$ such that $R+t_1\,*\in \Sym^2(\wedge^2\R^4)$ has all eigenvalues~$\geq\delta$.

Thus, finding extrema of $\SO(4)$-invariant quadratic forms on $R$ that \emph{do not} depend on $R_{\mathcal L}$ reduces to a quadratic optimization problem on the simplex $\esimp^5$. Even though general quadratic programming is NP-hard~\cite{sahni, qp-np-hard}, the cases at hand are manageable. For example, we are able to explicitly compute the maximum of $a|W_+|^2+b|W_-|^2$, for any $a,b\in\R$, where $W_\pm$ are the self-dual and anti-self-dual Weyl parts of a $\delta$-pinched curvature operator, and also characterize the equality case, see Proposition~\ref{prop:weylbound}. This sharp pointwise estimate, which is likely to have other applications, is the main new input to prove Theorems \ref{mainthm:bounds-pos} and \ref{mainthm:bounds-neg}.

On the other hand, the $1$-parameter family of $\SO(4)$-invariant quadratic forms $I_\lambda(R)=\underline{\chi}(R)-\frac{1}{\lambda}\underline{\sigma}(R)$ used to prove Theorem~\ref{mainthm:Lambda},
and the particular case $I_{\frac{1}{2}}(R)$ needed for Theorem~\ref{mainthm:classification}, \emph{do depend} on the traceless Ricci part $R_{\mathcal L}$. To overcome this, we prove (Proposition~\ref{prop:boundC}) an upper bound on $|R_{\mathcal L}|^2$ for any $R\in\Sym^2_b(\wedge^2\R^4)$ in terms of its minimal (or maximal) sectional curvature $k\in\R$, $\scal$, $W_\pm$, and a value $t\in\R$ such that $\pm(R-k\id)+t\,*$ is positive-semidefinite (which exists by the Finsler--Thorpe Trick). 
Aside from its independent interest, Proposition~\ref{prop:boundC} can be used to eliminate the dependence on $R_{\mathcal L}$ and find a quadratic form $Q_\lambda\colon\esimp^6\to\R$
on the augmented Einstein simplex $\esimp^6$ whose values bound those of $I_\lambda$ from below. Then, optimizing $Q_\lambda$ on $\esimp^6$ and requesting that its minimum be nonnegative 
%Through optimization, requesting that $Q_\lambda\geq0$ on $\esimp^6$
determines an explicit description of a semialgebraic set in the $(\delta,\lambda)$-plane.
Applying Cylindrical Algebraic Decomposition to this set gives, for each $\delta$, the desired explicit lower bound on $\lambda$ that is sufficient to ensure $I_\lambda(R)\geq0$ for all $\delta$-pinched curvature operators, and hence $\sigma(M)\leq \lambda\cdot\chi(M)$ for all $\delta$-pinched $4$-manifolds $(M^4,\g)$. This yields Theorem~\ref{thm:estimates-general}, which is our main technical result leading to Theorem~\ref{mainthm:Lambda}.

\begin{figure}[!ht]
\includegraphics[width=\textwidth]{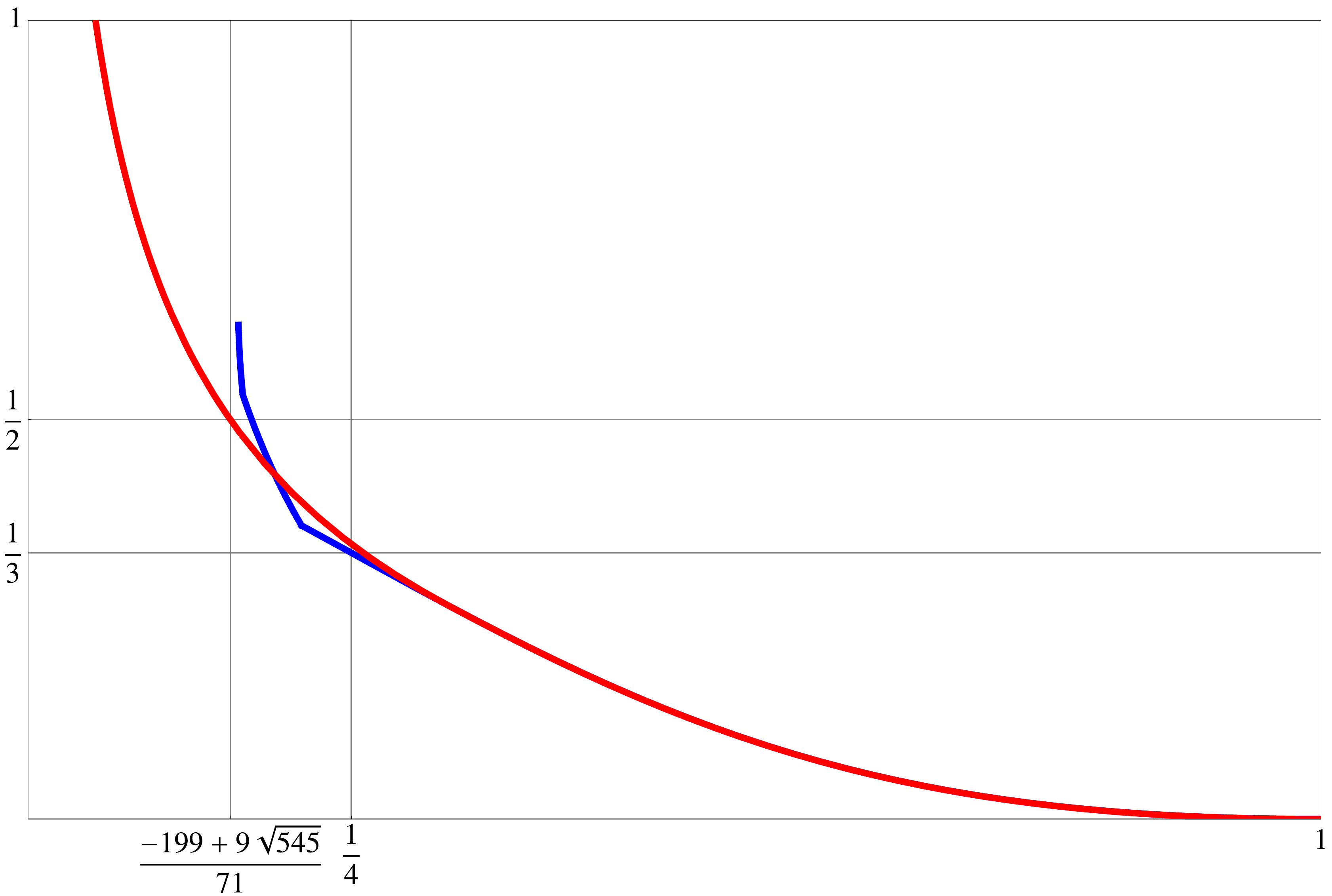}
\caption{Graphs of functions $\lambda^*$ from Theorem~\ref{thm:estimates-general} (red) and $\lambda^\V$ from Theorem~\ref{thm:ville} (blue); whose minimum is $\lambda$ given by \eqref{eq:bestlambda}.
}\label{fig:graph}
\end{figure}

\subsection{\texorpdfstring{Explicit function $\lambda(\delta)$}{Explicit function}}\label{subsec:explicit}
As mentioned above, Theorem~\ref{mainthm:Lambda} is the combination of Theorems~\ref{thm:estimates-general} and \ref{thm:ville},
each of which proves \eqref{eq:mainthm} for a certain explicit function of $\delta$. Comparing these functions and extracting $\lambda(\delta)=\min\{\lambda^*(\delta),\lambda^\V(\delta)\}$ gives:
\begin{equation}\label{eq:bestlambda} 
\lambda(\delta)=\begin{cases}
 \dfrac{\sqrt{\frac{24}{\delta}+8 -8 \delta+\delta^2}+\delta-4 }{6 (3-\delta)},
&  0<\delta<\delta_1,  \\[8pt]
\dfrac{4 }{3 \sqrt{15}} \dfrac{1-\delta}{\sqrt{\delta (\delta+2)}},
  &   \delta_1\leq \delta<\delta_2, \\[12pt]
\dfrac{26\delta^2+8\delta+2-2\sqrt3\sqrt{55\delta^4+40\delta^3+6\delta^2+8\delta-1}}{3(1-\delta)^2}, &  \delta_2\leq \delta\leq \delta_3, \\[6pt]
\dfrac{8(1-\delta)^2}{24 \delta^2-12 \delta+15}, &  \delta_3\leq \delta\leq 1,
\end{cases}
\end{equation}
where
\begin{enumerate}[\rm (i)]
\item $\delta_1\cong0.069$ is the smallest real root of the polynomial $2\delta^3-40\delta^2+89\delta-6$,  
\item $\delta_2\cong 0.191$ is the largest real root of the polynomial $2279 \delta^6 + 6246 \delta^5 + 4470 \delta^4 + 2060 \delta^3 - 450 \delta^2 - 24 \delta - 1$, 
\item $\delta_3\cong 0.211$ is the largest real root of the polynomial $140\delta^4 + 40\delta^3 - 6\delta^2 + 88\delta - 19$.
\end{enumerate}

\subsection*{Acknowledgements}
It is a great pleasure to thank
Grigori Avramidi,
Igor Belegradek, 
McFeely (Jackson) Goodman,
Karsten Grove, 
Matt Gursky, 
Claude LeBrun, 
T.~T\^am Nguy$\tilde{\hat{\mathrm{e}}}$n-Phan, and 
Marina Ville for many helpful conversations.
The first-named author was supported by the National Science Foundation grant DMS-1904342 and CAREER grant DMS-2142575;
the second-named author was supported by Deutsche Forschungsgemeinschaft grant 421473641;
the third-named author was supported by the Deutsche Forschungsgemeinschaft grants ME 4801/1-1 and SFB TRR 191, and the National Science Foundation grant DMS-2005373.

\section{Preliminaries on 4-manifolds}\label{sec:prelim4}

In this section, we discuss a multitude of basic topics regarding $4$-manifolds needed in the remainder of the paper, also fixing notations and conventions.

\subsection{Four-dimensional curvature operators}
We now briefly recall some well-known facts about curvature operators of oriented $4$-manifolds, for details, we refer the reader to \cite[Chap.~1.G-H]{Besse}.

The space $\Sym^2(\wedge^2 \R^4)$ of symmetric endomorphisms of $\wedge^2\R^4$ decomposes as the orthogonal direct sum of four irreducible pairwise non-isomorphic $\O(4)$-submodules,
\begin{equation}\label{eq:decomp}
\Sym^2(\wedge^2 \R^4)=\mathcal U \oplus \mathcal L \oplus \mathcal W\oplus \wedge^4\R^4,
\end{equation}
and, accordingly, we write $R=R_{\mathcal U}+R_{\mathcal L}+R_{\mathcal W}+R_{\wedge^4}$ to indicate the components of an element $R\in \Sym^2(\wedge^2 \R^4)$. 
The (real) dimensions of the spaces in \eqref{eq:decomp} are $1$, $9$, $10$, and $1$, respectively.
In particular, the summand $R_{\wedge^4}$ is a scalar multiple of the Hodge star operator $*\in \Sym^2(\wedge^2 \R^4)$, and it vanishes if and only if $R$ satisfies the first Bianchi identity. We denote by $\Sym^2_b(\wedge^2\R^4)$ the subspace of elements $R\in\Sym^2(\wedge^2 \R^4)$ with $R_{\wedge^4}=0$. Such $R$ are called \emph{algebraic curvature operators}, while general elements of $\Sym^2(\wedge^2 \R^4)$ are often called \emph{modified curvature operators}.

The curvature operator of a Riemannian $4$-manifold $(M^4,\g)$, at a point $p\in M$, is an element of $\Sym^2_b(\wedge^2\R^4)$ once $T_pM$ is isometrically identified with $\R^4$; and, conversely, every element of $\Sym^2_b(\wedge^2\R^4)$ can be realized as the curvature operator of a Riemannian $4$-manifold at a point. In terms of the Kulkarni--Nomizu product,
\begin{equation*}
R_{\mathcal U}= \tfrac{\scal}{24}\g\owedge \g, \quad\text{and} \quad R_{\mathcal L}=\tfrac12\g\owedge\!\left(\Ric-\tfrac{1}{4}\scal\right),
\end{equation*}
i.e., geometrically, $R_\mathcal U$ encodes the scalar curvature, $R_\mathcal L$ the traceless Ricci tensor, and $R_\mathcal W$ the Weyl tensor, so an algebraic curvature operators $R$ is called \emph{scalar flat} if $R_{\mathcal U}=0$, \emph{Einstein} if $R_\mathcal L=0$, and \emph{locally conformally flat} if $R_{\mathcal W}=0$.

\subsection{Further decompositions}
The Hodge star operator $*\in\Sym^2(\wedge^2\R^4)$ has eigenvalues $\pm1$, and the corresponding eigenspaces $\wedge^2_+\R^4$ and $\wedge^2_-\R^4$ are called the self-dual and anti-self-dual subspaces, respectively. The orthogonal direct sum
\begin{equation}\label{eq:selfdualdec}
    \wedge^2\R^4=\wedge^2_+\R^4\oplus\wedge^2_-\R^4
\end{equation}
is preserved by $\SO(4)$, whose action on \eqref{eq:selfdualdec} factors through the standard product action of its $\Z_2$-quotient $\SO(3)\times \SO(3)$ on $\R^3\oplus\R^3$.
Restricting the $\O(4)$-representation \eqref{eq:decomp} to $\SO(4)\subset\O(4)$, the subspace $\mathcal W$ further decomposes into two $\SO(4)$-irreducibles $\mathcal W=\mathcal W_+\oplus\mathcal W_-$, which are the $5$-dimensional subspaces $\mathcal W_\pm\cong\Sym^2_0(\wedge^2_\pm \R^4)$.
In particular, if $(M^4,\g)$ is an \emph{oriented} Riemannian $4$-manifold, then the summand $R_{\mathcal W}$ of its curvature operator splits accordingly as $R_{\mathcal W_+}+R_{\mathcal W_-}$.

Restricting the $\O(4)$-representation \eqref{eq:decomp} once more, to $\U(2)\subset\SO(4)$, the subspace $\wedge^2_+\R^4$ decomposes 
into two $\U(2)$-irreducibles, of dimensions $1$ and $2$, and thus $\mathcal W_+$ decomposes into three $\U(2)$-irreducibles, of dimensions $1$, $2$, and $2$, while $\mathcal L$ decomposes into two $\U(2)$-irreducibles, of dimensions $3$ and $6$, see \cite{armstrong,tricerri-vanhecke}. 

\subsection{Canonical form}
Given $R\in\Sym_b^2(\wedge^2\R^4)$, the above $\SO(4)$-action on \eqref{eq:selfdualdec} can be used to diagonalize $R$ on each subspace $\wedge^2_\pm\R^4$, obtaining orthonormal bases of $\wedge^2_+\R^4$ and $\wedge^2_-\R^4$ so that the matrix representing $R$ has the block structure
\begin{equation}\label{eq:curvop4}
R=\begin{pmatrix}
u\id+W_+ & C^\mathrm t \\
C & u\id+W_-
\end{pmatrix},
\end{equation}
where $u\in\R$ is a scalar, $W_\pm$ are traceless diagonal $3\times 3$ matrices, with eigenvalues $w_1^\pm\leq w_2^\pm\leq w_3^\pm$, % satisfying $\sum_{i=1}^3 w_i^\pm=0$,
and $C$ is a $3\times 3$ matrix.
To simplify notation, we often use vectors
\begin{equation}\label{eq:widef}
\vec w_\pm:=\big(w_1^\pm,w_2^\pm,w_3^\pm\big),
\end{equation}
and write $W_\pm=\diag(w_1^\pm,w_2^\pm,w_3^\pm)$.
For convenience, we assume henceforth that every algebraic curvature operator $R$ is in the above canonical form \eqref{eq:curvop4}.

The components in \eqref{eq:curvop4} correspond precisely to $R_\mathcal{U}$, $R_\mathcal{L}$, and $R_{\mathcal{W}_\pm}$. Note that 
\begin{equation*}
\textstyle u=\frac{1}{12}\scal = \frac{1}{6} \tr R
\end{equation*}
and $R$ is Einstein if and only if $C=0$, which, in turn, is equivalent to $R$ and~$*$ commuting. Moreover, $R$ is called \emph{half conformally flat} if one of the self-dual or anti-self-dual Weyl tensors $W_\pm$ vanishes.
The involution of \eqref{eq:selfdualdec} given by reversing orientation interchanges $\wedge^2_\pm \R^4$, and hence interchanges $W_\pm$, transposes $C$, but leaves $u$ invariant. 
This is the effect on the curvature operator of an oriented Riemannian $4$-manifold if its orientation is reversed. 

The algebraic curvature operator $R$ is \emph{K\"ahler} if $R\in\Sym^2_b(\mathfrak u(2))$, where $\mathfrak u(2)\subset\mathfrak o(4)\cong\wedge^2\R^4$ is the Lie algebra of $\U(2)\subset\O(4)$, or, equivalently, if $JR=RJ=R$, where $J\in\Sym^2(\wedge^2\R^4)\cong\Sym^2(\wedge^2\C^2)$ is the map $J(v\wedge w)= \sqrt{-1}\,v\wedge\sqrt{-1}\, w$. In terms of the canonical form \eqref{eq:curvop4}, this means that $\vec w_+=(-u,-u,2u)$, or $\vec w_+=(2u,-u,-u)$, depending on the sign of $u$, and $C^\mathrm t$ has at most one nonzero row, so that $\ker R$ contains $\mathfrak u(2)^\perp\subset\wedge^2_+\R^4$.

\begin{remark}
There are other \emph{canonical} ways to represent the matrix of a curvature operator $R\in\Sym^2_b(\wedge^2\R^4)$ aside from the above \eqref{eq:curvop4}, see e.g.~Ville~\cite{ville-negative,ville-positive} and~Appendix~\ref{app:ville}. Furthermore, for Einstein curvature operators, see \cite{berger-canonicalform,singer-thorpe}.
\end{remark}

\subsection{Sectional curvature and pinching}
Consider the oriented Grassmannian
\begin{equation}\label{eq:grassmannian}
\Gr^+(\R^4)=\big\{ \alpha\in\wedge^2\R^4 : |\alpha|^2=1, \, \alpha\wedge\alpha=0\big\}\subset \wedge^2 \R^4,
\end{equation}
which, using \eqref{eq:selfdualdec}, can be realized as $S^2\times S^2\subset\R^3\oplus\R^3$, cf.~\eqref{eq:grassmannian-realization}.
Given a modified curvature operator $R\in\Sym^2(\wedge^2\R^4)$, the \emph{sectional curvature} of $\alpha\in\Gr^+(\R^4)$ is % defined as
\begin{equation*}
\sec_R(\alpha):=\langle R(\alpha),\alpha\rangle.
\end{equation*}
Clearly, $\sec_R$ depends linearly on $R$. Moreover, given $k\in\R$, we write $\sec_R\geq k$ if $\sec_R(\alpha)\geq k$ for all $\alpha\in\Gr^+(\R^4)$, and analogously for $\sec_R\leq k$.

\begin{definition}\label{def:dpinched}
Given $0<\delta\leq1$, an algebraic curvature operator $R\in\Sym^2_b(\wedge^2\R^4)$ is called \emph{positively $\delta$-pinched} if $\delta\leq \sec_R\leq 1$, \emph{negatively $\delta$-pinched} if $-R$ is positively $\delta$-pinched, and \emph{$\delta$-pinched} if either $R$ or $-R$ is positively $\delta$-pinched.
\end{definition}

The following characterization of sectional curvature bounds for algebraic curvature operators is a consequence of Finsler's Lemma in Optimization Theory, that became known as Thorpe's trick in the Geometric Analysis community, see~\cite{cagco}.

\begin{lemma}[Finsler--Thorpe Trick]\label{lem:finslerthorpe}
Let $R\in \Sym_b^2(\wedge^2\R^4)$ be an algebraic curvature operator. Then $\sec_R\geq0$ if and only if there exists $t\in\R$ such that $R+t\,*\succeq0$.
\end{lemma}

In other words, the set of algebraic curvature operators with $\sec\ge0$ is the image of the set of positive-semidefinite operators in $\Sym^2(\wedge^2\R^4)$ under the orthogonal projection onto $\Sym_b^2(\wedge^2\R^4)$. Since $\sec_R$ depends linearly on $R$, and $\sec_{\id}=1$, the following is an immediate consequence of Lemma~\ref{lem:finslerthorpe}.

\begin{corollary}
An algebraic curvature operator $R\in\Sym^2_b(\wedge^2\R^4)$ is positively $\delta$-pinched if and only if there exist $t_1,t_2\in\R$ such that
\begin{equation}\label{eq:t1t2bound}
R-\delta \id + t_1\, * \succeq 0, \qquad \text{and}\qquad \id-R+t_2\,*\succeq 0.
\end{equation}
\end{corollary}

\begin{example}
Useful models of algebraic curvature operators are given by the curvature operators of symmetric spaces, which are trivially constant. Recall that each compact symmetric space has a noncompact dual, and their curvature operators are the opposite of one another. For instance, the round $4$-sphere $S^4$, and its noncompact dual, the hyperbolic space $H^4$, have curvature operators
\begin{equation*}
R_{S^4}=\id, \quad \text{and}\quad R_{H^4}=-\id,
\end{equation*}
and hence have constant sectional curvature $1$ and $-1$, respectively. Both are obviously Einstein and locally conformally flat. Moreover, $R_{S^4}$ is clearly the unique algebraic curvature operator that is positively $\delta$-pinched for all $0<\delta\leq1$.

The curvature operators of the complex projective plane $\C P^2$, and its noncompact dual, the complex hyperbolic plane $\C H^2$, written in the form \eqref{eq:curvop4}, are:
\begin{equation}\label{eq:RCPH2}
R_{\C P^2}=\begin{pmatrix}
\diag(0,0,6) &  \\
 & 2\id
\end{pmatrix} \quad \text{and}\quad
R_{\C H^2}=\begin{pmatrix}
\diag(-6,0,0) &  \\
 & -2\id
\end{pmatrix},
\end{equation}
and hence have sectional curvatures $1\leq\sec_{R_{\C P^2}}\leq 4$ and $-4\leq\sec_{R_{\C H^2}}\leq-1$, respectively. 
In particular, $R=\tfrac14R_{\C P^2}$ is positively $\tfrac14$-pinched, with \eqref{eq:t1t2bound} satisfied setting $t_1=\tfrac14$ and $t_2=\tfrac12$. Both curvature operators \eqref{eq:RCPH2} are half conformally flat, K\"ahler, and Einstein. Reversing orientations, one obtains the curvature operators $R_{\overline{\C P^2}}$ and $R_{\overline{\C H^2}}$, having the same diagonal blocks as \eqref{eq:RCPH2} but in the reverse order.
\end{example}

\subsection{\texorpdfstring{Topology of closed $4$-manifolds}{Topology of closed 4-manifolds}}
Let $M$ be a closed oriented (smooth) $4$-man\-i\-fold, and denote by 
$b_k(M)=\operatorname{rank} H_k(M,\Z)$ its Betti numbers.
The \emph{intersection form} of $M$ is the unimodular symmetric bilinear form defined on the torsion-free part of the second cohomology $H^2(M,\Z)$ as
\begin{align*}
&Q_M\colon H^2(M,\Z)/\mathrm{torsion}\times H^2(M,\Z)/\mathrm{torsion}\longrightarrow\Z\\ 
&Q_M(\alpha,\beta)=(\alpha\smile\beta)([M]),
\end{align*}
where $[M]\in H_4(M,\Z)$ denotes the fundamental class of $M$. Alternatively, using de Rham cohomology and representing $\alpha,\beta\in\Omega^2(M)$ as $2$-forms on $M$, one has $Q_M(\alpha,\beta)=\int_M \alpha\wedge\beta$.
The number of positive and negative eigenvalues of $Q_M$, counted with multiplicities, are denoted $b_+(M)$ and $b_-(M)$, respectively. By Hodge Theory, $b_\pm(M)=\dim\{\alpha\in\Omega^2(M) : \Delta \alpha=0, \, *\,\alpha=\pm\alpha\}$.

The manifold $M$ is called \emph{definite} if $Q_M$ is definite, i.e., if either $b_+(M)=0$ or $b_-(M)=0$, and called \emph{indefinite} otherwise. Denoting by $\overline M$ the manifold $M$ with its reverse orientation, $Q_{\overline M}=-Q_M$, so $b_\pm(\overline M)=b_\mp(M)$. Moreover, $Q_{M_1\# M_2}=Q_{M_1}\oplus Q_{M_2}$.
For instance, $\C P^2$ is definite, since $b_+(\C P^2)=1$ and $b_-(\C P^2)=0$, and $b_+(\#^r\C P^2\#^s \overline{\C P^2})=r$ and $b_-(\#^r\C P^2\#^s \overline{\C P^2})=s$, while $S^2\times S^2$ is indefinite, as $b_+(S^2\times S^2)=b_-(S^2\times S^2)=1$, and $b_+(\#^r (S^2\times S^2))=b_-(\#^r (S^2\times S^2))=r$. Recall that $\overline{S^2\times S^2}=S^2\times S^2$, since the antipodal map on $S^2$ is orientation-reversing.

Clearly, $b_2(M)=b_+(M)+b_-(M)$.
Since $M$ is oriented, by Poincar\'e duality, $b_0(M)=b_4(M)=1$ and $b_1(M)=b_3(M)$. The \emph{Euler characteristic} and \emph{signature} of $M$ are given by
\begin{align*}
\chi(M)&=2-2b_1(M)+b_+(M)+b_-(M)\\
\sigma(M)&=b_+(M)-b_-(M).
\end{align*}
In particular, it follows that, for all closed oriented $4$-manifolds,
\begin{equation}\label{eq:basic_geography}
\chi(M)\equiv \sigma(M)\mod 2, \quad \text{ and }\quad \chi(M)\geq |\sigma(M)|-2b_1(M)+2.
\end{equation}

If $M$ is simply-connected, then $H^2(M,\Z)$ is free and $b_1(M)=b_3(M)=0$. In this case, the celebrated works of Donaldson~\cite{Donaldson83} and Freedman~\cite{Freedman82}, combined with the $\widehat A$-genus obstruction to $\scal>0$ for spin manifolds yields the following: % classification of homeomorphism types:

\begin{theorem}[Donaldson, Freedman, Lichnerowicz]\label{thm:4Dmanifolds}
Let $M$ be a smooth, closed, oriented, simply-connected $4$-manifold that admits a metric with $\scal>0$.
\begin{enumerate}[\rm (i)]
\item If $M$ is non-spin, then $M$ is homeomorphic to $\#^r\C P^2\#^s \overline{\C P^2}$,
\item If $M$ is spin, then $\sigma(M)=0$ and $M$ is homeomorphic to $\#^r (S^2\times S^2)$,
\end{enumerate}
where $r=b_+(M)$ and $s=b_-(M)$; and if $r=0$ or $s=0$, then the corresponding (trivial) summand is $S^4$. 
\end{theorem}

\begin{proof}
By Donaldson~\cite{Donaldson83} and Freedman~\cite{Freedman82}, every smooth, closed, oriented, simply-connected $4$-manifold $M$ is homeomorphic to either:
\begin{itemize}
\item a connected sum of $\C P^2$'s and $\overline{\C P^2}$'s, if $M$ is non-spin,
\item a connected sum of $S^2\times S^2$'s and $M_{E_8}$'s, if $M$ is spin.
\end{itemize}
In the above, $M_{E_8}$ is a certain non-smooth $4$-manifold with $\sigma(M_{E_8})=8$, see \cite[Chap.~1]{DonaldsonKronheimer} for details. However, if $M$ is spin, then the existence of a metric with $\scal>0$ implies that $\widehat A(M)=0$, and hence $\sigma(M)=0$, see e.g.~\cite[\S 6.72]{Besse}.
Therefore, no copies of $M_{E_8}$ may appear in this situation, concluding the proof.
\end{proof}

\begin{remark}
The \emph{converse} to Theorem~\ref{thm:4Dmanifolds} also holds, in the sense that for all $r,s\in\mathds N\cup\{0\}$, the $4$-manifolds $\#^r\C P^2\#^s \overline{\C P^2}$ and $\#^r (S^2\times S^2)$ admit (smooth structures that support) metrics with $\scal>0$; in fact $\Ric>0$, see \cite{sha-yang,perelman}.
\end{remark}

\subsection{Integral formulas}\label{subsec:integrals}
According to the Chern--Gauss--Bonnet formula and the Hirzebruch signature formula, the Euler characteristic and signature of a closed oriented Riemannian $4$-manifold $(M^4,\g)$ can be respectively expressed as
\begin{equation}\label{eq:integrals}
\chi(M)=\frac{1}{\pi^2}\int_M \underline{\chi}(R) \,\vol_\g, \quad \text{ and } \quad \sigma(M)=\frac{1}{\pi^2}\int_M \underline{\sigma}(R)\,\vol_\g,
\end{equation}
where $\underline{\chi}(R)$ and $\underline{\sigma}(R)$ are quadratic forms on its curvature operator $R$, given by
\begin{align}
\underline{\chi}(R) &=\tfrac{1}{8}\left( 6u^2 + |W_+|^2 +|W_-|^2-2|C|^2\right), \label{eq:chi}\\
\underline{\sigma}(R) &=\tfrac{1}{12}\left( |W_+|^2 -|W_-|^2\right), \label{eq:sigma}
\end{align}
if $R$ is written in the canonical form~\eqref{eq:curvop4}, see e.g.~\cite[p.~161]{Besse}. We denote by $|A|$ the Hilbert--Schmidt norm of $A\in\mathrm{Mat}_{n\times n}(\R)$, i.e., $|A|^2=\sum_{i,j} a_{ij}^2$. In particular,
\begin{equation*}
|W_\pm|^2=|\vec w_\pm|^2=(w_1^\pm)^2+(w_2^\pm)^2+(w_3^\pm)^2,
\end{equation*}
is the usual Euclidean norm of the vectors $\vec w_\pm$ as in \eqref{eq:widef}. Since all terms in \eqref{eq:chi} and \eqref{eq:sigma} are of degree $2$, it follows that $\underline{\sigma}(-R)=\underline{\sigma}(R)$ and $\underline{\chi}(-R)=\underline{\chi}(R)$. Moreover, reversing the orientation changes the sign of $\underline{\sigma}$ but leaves $\underline{\chi}$ invariant.

\begin{remark}\label{rem:hitchin-thorpe}
If $R$ is Einstein, i.e., $C=0$, then $2\underline{\chi}(R)-3\underline{\sigma}(R)=\frac{3}{2}u^2+\frac{1}{2}|W_-|^2\geq0$ by \eqref{eq:chi} and \eqref{eq:sigma}. Therefore, any closed oriented Einstein $4$-manifold $(M^4,\g)$ satisfies the \emph{Hitchin--Thorpe inequality} $|\sigma(M)|\leq \frac{2}{3}\chi(M)$, see e.g.~\cite[Thm.~6.35]{Besse}.
\end{remark}

If $(M^4,\g)$ is a noncompact negatively $\delta$-pinched $4$-manifold with \emph{finite volume}, then the integrals \eqref{eq:integrals} converge. Such manifolds have \emph{finite topological type}; i.e., are diffeomorphic to the interior of a compact manifold with boundary, see e.g.~\cite[Thm.~10.5]{bgs}. In particular, $\chi(M)$ is well-defined. Moreover, $(M^4,\g)$ has \emph{bounded geometry} in the sense of Cheeger and Gromov \cite{cheeger-gromov1,cheeger-gromov2}, since its universal cover has infinite injectivity radius. Thus, by \cite[Thm.~3.1 (3), Thm. 6.1]{cheeger-gromov1}, both equalities in \eqref{eq:integrals} hold, where $\sigma(M)$ is to be understood as the $L^2$-signature $\sigma_{(2)}(M^4,\g)$, which is a proper homotopy invariant.

We shall also need the following consequence of \cite[Thm~A]{CGY03}.

\begin{lemma}[Chang--Gursky--Yang]\label{lemma:CGY}
If $(M^4,\g)$ is a closed oriented $4$-manifold with $\sec>0$, then either $M$ is diffeomorphic to $S^4$, or
\begin{equation*}
\chi(M)\leq \frac{1}{4\pi^2} \int_M  |W_+|^2 + |W_-|^2\,\vol_\g.
\end{equation*}
\end{lemma}

\begin{proof}
To apply \cite[Thm~A]{CGY03}, we need to show that the Yamabe invariant $Y(M,\g)$ is positive. By the solution of the Yamabe problem (see \cite[Thms~5.11, 5.30]{Aubin}), there exists a positive function $\varphi\in C^\infty(M)$ such that the metric $\g'=\varphi^2 \g$ has constant scalar curvature $c$, and such that the infimum in the definition of $Y(M,\g)$ is achieved at $\g'$. Thus $Y(M,\g)=c\Vol(M,\g')^{1/2}$. 
Let $p\in M$ be the global maximum of the conformal factor $\varphi$, so that $\Delta \varphi (p) \geq 0$. %The formula for the scalar curvature of $\g'$ in terms of $\scal_\g$ and $\varphi$ is given by (see \cite[Equation (1) on page 146]{Aubin}):
Since $6\Delta\varphi +\scal_\g \varphi=c\,\varphi^3$, we have that $c>0$, so $Y(M,\g)>0$.
The conclusion now follows from~\cite[Thm A]{CGY03}, keeping in mind the different norm conventions, see \cite[Rmk~2]{CGY03}.
\end{proof}

\subsection{Sphere theorem and volume bound}
Finally, for the reader's convenience, we now explain how to use standard techniques in Comparison Geometry to prove:

\begin{lemma}\label{lemma:vol}
Let $(M^4,\g)$ be a $4$-manifold with $\sec\geq\delta>0$. If $M$ is not homeomorphic to $S^4$, then $\Vol(M,\g)\leq \frac{4\pi^2}{3\delta^2}$.
\end{lemma}
\begin{proof}
Since $M$ is not homeomorphic to $S^4$, the Grove--Shiohama Diameter Sphere Theorem \cite{GS77} implies that $\operatorname{diam}(M)\leq \frac12\operatorname{diam}\!\big(S^4\big(\frac{1}{\sqrt{\delta}}\big)\big)=\frac{\pi}{2\sqrt{\delta}}$. In particular, choosing any $p\in M$, we have $M=B\big(p,\frac{\pi}{2\sqrt{\delta}}\big)$. Since $\sec\geq\delta$, the Ricci curvature of $(M^4,\g)$ is at least that of the sphere $S^4\big(\frac{1}{\sqrt{\delta}}\big)$, and hence, by the Bishop Volume Comparison Theorem (see e.g.~\cite[Chap.~0.H]{Besse}), the volume of the ball $B\big(p, \frac{\pi}{2\sqrt{\delta}}\big)$ is at most the volume of a hemisphere in $S^4\big(\frac{1}{\sqrt{\delta}}\big)$. In conclusion:
\begin{equation*}
\Vol(M,\g)=\Vol\!\big(B\big(p,\tfrac{\pi}{2\sqrt{\delta}}\big)\big) \leq \tfrac{1}{2}\Vol\!\big(S^4\big(\tfrac{1}{\sqrt{\delta}}\big) \big)=\tfrac{1}{2\delta^2}\Vol\!\big(S^4(1)\big) %= \tfrac{1}{2\delta^2}\tfrac{8\pi^2}{3}
=\tfrac{4\pi^2}{3\delta^2}.\qedhere
\end{equation*}
\end{proof}

\section{Preliminaries on Optimization}\label{sec:optimization}

This short section discusses an elementary (yet quite useful) approach to optimize quadratic forms on compact convex sets, particularly polytopes and simplices.

Let $Q\colon \R^n\to\R$ be a polynomial function of degree $2$. We say $Q$ is \emph{positive-} or \emph{negative-(semi)definite} if its Hessian matrix is positive- or negative-(semi)definite, and \emph{indefinite} otherwise. It is noteworthy that even though the quadratic program
\begin{align*}
&\text{maximize } Q(x), \\
&\text{subject to } A\,x+b\preceq 0,
\end{align*}
where $A\in\textnormal{Mat}_{n\times n}(\R)$ and $b\in\R^n$, can be solved in polynomial time if $Q$ is negative-semidefinite \cite{kozlov}, it becomes NP-hard if the Hessian of $Q$ is allowed to have positive eigenvalues~\cite{sahni,qp-np-hard}.

A \emph{face} of a closed convex set $K\subset\R^n$ is a convex subset $F\subset K$ such that, whenever $\frac{x+y}{2}\in F$ for some $x,y\in K$, then both $x,y\in F$.
Given a closed convex set $K\subset\R^n$, let $\mathcal F(K)=\{F\subset K : F \text{ is a face of } K\}$. Then
\begin{equation*}
K=\bigsqcup_{F\in\mathcal F(K)}\textnormal{relint}(F),
\end{equation*}
where $\textnormal{relint}(F)$ is the \emph{relative interior} of $F$, i.e., the interior of $F$ inside its \emph{affine hull} $\aff(F)$, which is the smallest affine subspace of $\R^n$ containing $F$. The dimension of a face $F$ is defined as the dimension of its supporting affine space~$\aff(F)$.

\begin{proposition}\label{lem:negdefface}
 Let $Q\colon \R^n\to\R$ be a polynomial function of degree $2$, and let $K\subset\R^n$ be a compact convex set. There is a face $F$ of $K$ such that $Q'=Q|_{\aff(F)}$ is negative-definite and a point $x_0\in\textnormal{relint}(F)$ such that $Q(x_0)=\max\limits_{x\in K}Q(x)$. Furthermore, $x_0$ is the unique point in $\aff(F)$ such that $\nabla Q'(x_0)=0$.
\end{proposition}

\begin{proof}
 Let $x_0\in K$ be such that $Q(x_0)=\max\limits_{x\in K}Q(x)$ and $F\in\mathcal F(K)$ the face whose relative interior contains $x_0$. We clearly have $Q'(x_0)=\max\limits_{x\in F}Q'(x)$. If $Q'$ is not negative-definite, then there is an affine line $L\subset\textnormal{aff}(F)$ through $x_0$, such that the restriction of $Q'$ to $L$ is a convex function. So there exists $x_1$ in the relative boundary of $F$, and hence in a face of smaller dimension, with $Q(x_1)\geq Q(x_0)$. This implies the first claim. Furthermore, since $Q'$ is negative-definite, it has a unique critical point, which is the global maximum of $Q'$. Since $x_0$ is in the relative interior of $F$, it must be the global maximum of $Q'$.
\end{proof}

Note that if $F=\{x_0\}$ is a singleton, then $\aff(F)=\textnormal{relint}(F)=F$, the restriction of $Q|_{\aff(F)}$ is simultaneously negative- and positive-definite, and $\nabla Q|_{\aff(F)}(x_0)=0$.

\begin{corollary}\label{cor:smallsig}
Let $Q\colon \R^n\to\R$ be a polynomial function of degree $2$ whose Hessian matrix has $d$ eigenvalues in $(-\infty,0)$, and $K\subset\R^n$ be a compact convex set. Then
 $$\max_{x\in K}Q(x)=\max_{\substack{x\in \textnormal{relint}(F),\\ F\in \mathcal F(K), \, \dim F\leq d}}Q(x).$$
\end{corollary}

\begin{proof}
 By Proposition~\ref{lem:negdefface} there exists a face $F$ of $K$ such that $Q|_{\textnormal{aff}(F)}$ is negative-definite and $x_0\in\textnormal{relint}(F)$ with $Q(x_0)=\max\limits_{x\in K}Q(x)$. Negative-definiteness of $Q|_{\textnormal{aff}(F)}$ implies that $\dim(F)\leq d$.
\end{proof}

\subsection{Optimizing quadratic polynomials on simplices}\label{subsec:optsimplex}
Suppose $\Delta^k\subset\R^n$ is a $k$-simplex with vertices $V=\{v_1,\dots, v_{k+1}\}$, i.e., $V$ is affinely independent and $\Delta^k=\conv(V)$, and $Q\colon \R^n\to\R$ is a polynomial function of degree~$2$. Recall that 
\begin{equation*}
\conv(S):=\left\{\textstyle\sum\limits_{j=1}^N x_j s_j \, :\, s_j \in S, \, x_j\geq 0,\, \sum\limits_{j=1}^N x_j=1,\, N\in\mathds N \right\}
\end{equation*}
denotes the \emph{convex hull} of a set $S\subset \R^n$.
Note that Proposition~\ref{lem:negdefface} gives a method to compute $\max\limits_{x\in \Delta^k}Q(x)$ that only involves Linear Algebra, and Corollary~\ref{cor:smallsig} renders this computation significantly easier if $d$ is small.

By Proposition~\ref{lem:negdefface}, the maximum of $Q$ on $\Delta^k$ is attained in the relative interior of a face on whose affine hull $Q$ is negative-definite. Since $\Delta^k$ is a simplex, any subset $S\subset V$ is affinely independent, its convex hull $\conv(S)$ is a face of $\Delta^k$, and every face of $\Delta^k$ is of this form. Thus, in order to find $\max\limits_{x\in \Delta^k}Q(x)$, first compute 
\begin{equation}\label{eq:verticesnegdef}
\mathcal{S}=\big\{S\subset V:\, Q|_{\aff(S)} \textrm{ is negative-definite}\big\}.
\end{equation}
Then, for each $S\in\mathcal{S}$, compute the unique point $x_S\in\aff(S)$ with $\nabla Q|_{\aff(S)}(x_S)=0$ and consider the set 
\begin{equation}\label{eq:verticesnegdefint}
\mathcal{S}'=\big\{S\in\mathcal{S}:\, x_S\in\textnormal{relint}(\conv(S))\big\}.
\end{equation}
Finally, Proposition~\ref{lem:negdefface} implies that
\begin{equation}\label{eq:max-max}
\max_{x\in \Delta^k}Q(x)=\max_{S\in\mathcal{S}'}Q(x_S).
\end{equation}

\section{Pinched curvature operators}\label{sec:simplex}

In this section, we establish foundational results on convex sets of $\delta$-pinched curvature operators that are
needed throughout the rest of the paper.

\subsection{Projection onto Einstein curvature operators}
Using the same notation in the decomposition \eqref{eq:decomp} of $\Sym^2(\wedge^2\R^4)$, define the sets
\begin{equation}\label{eq:dpinched}
\begin{aligned}
    \dpinched &:=\left\{R\in\Sym^2_b(\wedge^2\R^4) : R \text{ is positively } \delta\text{-pinched} \right\},\\
    \epinched &:=\dpinched\cap (\mathcal U\oplus\mathcal W),
\end{aligned}
\end{equation}
and consider the orthogonal projection onto Einstein (modified) curvature operators
\begin{equation}\label{eq:pr}
\begin{aligned}
\pr\colon \Sym^2(\wedge^4\R^4)&\longrightarrow\mathcal U \oplus\mathcal W\oplus\wedge^4 \R^4 \\
\pr(R)&=R-R_{\mathcal L}.
\end{aligned}
\end{equation}

\begin{lemma}\label{lem:projecteinsteinpsd}
If $R\in\Sym^2(\wedge^2\R^4)$ satisfies $R\succeq0$, then also $\pr(R)\succeq0$.
\end{lemma}

\begin{proof}
Consider the decomposition $\wedge^2 \R^4=\wedge^2_+\R^4\oplus\wedge^2_-\R^4$ as in
\eqref{eq:selfdualdec}, and denote by $\pi_\pm\colon\wedge^2\R^4\to\wedge^2_\pm\R^4$ the corresponding orthogonal projections. Since $R\succeq0$, we have $\pr(R)=\pi_+\circ R\circ \pi_+ + \pi_-\circ R\circ \pi_-\succeq0.$
\end{proof}

\begin{lemma}\label{lem:projecteinstein}
Let $R\in\Sym_b^2(\wedge^2\R^4)$ be an algebraic curvature operator.
\begin{enumerate}[\rm (a)]
  \item If $\sec_R\geq0$, then $\sec_{\pr(R)}\geq0$.
  \item If $R\in\dpinched$, then $\pr(R)\in\epinched$.
 \end{enumerate}
\end{lemma}

\begin{proof}
By the Finsler--Thorpe Trick (Lemma~\ref{lem:finslerthorpe}), if $\sec_R\geq0$, then there exists $S\in\Sym^2(\wedge^2\R^4)$, $S\succeq0$, whose orthogonal projection on $\Sym_b^2(\wedge^2\R^4)$ is $R$.
Clearly, $\pr(R)$ is the orthogonal projection on $\Sym_b^2(\wedge^2\R^4)$  of $\pr(S)$. Thus, in order to prove (a), it suffices to show that $\pr(S)\succeq0$, which holds by Lemma~\ref{lem:projecteinsteinpsd}. 
As $\sec_R$ depends linearly on $R$, and $\sec_{\id}=1$, applying (a) to $R-\delta\id$ and $\id-R$ yields (b).
\end{proof}

\begin{proposition}\label{prop:projd}
$\pr(\dpinched)=\epinched$.
\end{proposition}

\begin{proof}
Lemma~\ref{lem:projecteinstein}~(b) gives one inclusion, the other is clear from Definition~\ref{def:dpinched}.
\end{proof}

\begin{remark}
The properties of the projection \eqref{eq:pr} onto Einstein curvature operators in Lemma~\ref{lem:projecteinstein} and Proposition~\ref{prop:projd} \emph{do not} hold for the projection onto locally conformally flat curvature operators. Namely, there exists $R\in\Sym_b^2(\wedge^2\R^4)$ with $\sec_R\geq0$ whose locally conformally flat part $R_{\mathcal U}+R_{\mathcal L}$ does not have $\sec\geq0$.
\end{remark}

\subsection{Einstein simplices}
Fix $0<\delta<1$, and consider the curvature operators
\begin{equation}\label{eq:Rww}
R(\vec w_+,\vec w_-,u):=\diag(u+w_i^+,u+w_i^-)\in \mathcal U\oplus\mathcal W
\end{equation}
as in \eqref{eq:curvop4}, with $C=0$, and recall that $\vec w_\pm=\big(w_1^\pm,w_2^\pm,w_3^\pm\big)$ satisfy
\begin{equation}\label{eq:ww}
 w_1^{\pm}+w_2^{\pm}+w_3^{\pm}=0, \qquad  w_1^{\pm}\leq w_2^{\pm}\leq w_3^{\pm}.
\end{equation}
By the Finsler-Thorpe Trick (Lemma~\ref{lem:finslerthorpe}), such a curvature operator $R(\vec w_+,\vec w_-,u)$ is positively $\delta$-pinched if and only if there exist $t_1,t_2\in\R$ such that \eqref{eq:t1t2bound} holds, i.e.,
\begin{equation}\label{eq:defesimp}
 \begin{aligned}
 \delta\leq w_i^{+}+u+t_1, &\qquad  w_i^{+}+u+t_2\leq1, \\
 \delta\leq w_i^{-}+u-t_1, &\qquad  w_i^{-}+u-t_2\leq1,
 \end{aligned} \qquad i=1,2,3.
\end{equation}

We first get rid of some redundant inequalities:

\begin{lemma}\label{lem:einsteinineqs}
$R(\vec w_+,\vec w_-,u)\in\epinched$ if and only if there exist $t_1,t_2\in\R$ such that
\begin{equation}\label{eq:esimpreduced}
\begin{aligned}
\delta\leq w_1^{+}+u+t_1,  &\qquad\quad\;  w_3^{+}+u+t_2\leq1, \\ 
  \delta\leq w_1^{-}+u-t_1, &\qquad\quad\;  w_3^{-}+u-t_2\leq1.
\end{aligned}
\end{equation}
\end{lemma}

\begin{proof}
The inequalities in \eqref{eq:esimpreduced} are a subset of those in \eqref{eq:defesimp}, so 
 \eqref{eq:esimpreduced} obviously holds if $R(\vec w_+,\vec w_-,u)\in\epinched$. For the converse, observe that $\delta\leq w_1^{+}+u+t_1$ together with $w_1^+\leq w_i^+$ imply that $\delta\leq w_i^{+}+u+t_1$ for $i=1,2,3$. The other inequalities in \eqref{eq:defesimp} missing from \eqref{eq:esimpreduced} can be obtained in the same way, using \eqref{eq:ww}.
\end{proof}

The set of points $(\vec w_+,\vec w_-,u,t_1,t_2)\in\R^9$ that satisfy \eqref{eq:ww} and \eqref{eq:esimpreduced} is clearly an intersection of linear subspaces and affine half-spaces in $\R^9$. In order to eliminate the variables $t_1$ and $t_2$, we show it is actually a simplex and compute its vertices.

\begin{proposition}\label{prop:verticesboththorpe}
Let $\esimp^7\subset \R^7$ be the $7$-simplex defined as convex hull of the rows $v_1,\ldots, v_8$ of the matrix:
 $$\bordermatrix{
 &  w_1^{+} &  w_2^{+} & w_1^{-} & w_2^{-} & u & t_1 & t_2 \cr
v_1&\frac{2}{3}(\delta-1) &\frac{2}{3}(\delta-1) &0&0&\frac{1}{3}(2\delta+1)&\frac{1}{3}(1-\delta)&\frac{2}{3}(\delta-1)\rule[-6pt]{0pt}{10pt} \cr
v_2&\frac{4}{3}(\delta-1) &\frac{2}{3}(1-\delta) &0&0&\frac{1}{3}(\delta+2) &\frac{2}{3}(1-\delta)&\frac{1}{3}(\delta-1)  \rule[-6pt]{0pt}{10pt} \cr
v_3&0&0&\frac{2}{3}(\delta-1)&\frac{2}{3}(\delta-1)&\frac{1}{3}(2\delta+1) &\frac{1}{3}(\delta-1) &\frac{2}{3}(1-\delta) \rule[-6pt]{0pt}{10pt}
\cr
v_4&0&0&\frac{4}{3}(\delta-1) &\frac{2}{3}(1-\delta)&\frac{1}{3}(\delta+2) &\frac{2}{3}(\delta-1) &\frac{1}{3}(1-\delta) \rule[-6pt]{0pt}{10pt} \cr
v_5&0&0&0&0&1&\delta-1&0 \rule[-6pt]{0pt}{10pt} \cr
v_6&0&0&0&0&1&1-\delta&0 \rule[-6pt]{0pt}{10pt} \cr
v_7&0&0&0&0&\delta&0&\delta-1\rule[-6pt]{0pt}{10pt}  \cr
v_8&0&0&0&0&\delta&0&1-\delta }$$
and consider its image under the linear map $\iota_7\colon\R^7\to\R^9$, given by
\begin{equation*}
\iota_7\big(w_1^+,w_2^+,w_1^-,w_2^-,u,t_1,t_2\big)=\big(w_1^+,w_2^+,-w_1^+ -w_2^+,w_1^-,w_2^-,-w_1^- -w_2^-,u,t_1,t_2\big).
\end{equation*}
Then $(\vec w_+,\vec w_-,u,t_1,t_2)\in\R^9$ is in $\iota_7(\esimp^7)$ if and only if it satisfies \eqref{eq:ww} and~\eqref{eq:esimpreduced}.
\end{proposition}

\begin{proof}
Let $P_\delta$ be the set of $(\vec w_+,\vec w_-,u,t_1,t_2)\in\R^9$ that satisfy \eqref{eq:ww} and \eqref{eq:esimpreduced}. 
Since these are points such that $R(\vec w_+,\vec w_-,u)\in\epinched$ by Lemma~\ref{lem:einsteinineqs}, the set $P_\delta$ is bounded. Indeed, each entry $R_{ijkl}=\langle R(e_i\wedge e_j),e_k\wedge e_l \rangle$ of the matrix representing a positively $\delta$-pinched curvature operator $R\in\dpinched$, where $\{e_i\}$ is an orthonormal basis, satisfies $|R_{ijkl}|\leq 1$ by Berger's classical estimates, see e.g.~\cite{karcher-estimate}.
Thus, if $(\vec w_+,\vec w_-,u,t_1,t_2)\in P_\delta$, then $|u|$ and $|\vec w_\pm|$ are bounded, and hence so are $|t_1|$ and $|t_2|$ by \eqref{eq:esimpreduced}. 
Therefore, $P_\delta\subset \R^9$ is a polytope. Moreover, $P_\delta\subset\iota_7(\R^7)$ by \eqref{eq:ww}, so $\dim P_\delta\leq 7$. On the other hand, $\iota_7\left(\frac{\delta-1}{3},0,\frac{\delta-1}{3},0,\frac{\delta+1}{2},0,0\right)$ is in the relative interior of $P_\delta$, and hence $\dim P_\delta=7$.
Since $P_\delta$ is defined by $8$ inequalities, namely, $4$ in \eqref{eq:ww} and $4$ in \eqref{eq:esimpreduced}, it is a $7$-simplex. Its vertices are the points where $7$ of the $8$ inequalities are equalities. These are exactly the points $\iota_7(v_j)$, $1\leq j\leq 8$, where $v_j$ are the rows of the matrix above. For example, $\iota_7(v_1)$ saturates all inequalities in \eqref{eq:ww} and \eqref{eq:esimpreduced} except for $w_2^+\leq w_3^+$. Thus, $P_\delta=\iota_7(\esimp^7)$, concluding the proof.
\end{proof}

Remarkably, the $7$-simplex $\esimp^7$ in Proposition~\ref{prop:verticesboththorpe} is such that its images under projections that eliminate one or both of the variables $t_1$ and $t_2$ are also simplices. 

\begin{proposition}\label{prop:einsteinsimplex}
Let $\esimp^5\subset\R^5$ be the $5$-simplex defined as convex hull of the rows  $p_1,\ldots, p_6$ of the matrix:
 $$\bordermatrix{
	   &  w_1^{+} &  w_2^{+} & w_1^{-} & w_2^{-} & u  \cr
p_1&\frac{2}{3}(\delta-1)&\frac{2}{3}(\delta-1) &0&0&\frac{1}{3}(2\delta+1) \rule[-6pt]{0pt}{10pt} \cr
p_2&\frac{4}{3}(\delta-1) &\frac{2}{3}(1-\delta) &0&0&\frac{1}{3}(\delta+2) \rule[-6pt]{0pt}{10pt} \cr
p_3&0&0&\frac{2}{3}(\delta-1) &\frac{2}{3}(\delta-1) &\frac{1}{3}(2\delta+1)\rule[-6pt]{0pt}{10pt} 
\cr	   p_4&0&0&\frac{4}{3}(\delta-1) &\frac{2}{3}(1-\delta) &\frac{1}{3}(\delta+2) \rule[-6pt]{0pt}{10pt} \cr
% v_5,v_6
p_5 &0&0&0&0&1 \rule[-6pt]{0pt}{10pt} \cr
% v_7,v_8
p_6 &0&0&0&0&\delta }$$
and consider its image under the linear map $\iota_5\colon\R^5\to\R^7$, given by
\begin{equation*}
\iota_5\big(w_1^+,w_2^+,w_1^-,w_2^-,u\big)=\big(w_1^+,w_2^+,-w_1^+ -w_2^+,w_1^-,w_2^-,-w_1^- -w_2^-,u\big).
\end{equation*}
Then $(\vec w_+,\vec w_-,u)\in \R^7$ is in $\iota_5(\esimp^5)$ if and only if $R(\vec w_+,\vec w_-,u)\in\epinched$.
\end{proposition}

\begin{proof}
By Lemma~\ref{lem:einsteinineqs} and Proposition~\ref{prop:verticesboththorpe}, we have that $R(\vec w_+,\vec w_-,u)\in\epinched$ if and only if $(\vec w_+,\vec w_-,u)\in\R^7$ is such that $(\vec w_+,\vec w_-,u,t_1,t_2)\in\iota_7(\esimp^7)$ for some $t_1,t_2\in\R$. In other words, if and only if $(\vec w_+,\vec w_-,u)\in\Pi\big(\iota_7(\esimp^7)\big)$ where $\Pi\colon\R^9\to\R^7$ is the projection that eliminates the last two coordinates. The conclusion follows, since
\begin{equation*}
\begin{aligned}
\Pi\big(\iota_7(\esimp^7)\big) &=\Pi\big(\iota_7( \conv(v_1,\dots,v_8) )\big)\\
&=\conv\!\big(\Pi(\iota_7(v_1)),\dots,\Pi(\iota_7(v_8))\big)\\
&=\conv(\iota_5(p_1),\dots,\iota_5(p_6))\\
&=\iota_5(\esimp^5),
\end{aligned}
\end{equation*}
where the third equality holds because $\Pi(\iota_7(v_j))=p_j$ if $1\leq j\leq 4$, $\Pi(\iota_7(v_j))=p_5$ if $j=5,6$, and $\Pi(\iota_7(v_j))=p_6$ if $j=7,8$. 
\end{proof}

\begin{proposition}\label{prop:einsteinonethorpe}
Let $\esimp^6\subset\R^6$ be the $6$-simplex defined as convex hull of the rows  $q_1,\ldots, q_7$ of the matrix:
\begin{equation*}
\bordermatrix{
  &  w_1^{+} &  w_2^{+} & w_1^{-} & w_2^{-} & u & t_1  \cr
q_1&\frac{2}{3}(\delta-1)&\frac{2}{3}(\delta-1)&0&0&\frac{1}{3}(2\delta+1)& \frac{1}{3}(1-\delta)\rule[-6pt]{0pt}{10pt} \cr
q_2&\frac{4}{3}(\delta-1)&\frac{2}{3}(1-\delta)  &0&0&\frac{1}{3}(\delta+2) &\frac{2}{3}(1-\delta)\rule[-6pt]{0pt}{10pt} \cr
q_3&0&0&\frac{2}{3}(\delta-1) &\frac{2}{3}(\delta-1)&\frac{1}{3}(2\delta+1)&\frac{1}{3}(\delta-1)\rule[-6pt]{0pt}{10pt} \cr
q_4&0&0&\frac{4}{3}(\delta-1)&\frac{2}{3}(1-\delta) &\frac{1}{3}(\delta+2)&\frac{2}{3}(\delta-1)\rule[-6pt]{0pt}{10pt} \cr
q_5&0&0&0&0&1&\delta-1\rule[-6pt]{0pt}{10pt} \cr
q_6&0&0&0&0&1&1-\delta \rule[-6pt]{0pt}{10pt} \cr
% v_7,v_8
q_7 &0&0&0&0&\delta&0 }
\end{equation*}
and consider its image under the linear map $\iota_6\colon\R^6\to\R^8$, given by
\begin{equation*}
\iota_6\big(w_1^+,w_2^+,w_1^-,w_2^-,u,t_1\big)=\big(w_1^+,w_2^+,-w_1^+ -w_2^+,w_1^-,w_2^-,-w_1^- -w_2^-,u,t_1\big).
\end{equation*}
Then $(\vec w_+,\vec w_-,u,t_1)\in\R^8$ is in $\iota_6(\esimp^6)$ if and only if $R(\vec w_+,\vec w_-,u-\delta)+t_1\, *\succeq0$ and $R(\vec w_+,\vec w_-,u)\in\epinched$.
\end{proposition}

\begin{proof}
Lemma~\ref{lem:einsteinineqs} and Proposition~\ref{prop:verticesboththorpe} imply that $R(\vec w_+,\vec w_-,u-\delta)+t_1\, *\succeq0$ and $R(\vec w_+,\vec w_-,u)\in\epinched$ if and only if 
$(\vec w_+,\vec w_-,u,t_1)\in\iota_7(\esimp^7)$ for some $t_2\in\R$, i.e., $(\vec w_+,\vec w_-,u,t_1)\in \Pi'\big(\iota_7(\esimp^7)\big)$, where $\Pi'\colon \R^9\to\R^8$ is the projection that eliminates the last coordinate. Similarly to Proposition~\ref{prop:einsteinsimplex}, we have that $\Pi'\big(\iota_7(\esimp^7)\big)=\iota_6(\esimp^6)$, as $\Pi'(\iota_7(v_j))=q_j$ if $1\leq j\leq 6$, and $\Pi'(\iota_7(v_j))=q_7$ if $j=7,8$.
\end{proof}

We shall refer to $\esimp^5$ as the \emph{Einstein simplex}, and to $\esimp^6$ and $\esimp^7$ as \emph{augmented Einstein simplices}. 
The rationale for this nomenclature is that, by Proposition~\ref{prop:einsteinsimplex} and \eqref{eq:curvop4}, the set of conjugacy classes of positively $\delta$-pinched Einstein curvature operators is parametrized by $\esimp^5$. Indeed, $\varphi\colon \esimp^5\to \epinched$, where $\varphi=R\circ\iota_5$ and $R\colon \R^7\to\mathcal U\oplus\mathcal W$ is given by \eqref{eq:Rww}, is an affine map whose image is a section for the change of basis $\SO(4)$-action on $\epinched\subset \mathcal U\oplus\mathcal W$. Analogously, $\esimp^6$ and $\esimp^7$ parametrize this set together with the corresponding $t_1$ and $t_2$ for which \eqref{eq:t1t2bound} holds. 

\begin{remark}\label{rem:geominterp}
The vertices $p_1,\dots,p_6$ of the Einstein simplex $\esimp^5$ correspond to geometrically meaningful curvature operators. Namely, using $\varphi=R\circ\iota_5$, we have: 
\begin{equation*}
\begin{aligned}
\varphi(p_1)&=\tfrac{1-\delta}{3} R_{\C P^2}+\tfrac{4\delta-1}{3}R_{S^4}, & 
\varphi(p_3)&=\tfrac{1-\delta}{3} R_{\overline{\C P^2}}+\tfrac{4\delta-1}{3}R_{S^4}, &
\varphi(p_5)&= R_{S^4},\\
\varphi(p_2)&=\tfrac{1-\delta}{3} R_{\C H^2}+\tfrac{4-\delta}{3}R_{S^4},  &
\varphi(p_4)&=\tfrac{1-\delta}{3} R_{\overline{\C H^2}}+\tfrac{4-\delta}{3}R_{S^4}, &
\varphi(p_6)&=\delta R_{S^4},
\end{aligned}
\end{equation*}
where $R_{S^4}=\id$, while $R_{\C P^2}$ and $R_{\C H^2}$ are given in \eqref{eq:RCPH2}, and satisfy $1\leq \sec\leq 4$ and $-4\leq \sec\leq -1$ respectively. 
Recall that $\overline{\C P^2}$ and $\overline{\C H^2}$ are the manifolds $\C P^2$ and $\C H^2$ with the opposite orientation. Being positively $\delta$-pinched is invariant under change of orientation (which interchanges $\vec w_+$ and $\vec w_-$ and fixes $u$), so the collection of vertices also has this symmetry. This is clear by comparing the first two columns above and recalling that $S^4$ has orientation-reversing isometries. Finally, note that $\varphi(p_j)$ depend affinely on $\delta$, and, of course, become equal to $R_{S^4}$ if $\delta=1$.
\end{remark}

\section{Traceless Ricci bounds}
The purpose of this section is to prove a new upper bound (Proposition~\ref{prop:boundC}) on the norm of the traceless Ricci part of $4$-dimensional curvature operators with either a lower or upper sectional curvature bound. In addition to its role in the proof of Theorem~\ref{mainthm:Lambda}, we believe this result is of independent interest and may have other applications; e.g., it yields a simple proof of the algebraic Hopf question in dimension $4$, see Corollary~\ref{cor:alg_hopf_quest}. We begin with two algebraic lemmas.
 
\begin{lemma}\label{lem:derangement}
Let $0<\lambda_1<\cdots <\lambda_n$ and $0<\mu_1< \cdots<\mu_n$. For all permutations $\phi\in\mathfrak{S}_n$ we have
\begin{equation*}
\sum_{i=1}^n\lambda_i\mu_{\phi(i)}\leq\sum_{i=1}^n\lambda_i\mu_{i}.
\end{equation*}
\end{lemma}

\begin{proof}
 Suppose, by contradiction, that the permutation $\phi\in\mathfrak{S}_n$ that maximizes $\sum_{i=1}^n\lambda_i\mu_{\phi(i)}$ is not the identity. Then, there are $1\leq i<j\leq n$ with $\phi(i)>\phi(j)$. We have
 \begin{equation*}
 (\lambda_i\mu_{\phi(j)}+\lambda_j\mu_{\phi(i)})-(\lambda_i\mu_{\phi(i)}+\lambda_j\mu_{\phi(j)})=(\lambda_j-\lambda_i)(\mu_{\phi(i)}-\mu_{\phi(j)})>0,
 \end{equation*} 
 contradicting the maximality of $\sum_{i=1}^n\lambda_i\mu_{\phi(i)}$.
\end{proof}

\begin{lemma}\label{lem:niceinequality}
Let $0\leq\lambda_1\leq\cdots \leq\lambda_n$, $0\leq\mu_1\leq \cdots\leq\mu_n$, and $C\in\textnormal{Mat}_{n\times n}(\R)$ be such that  
\begin{equation}\label{eq:lemmapsd}
\begin{pmatrix}
 \diag(\lambda_1, \ldots, \lambda_n)& C^\mathrm t\\
C & \diag(\mu_1, \ldots, \mu_n)
\end{pmatrix}\succeq0.
\end{equation}
Then $|C|^2\leq \sum\limits_{i=1}^n \lambda_i \mu_i$.
\end{lemma}

\begin{proof}
By continuity, we shall assume $0<\lambda_1<\cdots <\lambda_n$ and $0<\mu_1< \cdots<\mu_n$. Using Schur complements, we see that \eqref{eq:lemmapsd} holds if and only if $$\diag(\mu_1,\ldots,\mu_n)-C\diag(\lambda^{-1}_1, \ldots, \lambda^{-1}_n)C^\mathrm t\succeq0.$$ This is equivalent to $D=(d_{ij})_{1\leq i,j\leq n}$ lying in the unit ball with respect to the spectral norm, where $d_{ij}=\frac{c_{ij}}{\sqrt{\lambda_i\mu_j}}$. We thus want to bound
\begin{equation}\label{eq:birkhoff}
|C|^2=\sum_{i,j=1}^n\lambda_i\mu_j\,d_{ij}^2
\end{equation}
from above. 
The extreme points of the unit ball in $\textnormal{Mat}_{n\times n}(\R)$ with respect to the spectral norm are orthogonal matrices (see e.g.~\cite[Thm.~4(i)]{grone}), so we may assume $D$ is orthogonal, as the right-hand side of \eqref{eq:birkhoff} is a convex function in its entries.
In that case, the matrix $D_2=(d_{ij}^2)_{1\leq i,j\leq n}$ is \emph{doubly stochastic}, i.e., each of its rows and columns sums to $1$. By the Birkhoff--von Neumann Theorem (see e.g.~\cite[Thm.~II.5.2]{barvinok}), every doubly stochastic matrix lies in the convex hull of permutation matrices. Thus, for bounding \eqref{eq:birkhoff} from above, we may further assume $D=D_2$ is a permutation matrix, so the conclusion follows from Lemma~\ref{lem:derangement}.
\end{proof}

We are now ready for the main result of this section. Although it solely regards algebraic curvature operators in $\Sym^2_b(\wedge^2\R^4)$, we state it as a pointwise estimate on a Riemannian $4$-manifold to render it more easily applicable elsewhere.

\begin{proposition}\label{prop:boundC}
Let $(M^4,\g)$ be a $4$-manifold, $p\in M$, and $k\in\R$. 
Let $R$ be the curvature operator at $p\in M$, and $u$, $\vec w_\pm$, and $C$ as in \eqref{eq:curvop4} and \eqref{eq:widef}.
\begin{enumerate}[\rm (i)]
\item If all $2$-planes in $T_pM$ have either $\sec\geq k$ or $\sec\leq k$, then
\begin{equation*}
|C|^2\leq 3(u-k)^2+\langle \vec w_+,\vec w_-\rangle,
\end{equation*}
\item If $t\in\R$ is such that $\pm (R-k\id) +t\,*\succeq0$ on $\wedge^2 T_pM$, then
\begin{equation*}
|C|^2\leq 3(u-k)^2-3t^2+\langle \vec w_+,\vec w_-\rangle.
\end{equation*}
\end{enumerate}
\end{proposition}

\begin{proof}
As elsewhere in the paper, we identify $T_pM\cong \R^4$ and assume the curvature operator $R$ is in the canonical form \eqref{eq:curvop4}. 
By Finsler--Thorpe's Trick (Lemma~\ref{lem:finslerthorpe}), if either $\sec\geq k$ or $\sec\leq k$, then there exists $t\in\R$ such that $R-k\id +t\,*\succeq0$ or $-(R-k\id) +t\,*\succeq0$, respectively; so it suffices to prove (ii).

Condition \eqref{eq:lemmapsd} is verified since $\pm (R- k\id)+t\,*\succeq0$, so we may apply Lemma~\ref{lem:niceinequality} with $n=3$, $\lambda_i=u-k+w_i^++t$, and $\mu_i=u-k+w_i^--t$, concluding that
\begin{align*}
   |C|^2&\leq \sum_{i=1}^3(u-k+w_i^++t)(u-k+w_i^--t) \\
  &= \sum_{i=1}^3 (u-k+t)(u-k-t)+w_i^+w_i^-\\
  &= 3(u-k)^2-3t^2+\langle \vec w_+,\vec w_-\rangle,
 \end{align*}
where the first equality uses \eqref{eq:ww}.
\end{proof}

While the general case of the Hopf question asking whether closed $2d$-dimensional manifolds $(M^{2d},\g)$ with $\sec\geq0$ or $\sec\leq0$ have $(-1)^d\chi(M)\geq0$ remains an important open problem, its \emph{algebraic} variant asking whether 
the Chern--Gauss--Bonnet integrand $\underline{\chi}(R)$ computed at an algebraic curvature operator $R\in\Sym^2(\wedge^2\R^{2d})$ with $\sec_{\pm R} \geq0$ satisfies $(-1)^d\underline{\chi}(R)\geq0$ was answered affirmatively if $2d=4$ by Milnor~\cite{chern,bishop-goldberg}, and negatively if $2d\geq6$ by Geroch~\cite{geroch,klembeck}.
The former result of Milnor can be easily recovered with Proposition~\ref{prop:boundC}, which also allows to characterize the equality case, as follows.

\begin{corollary}[Algebraic Hopf question in dimension 4]\label{cor:alg_hopf_quest}
If $R\in\Sym^2(\wedge^2\R^{4})$ has $\sec_{\pm R} \geq0$, then $\underline{\chi}(R)\geq0$. Moreover, $\underline{\chi}(R)=0$ if and only if $\pm R\succeq 0$, $W_+=W_-$, and $|C|^2=3u^2+|W_\pm|^2$; in particular, if $R$ is Einstein, then $R=0$.
\end{corollary}

\begin{proof}
As before, let $t\in\R$ be such that $\pm R+t\,*\succeq0$, see Finsler--Thorpe's Trick (Lemma~\ref{lem:finslerthorpe}). By Proposition~\ref{prop:boundC} (ii) and \eqref{eq:chi}, we obtain
\begin{align*}
      8\underline{\chi}(R) &= 6u^2+|W_+|^2+|W_-|^2 -2|C|^2 \\
         & \geq 6t^2+|\vec w_+|^2+|\vec w_-|^2 -2\langle \vec w_+,\vec w_-\rangle \\
        &=6t^2 + \big|\vec w_+ -\vec w_-\big|^2\geq0.
\end{align*}
Moreover, $\underline{\chi}(R)=0$ if and only if equality holds in all above inequalities.
\end{proof}

\section{Lower bounds}\label{sec:lower}

In this section, we establish a (pointwise) lower bound for the quadratic form $I_\lambda$ in the curvature operator of an oriented $4$-manifold $M$ that integrates to
\begin{equation}\label{eq:chi-1/lambdasigma}
\chi(M)-\tfrac{1}{\lambda}\sigma(M), \quad \lambda>0,
\end{equation}
see \eqref{eq:ilambda} for details. Given $0<\delta\leq1$, this lower bound gives sufficient conditions on $\lambda$ for the integrand $I_\lambda$ to be nonnegative on $\delta$-pinched curvature operators,
hence for \eqref{eq:chi-1/lambdasigma} to be nonnegative if $(M^4,\g)$ is $\delta$-pinched, see~Theorem~\ref{thm:estimates-general}. Combined with Theorem~\ref{thm:ville}, this yields Theorem~\ref{mainthm:Lambda} in the Introduction.

First, we focus on the particular case $\lambda=\tfrac12$, to demonstrate the optimization arguments used in the general case more concretely, 
and simplify the exposition for readers mainly interested in a self-contained proof of Theorem~\ref{mainthm:classification}, given below.

\begin{theorem}\label{thm:estimates-baby}
If $(M^4,\g)$ is a $\delta$-pinched oriented $4$-manifold, with finite volume and $\delta>\frac{-199+9\sqrt{545}}{71}\cong 0.156$, then $\chi(M)-2|\sigma(M)|>0$.
\end{theorem}

\begin{proof}[Proof of Theorem \ref{mainthm:classification}]
Since $(M^4,\g)$ is positively $\delta$-pinched, $\delta\geq\frac{1}{1+3\sqrt3}$, its intersection form is definite~\cite[Thm.~1]{DR19}. Up to reversing orientation, we assume it is positive-definite, i.e., $b_-(M)=0$, so $\chi(M)=2+b_+(M)$ and $\sigma(M)=b_+(M)\geq0$. By Theorem~\ref{thm:estimates-baby}, we have 
$0<\chi(M)-2|\sigma(M)| = 2-b_+(M)$,
so $b_+(M)=0$ or $1$. Therefore, by Theorem~\ref{thm:4Dmanifolds}, we conclude $M$ is homeomorphic to $S^4$ or $\C P^2$.
\end{proof}

\begin{proof}[Proof of Theorem~\ref{thm:estimates-baby}]
Given a $\delta$-pinched oriented $4$-manifold $(M^4,\g)$ with finite volume, up to reversing its orientation, we shall assume that $\sigma(M)\geq0$. Moreover, at each point $p\in M$, its curvature operator $R_p\in\Sym^2_b(\wedge^2\R^4)$ satisfies $\pm R\in\dpinched$, see \eqref{eq:dpinched}.
Writing $R$ in the canonical form \eqref{eq:curvop4}, we have from \eqref{eq:integrals}, \eqref{eq:chi}, and \eqref{eq:sigma}, that
\begin{equation}\label{eq:ville}
I_{\frac12}(R):=\underline{\chi}(R)-2\underline{\sigma}(R)=
\tfrac{3}{4}u^2-\tfrac{1}{24}|W_+|^2 +\tfrac{7}{24}|W_-|^2-\tfrac{1}{4}|C|^2
\end{equation}
satisfies $\int_M I_{\frac12}(R)\,\vol_\g=\chi(M)-2\sigma(M)$, and $I_{\frac12}(-R)=I_{\frac12}(R)$. Thus, it suffices to prove that 
\begin{equation}\label{eq:goal12}
\min_{R\in\dpinched} I_{\frac12}(R)>0, \; \text{ if } \; \delta>\tfrac{-199+9\sqrt{545}}{71}.
\end{equation}

Suppose $R\in\dpinched$, and let $t_1,t_2\in\R$ be as in \eqref{eq:t1t2bound}, see~Lemma~\ref{lem:finslerthorpe}. From Proposition~\ref{prop:boundC} (ii) with $k=\delta$ and $t=t_1$, we have:
\begin{equation}\label{eq:boundC-ours}
   |C|^2 \leq 3(u-\delta)^2-3t_1^2+\langle \vec w_+,\vec w_-\rangle.
 \end{equation}
Therefore, we may bound \eqref{eq:ville} from below using \eqref{eq:boundC-ours} as follows:
\begin{equation*}
I_{\frac12}(R)\geq\textstyle \tfrac{3}{4}u^2-\tfrac{1}{24}|W_+|^2 +\tfrac{7}{24}|W_-|^2 -\frac34(u-\delta)^2+\frac34 t_1^2-\frac14\langle \vec w_+,\vec w_-\rangle.
\end{equation*}
Moreover, using \eqref{eq:ww}, this lower bound can be written as the quadratic polynomial
\begin{align*}
Q_{\frac12}(w_1^+,w_2^+,w_1^-,w_2^-,u,t_1)  &:=\textstyle -\frac{1}{12}\big((w_1^+)^2 +(w_2^+)^2+w_1^+w_2^+ \big)\\
 &\quad\textstyle +\frac{7}{12}\big((w_1^-)^2 +(w_2^-)^2+w_1^-w_2^- \big)\\
 &\quad\textstyle  -\frac{1}{2}\big(w_1^+w_1^- +w_2^+w_2^- \big) -\frac{1}{4}\big(w_1^+w_2^- + w_2^+w_1^-\big) \\
 &\quad\textstyle +\frac{3}{4}t_1^2  +\frac{3}{2}\delta u-\frac{3}{4}\delta^2,
 \end{align*}
which depends solely on $t_1$ and the $5$ variables $w_1^+,w_2^+,w_1^-,w_2^-,u$ that determine $\pr(R)=R(\iota_5(w_1^+,w_2^+,w_1^-,w_2^-,u))\in\epinched$. Therefore, by Propositions \ref{prop:projd} and \ref{prop:einsteinonethorpe},  
\begin{equation}\label{eq:Qbound}
\min_{R\in\dpinched} \, I_{\frac12}(R)
%\underline{\chi}(R)-2\underline{\sigma}(R)
 \; \geq \; \min_{x\in\esimp^6} \, Q_{\frac12}(x),
\end{equation}
where $\esimp^6=\conv(q_1,\dots,q_7)$ is the augmented Einstein simplex in Proposition~\ref{prop:einsteinonethorpe}.

In order to compute the minimum value of $Q_{\frac12}\colon\R^6\to\R$ on $\esimp^6$, we apply the optimization method discussed in Section~\ref{subsec:optsimplex} to maximize $-Q_{\frac12}$ on $\esimp^6$. The first step is to compute the collection $\mathcal S$ of subsets $S\subset \{q_1,\dots,q_7\}$ on whose affine hull $\aff(S)$ the restriction of $Q_{\frac12}$ is positive-definite, see \eqref{eq:verticesnegdef}. 
Note that the restriction of $Q_{\frac12}\colon\R^6\to\R$ to any such affine subspace $\aff(S)\subset\R^6$ is positive-definite for \emph{some} $0<\delta<1$ if and only if it is positive-definite for \emph{all} $0<\delta<1$. Indeed, from Proposition~\ref{prop:einsteinonethorpe}, each coordinate of $q_j$ is a scalar multiple of $(1-\delta)$, except for the $u$-coordinate, on which $Q_{\frac12}$ has no degree 2 term. Therefore, the eigenvalues of $\operatorname{Hess}\!\big(Q_{\frac12}|_{\aff(S)}\big)$ are scalar multiples of $(1-\delta)^2$.

A direct computation shows that $\operatorname{Hess} Q_{\frac12}$ has eigenvalues
\begin{equation*}
\textstyle
\frac{2}{3},\; 2,\; \frac{3}{2},\; -\frac{1}{2}, \; -\frac{1}{6}, \; 0.
\end{equation*}
As there are $d=3$ positive eigenvalues, by Corollary~\ref{cor:smallsig}, it suffices to consider subsets $S$ consisting of at most $4$ vertices, i.e., such that the face $\conv(S) \subset \esimp^6$ has dimension $\leq 3$. 

All $0$-dimensional faces of $\esimp^6$, i.e., singletons $S=\{q_j\}$, $1\leq j \leq 7$, trivially belong to $\mathcal S$. 
Regarding $1$-dimensional faces, it is straightforward to verify that $18$ of the $21= \binom{7}{2}$ subsets $S=\{q_{j_1},q_{j_2}\}$ of 2 vertices belong to $\mathcal S$; namely, all except for $\{q_1,q_2\}$, $\{q_1,q_7\}$ and $\{q_2,q_6\}$. For instance, the Hessian $1\times 1$-matrix of the restriction of $Q_{\frac12}$ to $\aff(q_1,q_2)$ is $-\frac{1}{18}(1-\delta)^2$, while to $\aff(q_1,q_3)$ it is $\frac{10}{3}(1-\delta)^2$.
Similarly, concerning $2$-dimensional faces, $18$ of the $35= \binom{7}{3}$ subsets of 3 vertices $S=\{q_{j_1},q_{j_2},q_{j_3}\}$ belong to $\mathcal S$; namely, 
\begin{equation}\label{eq:2facesinS}
\begin{array}{cccccc}
\{q_1,q_3,q_4\}, &\!\!\! \{q_1,q_3,q_5\}, &\!\!\! \{q_1,q_3,q_6\}, &\!\!\! \{q_1,q_4,q_5\}, &\!\!\! \{q_1,q_4,q_6\}, &\!\!\! \{q_2,q_3,q_4\},\\[3pt]
\{q_2,q_3,q_5\}, &\!\!\! \{q_2,q_3,q_7\}, &\!\!\! \{q_2,q_4,q_5\}, &\!\!\! \{q_3,q_4,q_5\}, &\!\!\! \{q_3,q_4,q_6\}, &\!\!\! \{q_3,q_4,q_7\},\\[3pt]
\{q_3,q_5,q_6\}, &\!\!\! \{q_3,q_5,q_7\}, &\!\!\! \{q_3,q_6,q_7\}, &\!\!\! \{q_4,q_5,q_6\}, &\!\!\! \{q_4,q_5,q_7\}, &\!\!\! \{q_4,q_6,q_7\}.
\end{array}
\end{equation}
For example, the Hessian $2\times 2$-matrix of the restrictions of $Q_{\frac12}$ to $\aff(q_1,q_2,q_3)$ and $\aff(q_1,q_3,q_4)$ have eigenvalues 
$\tfrac{59\pm\sqrt{3737}}{36}(1-\delta)^2$ and $\tfrac{123\pm\sqrt{8473}}{36}(1-\delta)^2$, respectively. Finally, in regard to $3$-dimensional faces, 6 of the $35=\binom{7}{4}$ subsets of 4 vertices $S=\{q_{j_1},q_{j_2},q_{j_3},q_{j_4}\}$ are in $\mathcal S$; namely
\begin{equation}\label{eq:3facesinS}
\begin{array}{ccc}
\{q_1,q_3,q_4,q_5\},  &\{q_1,q_3,q_4,q_6\}, & \{q_2,q_3,q_4,q_5\}, \\[3pt]
\{q_3,q_4,q_5,q_6\}, & \{q_3,q_4,q_5,q_7\}, & \{q_3,q_4,q_6,q_7\}.
\end{array}
\end{equation}
For instance, the Hessian $3\times 3$-matrix of the restriction of $Q_{\frac12}$ to $\aff(q_1,q_2,q_3,q_4)$ and $\aff(q_1,q_3,q_4,q_5)$ have eigenvalues
\begin{equation*}
\tfrac{\alpha_1}{18}(1-\delta)^2, \; \tfrac{\alpha_2}{18}(1-\delta)^2, \; \tfrac{\alpha_3}{18}(1-\delta)^2, \quad\text{and}\quad  \tfrac{\beta_1}{18}(1-\delta)^2, \; \tfrac{\beta_2}{18}(1-\delta)^2, \; \tfrac{\beta_3}{18}(1-\delta)^2,
\end{equation*}
respectively, where $\alpha_1\cong 109.22$, $\alpha_2\cong20.12$, and $\alpha_3\cong -7.34$ are the roots of the polynomial $x^3-122x^2+1248x+16128$, and $\beta_1\cong 134.65$, $\beta_2\cong 19.27$, and $\beta_3\cong 13.06$ are the roots of the polynomial $x^3-167x^2+4608x-33936$.

The second step in the optimization procedure is to compute the unique critical point $x_S\in\aff(S)$ of $Q_{\frac12}|_{\aff(S)}$ for each of the above 49 subsets $S\in\mathcal S$, and build the subcollection $\mathcal S'\subset\mathcal S$ consisting of the $S\in\mathcal S$ such that $x_S$ is in the relative interior of the face $\conv(S)$, see \eqref{eq:verticesnegdefint}.
Differently from the above, this step \emph{depends} on the value of $0<\delta<1$, and several $S\in\mathcal S$ only join the collection $\mathcal S'$ for $\delta>0$ sufficiently small. Explicitly parametrizing each face $\conv(S)$ for $S\in\mathcal S$ with a standard simplex, and solving the corresponding inequalities in $\delta$ to determine if $x_S\in\textnormal{relint}(\conv(S))$, we compute the conditions for which $S\in\mathcal S'$ and the corresponding list of values $Q_{\frac12}(x_S)$ where the minimum of $Q_{\frac12}|_{\aff(S)}$ is achieved.

All singletons $S=\{q_j\}$ trivially belong to $\mathcal S'$ for all $0<\delta<1$, and have $x_S=q_j$. The value $Q_{\frac12}(x_S)=Q_{\frac12}(q_j)$ for each of these points is listed in Table~\ref{tab:0faces}.
\begin{table}[!ht]
\begin{tabular}{|c|l|}
\hline
$S$ & $Q_{\frac12}(x_S)\phantom{\Big|}$  \\[2pt]
\hline
$\{q_1\}$ & $\frac{2 }{9}\delta^2+\frac{5}{9}\delta-\frac{1}{36}$ \rule[4pt]{0pt}{8pt} \\[3pt]
$\{q_2\}$ & $-\frac{1}{36}\delta^2+\frac{5}{9}\delta+\frac{2}{9}$ \rule[4pt]{0pt}{8pt} \\[3pt]
$\{q_3\}$ & $\frac{10 }{9}\delta^2-\frac{11}{9}\delta+\frac{31}{36}$ \rule[4pt]{0pt}{8pt} \\[3pt]
$\{q_4\}$ & $\frac{31 }{36}\delta^2-\frac{11}{9}\delta+\frac{10}{9}$ \rule[4pt]{0pt}{8pt} \\[3pt]
\hline
\end{tabular}
\begin{tabular}{|c|l|}
\hline
$S$ & $Q_{\frac12}(x_S)\phantom{\Big|}$  \\[2pt]
\hline
$\{q_5\}$ & $\frac{3}{4}$  \rule[4pt]{0pt}{8pt} \\[6pt]
$\{q_6\}$ & $\frac{3}{4}$  \rule[4pt]{0pt}{8pt} \\[7pt]
$\{q_7\}$ & $\frac{3\delta^2}{4}$ \\[15pt]
\hline
\end{tabular}
\caption{Values of $Q_{\frac12}$ on the $0$-dimensional faces of $\esimp^6$.}\label{tab:0faces}
\end{table}

In 12 of the 18 subsets $S\in \mathcal S$ with $2$ vertices, the critical point $x_S\in\aff(S)$ of $Q_{\frac12}|_{\aff(S)}$ lies in the relative interior of the $1$-dimensional face $\conv(S)$ for some value of $0<\delta<1$, as listed in Table~\ref{tab:1faces}. For instance, the critical point of $Q_{\frac12}|_{\aff(q_1,q_3)}$ is $x_{\{q_1,q_3\}} =\frac{23}{30}\,q_1+\frac{7}{30}\,q_3$, which clearly lies in $\textnormal{relint}(\conv(q_1,q_3))$ for all $0<\delta<1$. Meanwhile, the critical point of $Q_{\frac12}|_{\aff(q_1,q_4)}$ is $x_{\{q_1,q_4\}}=\frac{52-43\delta}{63(1-\delta)}\,q_1+\frac{11-20\delta}{63(1-\delta)}\,q_4$, which lies in $\textnormal{relint}(\conv(q_1,q_4))$ if and only if $0<\delta<\frac{11}{20}$.

\begin{table}[!ht]
\begin{tabular}{|c|l|c|}
\hline
$S$ & $Q_{\frac12}(x_S)$ & $\phantom{\Big|}\delta_S\phantom{\Big|}$ \\[2pt]
\hline
$\{q_1,q_3\}$ & $\frac{71}{540}\delta^2+\frac{199}{270} \delta -\frac{16}{135} $ & $1$\rule[4pt]{0pt}{8pt} \\[3pt]
$\{q_1,q_4\}$ & $\frac{26}{567}\delta^2+\frac{425}{567} \delta -\frac{46}{567}$ & $\frac{11}{20}$ \\[3pt] 
$\{q_1,q_5\}$ & $-\frac{9}{44}\delta^2+\frac{9}{11} \delta -\frac{3}{44}$ & $\frac{4}{13}$ \\[3pt] 
$\{q_2,q_3\}$ & $-\frac{46}{567}\delta^2+\frac{425}{567} \delta +\frac{26}{567}$ & $\frac{43}{52}$ \\[3pt] 
$\{q_2,q_4\}$ & $-\frac{35}{108}\delta^2+\frac{31}{27} \delta -\frac{2}{27}$ & $1$ \\[3pt]
$\{q_2,q_5\}$ & $-\frac{36}{71}\delta^2+\frac{90}{71} \delta -\frac{3}{71}$ & $\frac{26}{35}$ \\[3pt]
\hline
\end{tabular}
\begin{tabular}{|c|l|c|}
\hline
$S$ & $Q_{\frac12}(x_S)$ & $\phantom{\Big|}\delta_S\phantom{\Big|}$ \\[2pt]
\hline
$\{q_3,q_4\}$ & $ \frac{70}{93}\delta^2-\frac{77}{93}\delta+\frac{70}{93}$ & $\frac{11}{20}$\rule[4pt]{0pt}{8pt} \\[3pt]
$\{q_3,q_5\}$ & $\frac{21}{40}$ & $\frac{11}{20}$ \\[3pt]
$\{q_3,q_6\}$ & $-\frac{9}{76}\delta^2+\frac{9}{19}\delta+\frac{21}{76}$ & $\frac{20}{29}$ \\[3pt] 
$\{q_4,q_5\}$ & $\frac{21}{31}$ & $\frac{22}{31}$ \\[3pt] 
$\{q_4,q_6\}$ & $-\frac{36}{103}\delta^2+\frac{90 }{103}\delta +\frac{21}{103}$ & $\frac{58}{67}$ \\[3pt] 
$\{q_5,q_6\}$ & $-\frac{3}{4}\delta^2+\frac{3}{2}\delta $ & $1$ \\[3pt]
\hline
\end{tabular}
\caption{Minimum of $Q_{\frac12}|_{\aff(S)}$, attained at $x_S\in\aff(S)$, which is in the relative interior of $\conv(S)$ if and only if $0<\delta<\delta_S$, for each $S\in\mathcal S$ such that $\conv(S)$ is a $1$-dimensional face of $\esimp^6$. If $x_S\notin\textnormal{relint}(\conv(S))$ for all $0<\delta<1$, then the corresponding entry $S\in\mathcal S$  is suppressed.}\label{tab:1faces}
\end{table}

Among the 18 subsets $S\in\mathcal S$ with $3$ vertices, listed in \eqref{eq:2facesinS},
only 9 are such that the critical point $x_S$ of $Q_{\frac12}|_{\aff(S)}$ lies in the relative interior of the $2$-dimensional face $\conv(S)$ for some value of $0<\delta<1$, as listed in Table~\ref{tab:2faces}. For example, the critical point of the restriction of $Q_{\frac12}$ to $\aff(q_1,q_3,q_4)$ is
$$\textstyle x_{\{q_1,q_3,q_4\}}=\frac{3 (212- 191 \delta)}{832 (1-\delta)} \, q_1+\frac{19 \delta+188}{832(1- \delta)} \, q_3+\frac{4-139 \delta}{416 (1-\delta)}\, q_4,$$
which lies in the relative interior of $\conv(q_1,q_3,q_4)$ if and only if $0<\delta<\frac{4}{139}$.

\begin{table}[!ht]
\begin{tabular}{|c|l|c|}
\hline
$S$ & $Q_{\frac12}(x_S)$ & $\phantom{\Big|}\delta_S\phantom{\Big|}$ \\[2pt]
\hline
$\{q_1,q_3,q_4\}$ & $\frac{227}{4992}\delta^2+\frac{463 }{624}\delta-\frac{37}{312}$ & $\frac{4}{139}$ \rule[4pt]{0pt}{8pt}\\[3pt]
$\{q_1,q_3,q_5\}$ & $-\frac{45}{202}\delta^2+\frac{153}{202}\delta-\frac{12}{101}$ & $\frac{4}{139}$\rule[4pt]{0pt}{8pt} \\[3pt]
$\{q_1,q_4,q_5\}$ & $\frac{1075503}{1763584}\delta^2+\frac{439605}{881792}\delta-\frac{162729}{1763584}$ & $\frac{43}{439}$ \rule[4pt]{0pt}{8pt}\\[3pt]
$\{q_2,q_3,q_4\}$ & $\frac{15719}{14884}\delta^2+\frac{1601 }{3721}\delta-\frac{121}{3721}$ & $\frac{4}{31}<\delta<\frac{23}{41}$\rule[4pt]{0pt}{8pt} \\[3pt]
$\{q_2,q_3,q_5\}$ & $-\frac{1137329}{2883204}\delta^2+\frac{6550079}{5766408}\delta-\frac{2782793}{46131264}$ & $\frac{349}{844}$\rule[4pt]{0pt}{8pt} \\[3pt]
$\{q_2,q_4,q_5\}$ & $\frac{410727}{1201216}\delta^2+\frac{202671}{300304}\delta+\frac{10935}{300304}$ & $\frac{10}{37}$ \rule[4pt]{0pt}{8pt}\\[3pt]
$\{q_3,q_4,q_6\}$ & $\frac{44103 }{153664}\delta^2+\frac{4329 }{10976} \delta +\frac{423}{3136}$ & $\frac{7}{13}$ \rule[4pt]{0pt}{8pt}\\[3pt]
$\{q_3,q_5,q_6\}$ & $-\frac{2321}{7056}\delta^2+\frac{353}{504}\delta+\frac{31}{144} $ & $\frac{7}{13}$\rule[4pt]{0pt}{8pt} \\[3pt]
$\{q_4,q_5,q_6\}$ & $\frac{2091}{12544}\delta^2+\frac{57}{224}\delta+\frac{3}{16}$ & $\frac{28}{43}$ \rule[4pt]{0pt}{8pt}\\[3pt]
\hline
\end{tabular}
\caption{Minimum of $Q_{\frac12}|_{\aff(S)}$, attained at $x_S\in\aff(S)$, which is in the relative interior of $\conv(S)$ if and only if $0<\delta<\delta_S$, for each $S\in\mathcal S$ such that $\conv(S)$ is a $2$-dimensional face of $\esimp^6$; except for $S=\{q_2,q_3,q_4\}$, for which $x_S\in\textnormal{relint}(\conv(S))$ if and only if $\frac{4}{31}<\delta<\frac{23}{41}$. If $x_S\notin\textnormal{relint}(\conv(S))$ for all $0<\delta<1$, then the corresponding entry $S\in\mathcal S$  is suppressed.}\label{tab:2faces}
\end{table}

Lastly, of the 6 subsets $S\in\mathcal S$ with $4$ vertices, see \eqref{eq:3facesinS}, only $S=\{q_2,q_3,q_4,q_5\}$ is such that the critical point $x_S$ of $Q_{\frac12}|_{\aff(S)}$ lies in the relative interior of the $3$-dimensional face $\conv(S)$ for some $0<\delta<1$. Namely, we have that
$$\textstyle
x_{\{q_2,q_3,q_4,q_5\}}=\frac{63 (19-20 \delta)}{1846 (1-\delta)} \, q_2 +\frac{9 (57 \delta-8)}{923 (1-\delta)} \, q_3 +\frac{27 (19-20 \delta)}{1846 (1-\delta)} \, q_4+\frac{4 (35-134 \delta)}{923 (1-\delta)} \, q_5
$$
is in the relative interior of $\conv(q_2,q_3,q_4,q_5)$ if and only if $\frac{8}{57}<\delta<\frac{35}{134}$, and 
\begin{equation}\label{eq:3faceValue}
Q_{\frac12}\big(x_{\{q_2,q_3,q_4,q_5\}}\big)=\textstyle\frac{634359 }{851929}\delta^2+\frac{372771}{851929}\delta-\frac{309393}{13630864}.
\end{equation}

Altogether, it follows that the minimum of $Q_{\frac12}\colon\esimp^6\to\R$ is equal to the smallest $Q_{\frac12}(x_S)$ among the $S\in\mathcal S$ that are in the subcollection $\mathcal S'$ for the given value of $\delta$, as listed in Tables~\ref{tab:0faces} to \ref{tab:2faces} and \eqref{eq:3faceValue}. By direct inspection, setting $\delta=\frac{-199+9\sqrt{545}}{71}$, all $Q_{\frac12}(x_S)$ for which $S\in\mathcal S'$ are strictly positive, except for $Q_{\frac12}\big(x_{\{q_1,q_3\}}\big)=0$. Furthermore, subsets $S\in\mathcal S$ only enter or leave the subcollection $\mathcal S'$ at one of \emph{finitely many} possible values of $\delta$, so provided $\varepsilon>0$ is sufficiently small,
\begin{equation*} %= \min Q|_{\aff(q_1,q_3)} 
\min_{x\in\esimp^6} Q_{\frac12}(x) = Q_{\frac12}\big(x_{\{ q_1,q_3\}}\big) = \textstyle\frac{71}{540}\delta^2+\frac{199}{270} \delta -\frac{16}{135}, \; \text{for all}\; \textstyle \left|\delta-\frac{-199+9\sqrt{545}}{71}\right|<\varepsilon,
\end{equation*}
and, linearizing this quadratic polynomial at $\delta=\frac{-199+9\sqrt{545}}{71}$, one easily  sees that $\min\limits_{x\in\esimp^6} Q_{\frac12}(x)>0$ for $\frac{-199+9\sqrt{545}}{71}<\delta<\frac{-199+9\sqrt{545}}{71}+\varepsilon$.
The above, combined with the fact that $\Delta_\delta^6\subset\Delta_{\delta'}^6$ if $\delta'<\delta$ and hence $\min\limits_{x\in\esimp^6} Q_{\frac12}(x)$ is a monotonically increasing function of $0<\delta<1$, implies that
\begin{equation*}
\min_{x\in\esimp^6} Q_{\frac12}(x) > 0, \quad \text{ for all }\quad \delta>\tfrac{-199+9\sqrt{545}}{71}.
\end{equation*}
The above inequality and \eqref{eq:Qbound} imply \eqref{eq:goal12}, concluding the proof.
\end{proof}

We now proceed to the case of general $\lambda>0$, leading to Theorem~\ref{mainthm:Lambda} in the Introduction. The method of proof follows the same outline  of Theorem~\ref{thm:estimates-baby}.

\begin{theorem}\label{thm:estimates-general}
If $(M^4,\g)$ is a $\delta$-pinched oriented $4$-manifold with finite volume, then
\begin{equation*}
|\sigma(M)|\leq \lambda^*(\delta)\, \chi(M),
\end{equation*}
where $\lambda^*(\delta)$ is the continuously differentiable function given by
\begin{equation}\label{eq:lambda*} %\frac{\sigma(M)}{\chi(M)}\leq 
\lambda^*(\delta)=\begin{cases}
 \dfrac{\sqrt{\frac{24}{\delta}+8 -8 \delta+\delta^2}+\delta-4 }{6 (3-\delta)},
%\dfrac{4}{3\delta(4-\delta +\sqrt{\delta^2-8 \delta+8+24/\delta} ) }, 
& \text{ if } 0<\delta<\delta_1^*,  \\[10pt]
\dfrac{4 }{3 \sqrt{15}} \dfrac{1-\delta}{\sqrt{\delta (\delta+2)}},  &  \text{ if } \delta_1^*\leq \delta<\delta_2^*, \\[10pt]
\dfrac{8(1-\delta)^2}{24 \delta^2-12 \delta+15}, & \text{ if } \delta_2^*\leq \delta\leq 1,
\end{cases}
\end{equation}
and 
\begin{enumerate}[\rm (i)]
\item $\delta_1^*\cong0.069$ is the smallest real root of the polynomial $2\delta^3-40\delta^2+89\delta-6$,
\item $\delta_2^*=4-\frac{3\sqrt6}{2}\cong 0.326$.
\end{enumerate}
\end{theorem}

\begin{remark}
The above semialgebraic function $\lambda^*$ is $C^1$, but not $C^2$.
\end{remark}

\begin{proof}
Just as in the proof of Theorem~\ref{thm:estimates-baby}, up to reversing orientation, we may assume $\sigma(M)\geq0$; and, at every point $p\in M$, we have that $\pm R\in\dpinched$. Writing $R$ in the canonical form \eqref{eq:curvop4}, we have from \eqref{eq:integrals}, \eqref{eq:chi}, and \eqref{eq:sigma}, that for all~$\lambda>0$,
\begin{equation}\label{eq:ilambda}
\begin{aligned}
I_\lambda(R)&:=\underline{\chi}(R)-\tfrac{1}{\lambda}\underline{\sigma}(R)\\
&\phantom{:}= \tfrac{3}{4}u^2+\left(\tfrac{1}{8}-\tfrac{1}{12\lambda}\right)|W_+|^2 +\left(\tfrac{1}{8}+\tfrac{1}{12\lambda}\right)|W_-|^2-\tfrac{1}{4}|C|^2
\end{aligned}
\end{equation}
satisfies $\int_M I_\lambda(R)\,\vol_\g=\chi(M)-\frac{1}{\lambda}\sigma(M)$, and $I_\lambda(-R)=I_\lambda(R)$. Thus, it suffices to prove that for all $0<\delta<1$,
\begin{equation}\label{eq:goal-lambda}
\min_{R\in\dpinched} I_\lambda(R)\geq0, \; \text{ if } \; \lambda=\lambda^*(\delta).
\end{equation}
Note that the conclusion holds in the trivial case $\delta=1$ and $\lambda^*(1)=0$, as $\sigma(M)=0$ if $(M^4,\g)$ is $1$-pinched, i.e., has constant curvature; so we shall assume $0<\delta<1$.

Given $R\in\dpinched$, let $t_1,t_2\in\R$ be as in \eqref{eq:t1t2bound}, see Lemma~\ref{lem:finslerthorpe}. Using Lemma~\ref{lem:niceinequality} and arguing exactly as in \eqref{eq:boundC-ours}, we may bound \eqref{eq:ilambda} from below:
\begin{equation*}
I_\lambda(R)\geq \tfrac{3}{4} u^2 +\left(\tfrac{1}{8}-\tfrac{1}{12\lambda}\right)|W_+|^2 +\left(\tfrac{1}{8}+\tfrac{1}{12\lambda}\right)|W_-|^2 -\tfrac34(u-\delta)^2+\tfrac34 t_1^2-\tfrac14\langle \vec w_+,\vec w_-\rangle,
\end{equation*}
and, using \eqref{eq:ww}, this lower bound can be written as the quadratic polynomial
\begin{equation*}
\begin{aligned}
Q_{\lambda}(w_1^+,w_2^+,w_1^-,w_2^-,u,t_1) &:=
\textstyle\left(\frac{1}{4}-\frac{1}{6\lambda}\right)\big((w_1^+)^2+(w_2^+)^2+w_1^+w_2^+\big)\\
&\quad+\textstyle\left(\frac{1}{4}+\frac{1}{6\lambda}\right)\big((w_1^-)^2+(w_2^-)^2+w_1^-w_2^-\big)\\
&\quad-\textstyle 
 \frac{1}{2}\big(w_1^+w_1^- +w_2^+w_2^-\big)
 -\frac{1}{4}\big(w_1^+w_2^- +w_2^+w_1^-\big)\\
 &\quad+\textstyle
 \frac{3}{4}t_1^2 +\frac{3}{2}\delta u-\frac{3}{4}\delta^2
 \end{aligned}
\end{equation*}
in $t_1$, and $w_1^\pm,w_2^\pm,u$, which determine $\pr(R)=R(\iota_5(w_1^+,w_2^+,w_1^-,w_2^-,u))\in\epinched$. Thus, analogously to \eqref{eq:Qbound}, Propositions \ref{prop:projd} and \ref{prop:einsteinonethorpe} imply that for all $\lambda>0$,
\begin{equation}\label{eq:Qbound-lambda}
\min_{R\in\dpinched} \, I_{\lambda}(R) \; \geq \; \min_{x\in\esimp^6} \, Q_{\lambda}(x),
\end{equation}
where $\esimp^6=\conv(q_1,\dots,q_7)$ is the augmented Einstein simplex in Proposition~\ref{prop:einsteinonethorpe}.

Once again, we apply the optimization method in Section~\ref{subsec:optsimplex} to maximize $-Q_\lambda$ on $\esimp^6$. 
The first step is to determine the collection $\mathcal S_\lambda$ of subsets $S\subset\{q_1,\dots,q_7\}$ on whose affine hull $\aff(S)$ the restriction of $Q_\lambda$ is positive-definite, see \eqref{eq:verticesnegdef}. As observed in the proof of Theorem~\ref{thm:estimates-baby}, since the coordinates of $q_j$ are scalar multiples of $(1-\delta)$, except for the $u$-coordinate, on which $Q_\lambda$ has no degree $2$ term, the eigenvalues of 
$\operatorname{Hess}\!\big(Q_{\lambda}|_{\aff(S)}\big)$ are scalar multiples of $(1-\delta)^2$ for any subset $S\subset\{q_1,\dots,q_7\}$. Thus, the restriction of $Q_\lambda\colon\R^6\to\R$ to $\aff(S)\subset\R^6$ is either positive-definite for every $0<\delta<1$, or for no $0<\delta<1$ at all. 
However, for a fixed $S$, the restriction $Q_\lambda|_{\aff(S)}$ may be positive-definite for some values of $\lambda>0$, and indefinite or negative-definite for other values of~$\lambda>0$. Thus, even though the collection $\mathcal S_\lambda$ is independent of $\delta$, it \emph{does} depend on $\lambda$.

A simple computation shows that $\operatorname{Hess} Q_{\lambda}$ has the following eigenvalues:
\begin{equation*}
\textstyle
\frac{3\lambda+\sqrt{9\lambda^2+4}}{12\lambda}, \; \frac{3\lambda+\sqrt{9\lambda^2+4}}{4\lambda},\; \frac{3}{2}, \; \frac{3\lambda-\sqrt{9\lambda^2+4}}{4\lambda},\; \frac{3\lambda-\sqrt{9\lambda^2+4}}{12\lambda}, \; 0,
\end{equation*}
so, for all $\lambda>0$, there are exactly $d=3$ positive eigenvalues. Thus, by Corollary~\ref{cor:smallsig}, it suffices to inspect faces of $\esimp^6$ that have dimension $\leq 3$.

All singletons $S=\{q_j\}$, $1\leq j\leq 7$, i.e., $0$-dimensions faces, trivially belong to $\mathcal S_\lambda$, for all $\lambda>0$.
Regarding $1$-dimensional faces, all of the $21= \binom{7}{2}$ subsets % $S=\{q_{j_1},q_{j_2}\}$ 
of 2 vertices belong to $\mathcal S_\lambda$ for large enough $\lambda>0$.
For instance, the Hessian $1\times 1$-matrix of the restriction of $Q_\lambda$ to $\aff(q_1,q_2)$ is $\frac{15 \lambda -8}{18 \lambda }(1-\delta)^2$, which is positive if and only if $\lambda > \frac{8}{15}$. In general, $S=\{q_{j_1} ,q_{j_2} \}\in\mathcal S_\lambda$ if and only if $\lambda>\lambda_{j_1,j_2}$, where % $\lambda_{j_1,j_2}$ are given by:
\begin{equation*}
\textstyle
\lambda_{1,2}=\lambda_{1,7}=\lambda_{2,6}=\frac{8}{15},\;
\lambda_{1,6}=\lambda_{2,7}=\frac{1}{3},\;
\lambda_{1,5}=\frac{2}{15}, \;
\lambda_{2,5}=\frac{8}{87},
\end{equation*}
and $\lambda_{j_1,j_2}=0$ for all other $1\leq j_1<j_2\leq 7$.
Regarding $2$-dimensional faces, $32$ of the $35= \binom{7}{3}$ subsets %$S=\{q_{j_1},q_{j_2},q_{j_3}\}$
of 3 vertices belong to $\mathcal S_\lambda$ for large enough $\lambda>0$. 
Namely, $\{q_1,q_3,q_7\}$, $\{q_2,q_4,q_7\}$, and $\{q_5,q_6,q_7\}$ do not belong to $\mathcal S_\lambda$ for any $\lambda>0$, and the remaining $S=\{q_{j_1},q_{j_2},q_{j_3}\}$ belong to $\mathcal S_\lambda$ if and only if $\lambda>\lambda_{j_1,j_2,j_3}$, where  %$\lambda_{j_1,j_2,j_3}$ are given by
\begin{align*} 
&\begin{array}{llll}
\textstyle
\lambda_{1,2,3}=\frac{9+\sqrt{105}}{36}, &
\lambda_{1,3,4}=\frac{-9+\sqrt{105}}{36}, &
\lambda_{1,2,4}=\frac{6+\sqrt{42}}{18}, &
\lambda_{2,3,4}=\frac{-6+\sqrt{42}}{18}, \\[3pt]
\textstyle
\lambda_{1,3,5}=\frac{-1+\sqrt{5}}{9}, &
\lambda_{1,3,6}=\frac{1+\sqrt{5}}{9}, &
\lambda_{1,4,5}=\frac{-15+\sqrt{609}}{72}, &
\lambda_{2,3,6}=\frac{15+\sqrt{609}}{72}, \\[3pt]
\textstyle
\lambda_{1,4,6}=\frac{21+\sqrt{1113}}{126}, &
\lambda_{2,3,5}=\frac{-21+\sqrt{1113}}{126}, &
\lambda_{1,4,7}=\frac{3+\sqrt{105}}{18}, &
\lambda_{2,3,7}=\frac{-3+\sqrt{105}}{18}, \\[3pt]
\textstyle
\lambda_{2,4,5}=\frac{-2+2\sqrt{2}}{9}, &
\lambda_{2,4,6}=\frac{2+2\sqrt{2}}{9}, & &
\end{array}
\\[2pt]
&\textstyle\;\;
\lambda_{1,2,5}=\lambda_{1,2,6}=\lambda_{1,2,7}=\lambda_{1,5,6}=\lambda_{1,5,7}=\lambda_{1,6,7}=\lambda_{2,5,6}=\lambda_{2,5,7}=\lambda_{2,6,7}=\frac{2}{3},  
\end{align*}
and $\lambda_{j_1,j_2,j_3}=0$ for all other $1\leq j_1<j_2<j_3\leq 7$. 
For instance, the eigenvalues of the Hessian $2\times 2$-matrix of $Q_\lambda|_{\aff(q_1,q_2,q_3)}$ are $\frac{75 \lambda -8 \pm \sqrt{2169 \lambda ^2+528 \lambda +128}}{36 \lambda } (1-\delta)^2$, which are positive if and only if $\lambda>\lambda_{1,2,3}=\frac{9+\sqrt{105}}{36}$.
%9
Finally, regarding $3$-dimensional faces, $20$ of the $35= \binom{7}{4}$ subsets of 4 vertices belong to $\mathcal S_\lambda$ for large enough $\lambda>0$.
Namely, $S=\{q_{j_1},q_{j_2},q_{j_3},q_{j_4}\}$ belongs to $\mathcal S_\lambda$ if and only if $\lambda>\lambda_{j_1,j_2,j_3,j_4}$, where 
\begin{align*} 
&\textstyle\lambda_{1,2,3,6}= \lambda_{1,2,5,6}= \lambda_{1,2,5,7}=\lambda_{1,2,6,7}=\frac{2}{3}, \; \lambda_{1,3,4,5}=\frac{-1+\sqrt5}{9} \\
&\textstyle
\lambda_{1,4,5,6}=\lambda_{1,4,5,7}=\lambda_{1,4,6,7}=\lambda_{2,3,5,6}=\lambda_{2,3,5,7}=\lambda_{2,3,6,7}=\frac{4}{3\sqrt3},\\
&\lambda_{1,2,3,5}\cong 0.818\; \text{ is the largest real root of }
243 \lambda^3-108 \lambda^2-84 \lambda+8, \\
&\lambda_{1,2,3,7}\cong 0.701\; \text{ is the largest real root of }
243 \lambda^3-54 \lambda^2-93 \lambda+8, \\
&\lambda_{1,2,4,6}\cong 0.795\; \text{ is the largest real root of }
243 \lambda^3-270 \lambda^2+51 \lambda+8, \\
&\lambda_{1,2,4,6}\cong 0.461\; \text{ is the largest real root of }
 243 \lambda^3+108 \lambda^2-84 \lambda-8, \\
&\lambda_{2,3,4,5}\cong 0.0996\! \text{ is the largest real root of }
243 \lambda^3+270 \lambda^2+51 \lambda-8, \\
&\lambda_{2,3,4,6}\cong 0.562\; \text{ is the largest real root of }
243 \lambda^3+54 \lambda^2-93 \lambda-8, \\
&\lambda_{3,4,5,6}=\lambda_{3,4,5,7}=\lambda_{3,4,6,7}=0,
\end{align*}
and the remaining 15 subsets do not belong to $\mathcal S_\lambda$ for any $\lambda>0$.

The second step is to compute the critical point $x_S\in\aff(S)$ of $Q_\lambda|_{\aff(S)}$ for each of the above subsets $S\in \mathcal S_\lambda$, and determine the values of $0<\delta<1$ and $\lambda>0$ such that $x_S$ is in the relative interior of the face $\conv(S)$. 
These subsets $S$ define a subcollection $\mathcal S'_\lambda$ of $\mathcal S_\lambda$, which depend on both $\delta$ and $\lambda$, such that (cf.~\eqref{eq:max-max})
\begin{equation}\label{eq:min-min}
\min_{\esimp^6} Q_\lambda=\min_{S\in\mathcal S'_\lambda}  Q_\lambda(x_S).
\end{equation}

All singletons $S=\{q_j\}$ belong to $\mathcal S'_\lambda$ for all $0<\delta<1$ and $\lambda>0$, and $x_S=q_j$. The value $Q_\lambda(x_S)=Q_\lambda(q_j)$ for each of these points is listed below in Table~\ref{tab:0faces-lambda}.
\begin{table}[!ht]
\begin{tabular}{|c|l|}
\hline
$S$ & $Q_{\lambda}(x_S)\phantom{\Big|}$  \\[2pt]
\hline
$\{q_1\}$ & $\frac{2 (3 \lambda -1)}{9 \lambda }\delta ^2-\frac{3 \lambda -4}{9 \lambda }\delta+\frac{15 \lambda -8}{36 \lambda } \phantom{\Big|}$  \\[2pt]
$\{q_2\}$ & $\frac{15 \lambda -8}{36 \lambda }\delta ^2 -\frac{3 \lambda -4}{9 \lambda }\delta+\frac{24 \lambda -8}{36 \lambda }$ \\[2pt]
$\{q_3\}$ & $\frac{2 (3 \lambda +1)}{9 \lambda }\delta ^2 -\frac{3 \lambda +4}{9 \lambda }\delta+\frac{15 \lambda +8}{36 \lambda }$  \\[2pt]
$\{q_4\}$ & $\frac{15 \lambda +8}{36 \lambda }\delta ^2-\frac{3 \lambda +4}{9 \lambda }\delta+\frac{24 \lambda +8}{36 \lambda }$  \\[2pt]
\hline
\end{tabular}
\begin{tabular}{|c|l|}
\hline
$S$ & $Q_{\lambda}(x_S)\phantom{\Big|}$  \\[2pt]
\hline
$\{q_5\}$ & $\frac{3}{4}\phantom{\Big|}$ \\[5pt]
$\{q_6\}$ & $\frac{3}{4}$ \\[5pt]
$\{q_7\}$ & $\frac{3\delta^2}{4}$ \\[11pt]
\hline
\end{tabular}
\caption{Values of $Q_{\lambda}$ on the $0$-dimensional faces of $\esimp^6$.}\label{tab:0faces-lambda}
\end{table}

There are $15$ of the $21$ subsets $S\in\mathcal S_\lambda$ with $2$ vertices for which the critical point $x_S\in\aff(S)$ of $Q_\lambda|_{\aff(S)}$ lies in the relative interior of $\conv(S)$ for some value of $0<\delta<1$ and $\lambda>0$, as listed in Table~\ref{tab:1faces-lambda}.
Similarly, there are $15$ of the $32$ subsets $S\in\mathcal S_\lambda$ with $3$ vertices, 
and $4$ of the $20$ subsets $S\in\mathcal S_\lambda$ with $4$ vertices 
for which that happens.
\begin{table}[!ht]
\resizebox{\columnwidth}{!}{
\begin{tabular}{|c|l|l|}
\hline
$S$ & $Q_{\lambda}(x_S)\phantom{\Big|}$ & Conditions for $S\in\mathcal S'_\lambda$ \\[2pt]
\hline
$\{q_1,q_2\}$ & $\frac{18 (\lambda -1) \lambda +4}{3 \lambda  (15 \lambda -8)}\delta ^2 +\frac{ (9 (\lambda -2) \lambda +8)}{3 (8-15 \lambda ) \lambda }\delta  +\frac{18 (\lambda -1) \lambda +4}{3 \lambda  (15 \lambda -8)}$\rule[4pt]{0pt}{8pt}
& $\delta<\frac{1}{4}$, \; $\lambda>\frac{4(1-\delta)}{3-12\delta}$ \\[4pt]
\hline
$\{q_1,q_3\}$ & $\left(\frac{1}{4}-\frac{4}{135 \lambda ^2}\right)\delta ^2+\left(\frac{8}{135 \lambda ^2}+\frac{1}{2}\right)\delta-\frac{4}{135 \lambda ^2}$ & $\lambda>\frac{4}{15}$\rule[4pt]{0pt}{8pt} \\[2pt]
\hline
$\{q_1,q_4\}$ &$\frac{18 \lambda  (3 \lambda +1)-16}{567 \lambda ^2}\delta ^2 +\left(\frac{32}{567 \lambda ^2}+\frac{11}{21}\right)\delta+\frac{18 \lambda  (3 \lambda -1)-16}{567 \lambda ^2} $ & $\delta<\frac{3}{4}$, \; $\lambda>\frac{8(1-\delta)}{27-36\delta}$ \rule[4pt]{0pt}{8pt} \\[2pt]
\hline
$\{q_1,q_5\}$ & $\frac{9 \lambda }{8-60 \lambda }\delta ^2 +\frac{9 \lambda }{15 \lambda -2}\delta +\frac{6-9 \lambda }{8-60 \lambda }$ & $\delta<\frac{4}{7}$, \; $\lambda>\frac{4(1-\delta)}{12-21\delta}$ \rule[4pt]{0pt}{8pt} \\[2pt]
\hline
$\{q_1,q_6\}$ & $\frac{6-9\lambda}{8-24\lambda}$ & $\delta<\frac{1}{4}$, \; $\lambda>\frac{4(1-\delta)}{3-12\delta}$ \rule[4pt]{0pt}{8pt} \\[2pt]
\hline
$\{q_2,q_3\}$ & $\frac{18 \lambda  (3 \lambda -1)-16}{567 \lambda ^2}\delta ^2 +\left(\frac{32}{567 \lambda ^2}+\frac{11}{21}\right)\delta+\frac{18 \lambda  (3 \lambda +1)-16}{567 \lambda ^2}$\rule[4pt]{0pt}{8pt} &
\!\!\!$\begin{array}{l}
\delta\leq\frac{3}{4}, \;\;\, \lambda>\frac{8(1-\delta)}{36-27\delta}\rule[4pt]{0pt}{8pt} \\[2pt]
\delta>\frac{3}{4}, \;\, \frac{8(1-\delta)}{36-27\delta}<\lambda<\frac{8(1-\delta)}{36\delta-27} \rule[4pt]{0pt}{8pt}
\end{array}$ \\[14pt]
\hline
$\{q_2,q_4\}$ & $-\left(\frac{1}{54 \lambda ^2}+\frac{1}{4}\right)\delta ^2 + \left(\frac{1}{27 \lambda ^2}+1\right)\delta-\frac{1}{54 \lambda ^2}$ & $\lambda>\frac{1}{6}$\rule[4pt]{0pt}{8pt} \\[2pt]
\hline
$\{q_2,q_5\}$ & $\frac{36 \lambda }{8-87 \lambda }\delta ^2 +\frac{90 \lambda }{87 \lambda -8}\delta+\frac{6-9 \lambda }{8-87 \lambda }$ & $\delta<\frac{14}{17}$,\, $\lambda>\frac{8(1-\delta)}{42-51\delta}$ \rule[4pt]{0pt}{8pt}\\[2pt]
\hline
$\{q_2,q_6\}$ & $\frac{6-9\lambda}{8-15\lambda}$ & $\delta<\frac{2}{5}$, \; $\lambda>\frac{8(1-\delta)}{6-15\delta}$ \rule[4pt]{0pt}{8pt} \\[2pt]
\hline
$\{q_3,q_4\}$ &
$\frac{18 \lambda (\lambda +1)+4}{3 \lambda  (15 \lambda +8)}\delta ^2+\frac{1}{15}\!\left(\frac{9}{15 \lambda +8}-\frac{5}{\lambda }-3\right)\!\delta+\frac{18 \lambda  (\lambda +1)+4}{3 \lambda  (15 \lambda +8)}$ & \!\!\!$\begin{array}{l}
\delta\leq\frac{1}{4} \rule[4pt]{0pt}{8pt} \\[2pt]
\delta>\frac{1}{4}, \;\, \lambda<\frac{4(1-\delta)}{12\delta-3} \rule[4pt]{0pt}{8pt}
\end{array}$ \\[14pt]
\hline
$\{q_3,q_5\}$ & $\frac{6+9\lambda}{8+24\lambda}$ &
\!\!\!$\begin{array}{l}
\delta\leq\frac{1}{4} \rule[4pt]{0pt}{8pt} \\[2pt]
\delta>\frac{1}{4}, \;\, \lambda<\frac{4(1-\delta)}{12\delta-3} \rule[4pt]{0pt}{8pt}
\end{array}$ \\[14pt]
\hline
$\{q_3,q_6\}$ & $-\frac{9 \lambda }{60 \lambda +8} \delta ^2+\frac{9 \lambda }{15 \lambda +2}\delta+\frac{9 \lambda +6}{60 \lambda +8}$ & 
\!\!\!$\begin{array}{l}
\delta\leq \frac{4}{7} \rule[4pt]{0pt}{8pt} \\[2pt]
\delta>\frac{4}{7}, \;\, \lambda<\frac{4(1-\delta)}{21\delta-12} \rule[4pt]{0pt}{8pt}
\end{array}$ \\[12pt]
\hline
$\{q_4,q_5\}$ & $\frac{6+9\lambda}{8+15\lambda}$ & \!\!\!$\begin{array}{l}
\delta\leq \frac{2}{5} \rule[4pt]{0pt}{8pt} \\[2pt]
\delta>\frac{2}{5}, \;\, \lambda<\frac{8(1-\delta)}{15\delta-6} \rule[4pt]{0pt}{8pt}
\end{array}$ \\[12pt]
\hline
$\{q_4,q_6\}$ & $-\frac{36\lambda }{87 \lambda +8}\delta ^2 +\frac{90 \lambda }{87 \lambda +8}\delta +\frac{9 \lambda +6}{87 \lambda +8}$ & 
\!\!\!$\begin{array}{l}
\delta\leq \frac{14}{17} \rule[4pt]{0pt}{8pt} \\[2pt]
\delta>\frac{14}{17}, \;\, \lambda<\frac{8(1-\delta)}{51\delta-42} \rule[4pt]{0pt}{8pt}
\end{array}$ \\[12pt]
\hline
$\{q_5,q_6\}$ & $-\frac{3}{4}\delta ^2+\frac{3}{2}\delta $\rule[4pt]{0pt}{8pt} & $0<\delta<1$, $\lambda>0$ \\[2pt]
\hline
\end{tabular}
}
\caption{Minimum of $Q_{\lambda}|_{\aff(S)}$, attained at $x_S\in\aff(S)$, with necessary and sufficient conditions on $\delta$ and $\lambda$ for $Q_\lambda|_{\aff(S)}$ to be positive-definite and $x_S\in\operatorname{relint}(\conv(S))$, 
i.e., for $S\in\mathcal S'_\lambda$.}\label{tab:1faces-lambda}
\end{table}
For instance, $S=\{q_1,q_2,q_5\}\in\mathcal S'_\lambda$ if and only if $\delta<\frac{1}{3}$ and $\lambda>\frac{2(1-\delta)}{3(1-3\delta)}$, in which case $Q_\lambda|_{\aff(q_1,q_2,q_5)}$ has minimum 
\begin{equation*}
Q_\lambda\big(x_{\{q_1,q_2,q_5\}}\big)=\textstyle
\frac{9 \lambda  (8-21 \lambda )}{8 \left(36 \lambda ^2-27 \lambda +2\right)}\delta ^2 +\frac{45  \lambda }{48 \lambda -4}\delta  +\frac{6-9 \lambda }{8-96 \lambda},
\end{equation*}
while $S=\{q_2,q_3,q_5,q_6\}$ belongs to $\mathcal S'_\lambda$ if and only if $\delta<\frac{1}{6}$ and $\lambda > \frac{\delta+\sqrt{289\delta^2-336\delta+48} }{9(1-6\delta)}$, in which case $Q_\lambda|_{\aff(q_2,q_3,q_5,q_6)}$ has minimum
\begin{equation*}
\begin{aligned}
Q_\lambda\big(x_{\{q_2,q_3,q_5,q_6\}}\big) &=\textstyle
\frac{3 \left(3321 \lambda ^4+270 \lambda ^3-1287 \lambda ^2-144 \lambda +64\right)}{4 \left(16-27 \lambda ^2\right)^2}\delta ^2+\frac{3 \left(36 \lambda ^2-3 \lambda +8\right)}{64-108 \lambda ^2}\delta +\frac{3}{16}.
\end{aligned}
\end{equation*}
The remaining values $Q_\lambda(x_S)$ are omitted to simplify the exposition, but the reader may find them in the particular case $\lambda=\frac{1}{2}$ in Table~\ref{tab:2faces} and \eqref{eq:3faceValue}.

Altogether, there are 41 subsets $S$ that belong to the collection $\mathcal S'_\lambda$ for some $(\delta,\lambda)\in H$, where $H=(0,1)\times (0,+\infty)$ is a vertical strip in $\R^2$. The corresponding sentences ``if $S\in\mathcal S'_\lambda$, then $Q_\lambda(x_S)\geq0$'' give a description of the semialgebraic set
\begin{equation*}
\mathfrak X:=\left\{(\delta,\lambda)\in H : \min_{S\in\mathcal S'_\lambda}  Q_\lambda(x_S) \geq 0 \right\}
\end{equation*}
involving (finitely many) polynomial inequalities in $(\delta,\lambda)$, connected by ``and'' and ``or''. Using Cylindrical Algebraic Decomposition, see e.g.~\cite[Sec~5.1]{basu},
any semialgebraic set in $\R^2$ can be written as a finite disjoint union of $2$-\emph{cells}, i.e., points, vertical open intervals, graphs of the form $\big\{(\delta,\lambda)\in \R^2 : a<\delta<b, \, \lambda=\varphi(\delta) \big\}$, and bands of the form $\big\{(\delta,\lambda)\in \R^2 : a<\delta<b, \, \varphi(\delta)<\lambda<\psi(\delta) \big\}$, where $a,b\in\R$, and $\varphi,\psi\colon (a,b)\to [-\infty,+\infty]$ are continuous semialgebraic functions. (The latter are similar to what Calculus textbooks often call \emph{regions of type I} in integration of functions of two variables.)
Applied to the semialgebraic set $\mathfrak X$ in the $(\delta,\lambda)$-plane, cylindrical algebraic decomposition~yields:
\begin{equation*}
\begin{aligned}
\mathfrak X&=\left\{\delta\in(0,\delta_1^*], \, \lambda\geq\tfrac{\sqrt{\frac{24}{\delta}+8 -8 \delta+\delta^2}+\delta-4 }{6 (3-\delta)} \right\}\cup\left\{\delta\in[\delta_1^*,\delta_2^*], \, \lambda\geq \tfrac{4 }{3 \sqrt{15}} \tfrac{1-\delta}{\sqrt{\delta (\delta+2)}} \right\} \\
&\qquad \cup\left\{\delta\in[\delta_2^*,1), \, \lambda\geq \tfrac{8(1-\delta)^2}{24 \delta^2-12 \delta+15} \right\},
\end{aligned}
\end{equation*}
i.e., $\mathfrak X= \big\{\delta\in (0,1), \, \lambda\geq \lambda^*(\delta)\big\}$, where $\lambda^*\colon (0,1)\to\R$ is the piecewise continuous function defined in \eqref{eq:lambda*}.
From \eqref{eq:Qbound-lambda} and \eqref{eq:min-min}, we have that $(\delta,\lambda)\in \mathfrak X$ implies $\min\limits_{R\in\dpinched} I_\lambda(R)\geq0$, so \eqref{eq:goal-lambda} holds, concluding the proof.
\end{proof}

\begin{proof}[Proof of Theorem~\ref{mainthm:Lambda}]
Consider the functions $\lambda^*\colon(0,1]\to\R$ and $\lambda^\V \colon \left[\delta_0^\V,1\right]\to\R$, defined in \eqref{eq:lambda*} and Theorem~\ref{thm:ville}, respectively. Define $\lambda\colon (0,1]\to\R$ as follows
\begin{equation*}
\lambda(\delta)=\begin{cases}
\lambda^*(\delta), & \text{ if } \delta\in\left(0,\delta_0^\V\right], \\
\min\!\left\{\lambda^*(\delta),\lambda^\V(\delta)\right\}, & \text{ if } \delta\in\left[\delta_0^\V,1\right].
\end{cases}
\end{equation*}
Routine computations show that $\lambda$ agrees with \eqref{eq:bestlambda}, and satisfies $\lim\limits_{\delta\searrow0}\lambda(\delta)=+\infty$, $\lambda\!\left(\tfrac{1}{1+3\sqrt3}\right)<\tfrac12$, $\lambda\!\left(\tfrac14\right)=\tfrac13$, and $\lambda(1)=0$. From Theorems~\ref{thm:estimates-general} and \ref{thm:ville}, a $\delta$-pinched oriented $4$-manifold $(M^4,\g)$ with finite volume satisfies $|\sigma(M)|\leq \lambda(\delta)\,\chi(M)$.
\end{proof}

\begin{remark}
Polombo~\cite[Thm II.13]{polombo} proved a similar \emph{explicit} inequality for $\delta$-pinched $4$-manifolds, with $0<\delta\leq \frac{1}{4}$; namely $|\sigma(M)|\leq \frac{2}{27}\left(\frac{2}{\delta^2}-\frac{7}{\delta}+5\right) \, \chi(M)$. It is straightforward to check that $\lambda(\delta)<\frac{2}{27}\left(\frac{2}{\delta^2}-\frac{7}{\delta}+5\right)$ for all $0<\delta\leq \frac{1}{4}$.
\end{remark}

\section{Upper bounds}\label{sec:upper}

In this section, we discuss further applications of the optimization methods from Section~\ref{sec:optimization}, proving \emph{upper bounds} for $\chi(M)$ and $\sigma(M)$ if $M$ is a $\delta$-pinched oriented $4$-manifold with finite volume.

\subsection{Weyl tensor}
A key step towards the above goal is to establish an upper bound on $|W_\pm|^2$ for a $\delta$-pinched curvature operator $R$. By Proposition~\ref{prop:projd}, no generality is lost if we assume that $R$ is Einstein. This pointwise problem has received great attention in the literature, see e.g.~\cite[Lemma 4.1]{yang-einstein}, \cite[Lemma 1]{gursky-lebrun}, and \cite[Lemma 3.1]{cao-tran}. The following result provides a useful \emph{sharp} upper bound for any linear combinations of $|W_\pm|^2$ when $R$ is a $\delta$-pinched curvature operator.

\begin{proposition}\label{prop:weylbound}
For all $-1\leq \eta \leq 1$ and $0<\delta\leq 1$, if $R$ is $\delta$-pinched, then 
\begin{equation}\label{eq:weylbound}
|W_+|^2+\eta\, |W_-|^2\leq \tfrac{8}{3}(1-\delta)^2.
\end{equation}
For $\eta\neq 1$, equality in \eqref{eq:weylbound} holds if and only if $\pr(\pm R)=t \, \iota_5(p_1)+(1-t)\,\iota_5(p_2)$, $t\in [0,1]$, and, for $\eta=1$, if and only if $\pr(\pm R)$ is in the convex hull of $\iota_5(p_j)$, $1\leq j\leq 4$, where $\pm R\in\dpinched$, using the notation in \eqref{eq:dpinched}, \eqref{eq:pr}, and Proposition~\ref{prop:einsteinsimplex}.
\end{proposition}

\begin{proof}
For all $-1\leq \eta\leq1$, given $R$ in the canonical form \eqref{eq:curvop4}, by \eqref{eq:ww},
\begin{equation*}
|W_+|^2+\eta\, |W_-|^2  = Q_\eta(w_1^+,w_2^+,w_1^-,w_2^-,u),
\end{equation*}
where $Q_\eta\colon \R^5\to\R$ is the quadratic polynomial
\begin{equation*}
\begin{aligned}
Q_\eta(w_1^+,w_2^+,w_1^-,w_2^-,u) &:=
\textstyle 2\big((w_1^+)^2+(w_2^+)^2+w_1^+w_2^+\big)\\
&\quad+\textstyle 2\eta \big((w_1^-)^2+(w_2^-)^2+w_1^-w_2^-\big).
\end{aligned}
\end{equation*}

Given $R\in\dpinched$, let $\pr(R)=R(\vec w_+,\vec w_-,u)\in \epinched$, see Lemma~\ref{lem:projecteinstein}, \eqref{eq:pr}, and \eqref{eq:Rww}. By Proposition~\ref{prop:einsteinsimplex}, we have that $(w_1^+,w_2^+,w_1^-,w_2^-,u) \in \esimp^5=\conv(p_1,\dots,p_6)$. 
By a straightforward computation, the eigenvalues of $\operatorname{Hess} Q_\eta$ are 
\begin{equation*}
1,\;  3,\;  \eta,\;  3\eta,\;  0,
\end{equation*}
so the number of negative eigenvalues is either $0$ or $2$, according to whether $\eta\geq 0$ or $\eta<0$, respectively.
By Corollary~\ref{cor:smallsig}, this means $\max_{x\in\esimp^5} Q_\eta(x)$ is achieved at some vertex $p_j$ if $\eta\geq0$, while we must inspect faces of dimension $\leq 2$, i.e., convex combinations of up to $3$ vertices $p_j$'s, if $\eta<0$.
The values assumed by $Q_\eta$ at $p_j$ are:
\begin{align}\label{eq:vertexvaluesQeta}
Q_\eta(p_1)&=Q_\eta(p_2)=\tfrac{8}{3}(1-\delta)^2, \nonumber \\[2pt]
Q_\eta(p_3)&=Q_\eta(p_4)=\tfrac{8\eta}{3}(1-\delta)^2,\\[2pt]
Q_\eta(p_5)&=Q_\eta(p_6)=0,\nonumber 
\end{align}
so the inequality $Q_\eta(x)\leq \frac{8}{3}(1-\delta)^2$ clearly holds for all $x\in\esimp^5$ and $0\leq \eta\leq 1$, with equality achieved if and only if $x\in\conv(p_j)$, where $1\leq j\leq 2$ if $0\leq \eta<1$, and $1\leq j\leq 4$ if $\eta=1$. Thus, let us now assume $-1\leq \eta<0$, and follow the optimization procedure in Section~\ref{subsec:optsimplex}.

The first step is to determine the collection $\mathcal S$ of subsets $S\subset\{p_1,\dots,p_6\}$ such that the restriction of $Q_\eta$ to the affine hull $\aff(S)$ is negative-definite, see \eqref{eq:verticesnegdef}. All singletons $S=\{p_j\}$, $1\leq j\leq 6$ trivially belong to $\mathcal S$. Regarding $1$-dimensional faces, we find that $Q_\eta|_{\aff(S)}$ is negative-definite if and only if $S=\{p_{j_1},p_{j_2}\}$ is one of the following:
\begin{equation}\label{eq:1facesQeta}
\{p_3,p_4\}, \; \{p_3,p_5\}, \; \{p_3,p_6\}, \; \{p_4,p_5\}, \; \{p_4,p_6\}.
\end{equation}
In all cases, the single entry of $\operatorname{Hess}\!\big(Q_\eta|_{\aff(S)}\big)$ is $\tfrac{8\eta}{3}(1-\delta)^2$. Regarding $2$-dimensional faces, % the $2\times 2$-matrix
$Q_\eta|_{\aff(S)}$ is negative-definite if and only if $S=\{p_3,p_4,p_5\}$ or $S=\{p_3,p_4,p_6\}$. In both cases, its eigenvalues of $\operatorname{Hess}\!\big(Q_\eta|_{\aff(S)}\big)$ are $4\eta (1-\delta)^2$ and $\frac{4\eta}{3}(1-\delta)^2$.

The second step is to extract the subcollection $\mathcal S' \subset \mathcal S$ such that the unique critical point $x_S\in\aff(S)$ of $Q_\eta|_{\aff(S)}$ is in the relative interior of the face $\conv(S)$, see  \eqref{eq:verticesnegdefint}.
Every singleton $S=\{p_j\}$ is trivially in $\mathcal S'$ and has $x_S=p_j$; recall that the values of $Q_\eta(p_j)$ are given in \eqref{eq:vertexvaluesQeta}. Among the subsets \eqref{eq:1facesQeta}, only $\{p_3,p_4\}\in\mathcal S'$, since $x_{\{p_3,p_4\}}=\frac{1}{2}p_3+\frac{1}{2}p_4\in\operatorname{relint}(\conv(p_{3},p_{4}))$, and $Q_\eta$ assumes the value
\begin{equation}\label{eq:Q34}
Q_\eta\!\big(x_{\{p_3,p_4\}}\big) = 2\eta (1-\delta)^2.
\end{equation}
In all other $S=\{p_{j_1},p_{j_2}\}$, $j_1<j_2$, one has $x_{\{p_{j_1},p_{j_2}\}}=p_{j_2}\notin\operatorname{relint}(\conv(p_{j_1},p_{j_2}))$.
Similarly, neither $\{p_3,p_4,p_5\}$ nor $\{p_3,p_4,p_6\}$ belong to $\mathcal S'$, since a direct computation shows that $x_{\{p_3,p_4,p_5\}}=p_5$ and $x_{\{p_3,p_4,p_6\}}=p_6$. Thus, by Corollary~\ref{cor:smallsig}, see also \eqref{eq:max-max}, 
$\max_{x\in\esimp^5} Q_\eta(x)= Q_\eta(p_1)= Q_\eta(p_2)=\frac{8}{3}(1-\delta)^2$ is the largest value among \eqref{eq:vertexvaluesQeta} and \eqref{eq:Q34}. This concludes the proof of \eqref{eq:weylbound} and its equality case.
\end{proof}

\begin{remark}
Equality in \eqref{eq:weylbound} is achieved by $\delta$-pinched curvature operators whose projection onto the set of Einstein $\delta$-pinched curvature operators can be written as certain linear combinations of $R_{S^4}$, $R_{\C P^2}$, and $R_{\C H^2}$, see~Remark~\ref{rem:geominterp}.
\end{remark}

\begin{remark}
Proposition~\ref{prop:weylbound} is reminiscent of a bound obtained by Yang~\cite[Lemma 4.1(a)]{yang-einstein} for Einstein $4$-manifolds. Namely, replacing $\delta\leq\sec\leq 1$ with $\delta\leq \sec\leq \Delta$, inequality \eqref{eq:weylbound} with $\eta=1$ yields $|W|^2\leq \frac{8}{3}\left(\Delta-\delta\right)^2$. Furthermore, if $\Ric_R=\g$, then $\Delta \leq 1-2\delta$, hence $|W|^2\leq \frac{8}{3}(1-3\delta)^2$, cf.~ \cite[Equation (4.6)]{yang-einstein}.
\end{remark}

\subsection{Euler characteristic and signature}\label{sec:upper-bounds}
We now use the bounds in Proposition~\ref{prop:weylbound} and Comparison Geometry to prove Theorems~\ref{mainthm:bounds-pos} and \ref{mainthm:bounds-neg}, and Corollary~\ref{cor:homeotypes}. % in the Introduction.

\begin{proof}[Proof of Theorem \ref{mainthm:bounds-pos}]
Suppose $(M^4,\g)$ is positively $\delta$-pinched, so that $R\in\dpinched$ at all points, and not diffeomorphic to $S^4$.
We may then apply Lemmas~\ref{lemma:CGY} and \ref{lemma:vol}, and Proposition~\ref{prop:weylbound} with $\eta=1$, obtaining:
\begin{equation*}
\chi(M)\leq \frac{1}{4\pi^2} \int_M |W_+|^2 + |W_-|^2\,\vol_\g \leq \frac{1}{4\pi^2} \,\frac{8}{3}(1-\delta)^2 \, \frac{4\pi^2}{3\delta^2} = \frac{8}{9}\left(\frac{1}{\delta}-1\right)^2.
\end{equation*}
Up to reversing orientation, we assume without loss of generality that $\sigma(M)\geq0$. Using \eqref{eq:sigma} instead of Lemma~\ref{lemma:CGY}, and Proposition~\ref{prop:weylbound} with $\eta=-1$, we have:
\begin{equation*}
\sigma(M)= \frac{1}{12\pi^2}\int_M |W_+|^2 - |W_-|^2\,\vol_\g \leq \frac{1}{12\pi^2} \, \frac{8}{3}(1-\delta)^2 \, \frac{4\pi^2}{3\delta^2} = \frac{8}{27}\left(\frac{1}{\delta}-1\right)^2.\qedhere
\end{equation*}
\end{proof}

\begin{remark}
Given the generality afforded by $-1\leq \eta\leq 1$ in Proposition~\ref{prop:weylbound}, it is tempting to re-examine the proof of Theorem~\ref{mainthm:bounds-pos} with upper bounds of the form
\begin{equation}\label{eq:nothing}
\begin{aligned}
a\, \chi(M)+b\, |\sigma(M)| &\leq \frac{1}{4\pi^2}\left(a+\frac{b}{3}\right) \int_M |W_+|^2 + \eta\, |W_-|^2 \,\vol_\g\\
&\leq  \frac{8}{9} \left(a+\frac{b}{3}\right) \left(\frac{1}{\delta}-1\right)^2,
\end{aligned}
\end{equation}
where $a,b\geq0$ are not both zero, % $(a,b)\neq (0,0)$,
and $\eta=\frac{3a-b}{3a+b}$. However, all such bounds \eqref{eq:nothing} are directly implied by the extreme cases $(a,b)=(1,0)$ and $(a,b)=(0,1)$, which form the statement of Theorem~\ref{mainthm:bounds-pos}. Indeed, the intersection of all affine half-spaces \eqref{eq:nothing} in the $(|\sigma|,\chi)$-plane is precisely the rectangle $\chi\leq \frac{8}{9}\left(\frac{1}{\delta}-1\right)^2$ and $|\sigma|\leq \frac{8}{27}\left(\frac{1}{\delta}-1\right)^2$.
\end{remark}

\begin{remark}\label{rem:gromov-abresch}
By a celebrated result of Gromov~\cite{gromov-total}, closed $n$-manifolds with $\sec\geq0$ have bounded total Betti number $\sum_k b_k(M)\leq C(n)$. Thus, if $(M^4,\g)$ is a closed oriented $4$-manifold with $\sec>0$, then $\chi(M)\leq C(4)$, as
$b_0(M)=b_4(M)=1$ and $b_1(M)=b_3(M)=0$, hence $\chi(M)=2+b_2(M)=\textstyle\sum_k b_k(M)$. 
In particular, this also gives an upper bound $|\sigma(M)|\leq\chi(M)-2\leq C(4)-2$ by \eqref{eq:basic_geography}.

Although Gromov conjectured that $C(n)=2^n$, which would be sharp since the torus $T^n$ has $\sum_k b_k(T^n)=2^n$, the best known estimates for $C(n)$ grow exponentially in $n^3$, see Abresch~\cite{abreschII}.
Using~\cite[p.~477]{abreschII}, we have that the Poincar\'e polynomial $P_t(M)=1+b_2(M)t^2+t^4$ of 
 $(M^4,\g)$ satisfies $P_{t(4)^{-1}}(M)\leq e^{466}$, where $t(4)=5^{16}8^4e^{8/15}$. Thus, $b_2(M)\leq t(4)^2(e^{466}-1)+t(4)^{-2} \lesssim 2.731\times 10^{232}$, so also
\begin{equation}\label{eq:gromov-abresch}
\chi(M)\lesssim 2.731\times 10^{232}.
\end{equation} 
Therefore, the upper bound $\chi(M)\leq \frac{8}{9}\left(\frac{1}{\delta}-1 \right)^2$ 
in Theorem~\ref{mainthm:bounds-pos} is smaller than \eqref{eq:gromov-abresch} only if $\delta\gtrsim5.705\times 10^{-117}$. Nevertheless, it is \emph{hundreds} of orders of magnitude smaller than \eqref{eq:gromov-abresch} for larger $\delta$; e.g., it gives $\chi(M)\leq 10^2$ if $\delta\gtrsim0.086$. 
\end{remark}

\begin{proof}[Proof of Corollary~\ref{cor:homeotypes}]
Given $\delta>0$, combining Theorem~\ref{mainthm:bounds-pos} and Theorem~\ref{thm:4Dmanifolds}, it follows that an \emph{orientable} positively $\delta$-pinched $4$-manifold $(M^4,\g)$ is homeomorphic~to
\begin{enumerate}[\rm (i)]
\item $\#^r \C P^2\#^s \overline{\C P^2}$, $r+s+2\leq \frac{8}{9}(\frac{1}{\delta}-1)^2$, $|r-s|\leq \frac{8}{27}(\frac{1}{\delta}-1)^2$, if $M$ is non-spin; 
 \item $\#^r (S^2\times S^2)$, $2r+2\leq \frac{8}{9}(\frac{1}{\delta}-1)^2$ if $M$ is spin.
\end{enumerate}
Instead, a \emph{non-orientable} positively $\delta$-pinched $4$-manifold $(M^4,\g)$ has $\pi_1(M)\cong\Z_2$, by Synge's Theorem. Applying Theorem~\ref{mainthm:bounds-pos} to its double-cover $(\widetilde M,\widetilde\g)$, endowed with the pullback metric, we have $\chi(M)=\tfrac12\,\chi(\widetilde M)\leq \frac{4}{9}(\frac{1}{\delta}-1)^2$. 
According to \cite[Thm.~1]{hkt}, for each given value of $\chi(M)$, the homeomorphism type of such $M$ is completely determined by topological invariants that can only take finitely many different values. Moreover, an explicit list of closed non-orientable $4$-manifolds with $\pi_1(M)\cong\Z_2$ realizing these homeomorphism types is given in \cite[Thm.~3]{hkt}.
\end{proof}

\begin{proof}[Proof of Theorem~\ref{mainthm:bounds-neg}]
The statement about $\sigma(M)$ follows exactly as in the proof of Theorem~\ref{mainthm:bounds-pos}. Without loss of generality, assume $\sigma(M)\geq0$. Since $-R\in\dpinched$ at all points of $(M^4,\g)$ and $\underline{\sigma}(-R)=\underline{\sigma}(R)$, 
from \eqref{eq:sigma} and Proposition~\ref{prop:weylbound} with $\eta=-1$, 
\begin{equation*}
\sigma(M)= \frac{1}{12\pi^2}\int_M |W_+|^2 - |W_-|^2\,\vol_\g \leq % \frac{1}{12\pi^2} \, \frac{8}{3}(1-\delta)^2 \, \Vol(M,\g) =
\frac{2}{9\pi^2} (1-\delta)^2\Vol(M,\g). 
\end{equation*}

Regarding $\chi(M)$, due to the absence of a negatively pinched counterpart to Lemma~\ref{lemma:CGY}, we use the optimization methods of Section~\ref{subsec:optsimplex} directly on \eqref{eq:chi}.
In order to have a quadratic form defined on a simplex, 
given $R$ in the canonical form \eqref{eq:curvop4}, we discard the nonpositive term $-\tfrac14|C|^2$ in \eqref{eq:chi}, and consider the quantity
$\tfrac34 u^2+ \frac18 |W_+|^2 +\frac18 |W_-|^2 =Q(w_1^+,w_2^+,w_1^-,w_2^-,u)$, where $Q\colon \R^5\to\R$ is given by
\begin{equation*}
\begin{aligned}
Q(w_1^+,w_2^+,w_1^-,w_2^-,u) &:= \textstyle  \tfrac14 \big((w_1^+)^2+(w_2^+)^2+w_1^+w_2^+\big)\\
&\quad+\textstyle \tfrac14 \big((w_1^-)^2+(w_2^-)^2+w_1^-w_2^-\big) +\frac34 u^2. 
\end{aligned}
\end{equation*}

For $R\in\dpinched$, let $\pr(R)=R(\vec w_+,\vec w_-,u)\in \epinched$, see Lemma~\ref{lem:projecteinstein} and \eqref{eq:Rww}. By Proposition~\ref{prop:einsteinsimplex}, we have that $(w_1^+,w_2^+,w_1^-,w_2^-,u) \in \esimp^5=\conv(p_1,\dots,p_6)$. 
Since $Q$ is evidently positive-definite, Corollary~\ref{cor:smallsig} implies that its maximum is achieved at a vertex $p_j$, and is hence the largest among the following values:
\begin{equation}\label{eq:vertexvaluesQ}
\begin{array}{lcl}
Q(p_1)= Q(p_3)=\textstyle\tfrac{2}{3}\delta^2 -\frac{1}{3}\delta +\frac{5}{12}, & & Q(p_2)=Q(p_4)=\textstyle\tfrac{5}{12}\delta^2 -\frac{1}{3}\delta +\frac{2}{3}, \\[3pt]
Q(p_5)=\textstyle \frac{3}{4}, & & Q(p_6)=\textstyle\frac{3}{4}\delta^2.
\end{array}
\end{equation}
Therefore, we have 
\begin{equation}\label{eq:maxChi}
\max_{R\in\dpinched} \underline{\chi}(R) \leq \max_{x\in \esimp^5} Q(x) = \max_{1\leq j\leq 6} Q(p_j) = Q(p_5)=\tfrac34.
\end{equation}
Thus, as $-R\in\dpinched$ at all points of $(M^4,\g)$ and $\underline{\chi}(-R)=\underline{\chi}(R)$, from \eqref{eq:maxChi} and \eqref{eq:chi},
\begin{equation}\label{eq:eulerbound-neg}
\chi(M)=\frac{1}{\pi^2}\int_M \underline{\chi}(R)\,\vol_\g \leq \frac{3}{4\pi^2}\Vol(M,\g).
\end{equation}
Clearly, equality in \eqref{eq:eulerbound-neg} holds if and only if $\underline{\chi}(R)=\frac{3}{4}$ at all points of $(M^4,\g)$, which, by \eqref{eq:vertexvaluesQ} and \eqref{eq:maxChi} is equivalent to $-R=\iota_5(p_5)=R_{S^4}$, i.e., $\sec_M\equiv-1$.
\end{proof}

\begin{remark}
As stated in the Introduction, \eqref{eq:eulerbound-neg} was also observed by Ville~\cite{ville-vol}.
\end{remark}

\appendix

\section{Revisiting Ville's estimates}\label{app:ville}

The seminal works of Ville~\cite{ville-negative,ville-positive} on the geography of pinched $4$-manifolds has been partially extended by several authors, see e.g.~\cite{ko2,DRR}. In this Appendix, we give a uniform and general treatment of Ville's estimates, that pushes the method to its natural limit, yielding the following result.

\begin{theorem}\label{thm:ville}
If $(M^4,\g)$ is a $\delta$-pinched oriented $4$-manifold, with finite volume and $\delta\geq\delta_0^\V$, then
\begin{equation*}
|\sigma(M)|\leq \lambda^{\V}(\delta)\,\chi(M),    
\end{equation*}
where $\lambda^\V\colon [\delta_0^{\V},1]\to\R$ is given by
\begin{equation*}
\lambda^{\V}(\delta)=\begin{cases}
\dfrac{7\delta^2+10\delta+1-\sqrt3\sqrt{11\delta^4+68\delta^3+6\delta^2+28\delta-5}}{6(1-\delta)^2}, & \text{ if } \delta\in\left[\delta^\V_0,\delta^\V_1\right],\\[10pt]
\dfrac{2}{3}\dfrac{13\delta^2+4\delta+1-\sqrt3\sqrt{55\delta^4+40\delta^3+6\delta^2+8\delta-1}}{(1-\delta)^2}, & \text{ if } \delta\in\left[\delta_1^\V,\delta_2^\V\right],\\[5pt]
\dfrac{8(1-\delta)^2}{24\delta^2-12\delta+15}, & \text{ if } \delta\in\left[\delta_2^\V,1\right],
\end{cases}
\end{equation*}
and
\begin{enumerate}[\rm (i)]
\item $\delta_0^\V\cong0.163$ is the smallest real root of the polynomial $\delta^4-18\delta^3+2\delta^2-6\delta+1$,\smallskip
\item $\delta_1^\V\cong0.166$ is the only real root of the polynomial $31\delta^3+\delta^2+5\delta-1$,\smallskip
\item $\delta_2^\V\cong0.211$ is the largest real root of the polynomial $140\delta^4+40\delta^3-6\delta^2+88\delta-19$.
\end{enumerate}
\end{theorem}

\begin{remark}
The original instances of Theorem~\ref{thm:ville} that appear in the works of Ville~\cite{ville-negative,ville-positive} are that negatively $\tfrac14$-pinched oriented $4$-manifolds with finite volume satisfy $|\sigma(M)|\leq \frac{1}{3}\,\chi(M)$, and positively $\tfrac{4}{19}$-pinched oriented $4$-manifolds satisfy $|\sigma(M)| < \frac{1}{2}\,\chi(M)$. These statements derive, respectively, from $\lambda^\V\!\left(\frac{1}{4}\right)=\frac{1}{3}$ and $\lambda^\V\!\left(\frac{4}{19}\right)=\frac{2(97-7\sqrt{141})}{75}<\frac{1}{2}$.
\end{remark}

\begin{proof}[Proof of Theorem~\ref{thm:ville}]
Given a $\delta$-pinched oriented $4$-manifold $(M^4,\g)$ with finite volume, up to reversing its orientation, we shall assume $\sigma(M)\geq0$. Moreover, at each $p\in M$, its curvature operator $R_p\in\Sym^2_b(\wedge^2\R^4)$ satisfies $\pm R\in\dpinched$, see \eqref{eq:dpinched}.

For each $\lambda>0$, consider the quadratic form $\underline{\Delta}_\lambda\colon\Sym^2_b(\wedge^2\R^4)\to \R$ defined as
\begin{equation*}%\label{eq:defdeltalambda}
\underline{\Delta}_\lambda(R):=8\left(\underline{\chi}(R)-\tfrac{1}{\lambda}\,\underline{\sigma}(R)\right),
\end{equation*}
where $\underline{\chi}$ and $\underline{\sigma}$ are given by \eqref{eq:chi} and \eqref{eq:sigma}. That is, $\underline{\Delta}_\lambda(R)=8I_\lambda(R)$, where $I_\lambda$ is defined in~\eqref{eq:ilambda}, so writing $R$ in the canonical form \eqref{eq:curvop4}, we have:
\begin{equation}\label{eq:deltalambda}
\underline{\Delta}_\lambda(R)=6u^2+\left(1-\tfrac{2}{3\lambda}\right)|W_+|^2+\left(1+\tfrac{2}{3\lambda}\right)|W_-|^2-2|C|^2.
\end{equation}
Clearly, $\underline{\Delta}_\lambda(-R)=\underline{\Delta}_\lambda(R)$. Therefore, it follows from \eqref{eq:integrals} that if 
\begin{equation}\label{eq:goalville}
\min_{R\in\dpinched}\underline{\Delta}_\lambda(R)\geq0,
\end{equation}
then $\sigma(M)\leq \lambda \cdot \chi(M)$. Thus, it suffices to prove \eqref{eq:goalville} holds if $\lambda=\lambda^\V(\delta)$, cf.~\eqref{eq:goal-lambda}.

Fix $R\in\dpinched$. 
Using the same notation as \cite{ville-negative,ville-positive}, let $H_i\in\wedge^2_+\R^4$, $i=1,2,3$, be an orthonormal basis that diagonalizes $W_+$, and set
\begin{equation}\label{eq:wipVille}
w^+_i:=\langle W_+ H_i,H_i\rangle, \quad i=1,2,3,
\end{equation}
where $w_1^+\leq w_2^+\leq w_3^+$.
Consider the traceless Ricci component of $R$ as a linear map $C\colon \wedge^2_+\R^4\to\wedge^2_-\R^4$, which is denoted $Z_1$ in \cite{ville-negative,ville-positive}. Let $K_i\in\wedge^2_-\R^4$ be unit vectors such that\footnote{In other words, $K_i=\pm\frac{CH_i}{\|CH_i\|}$ if $CH_i\neq 0$, but $K_i$ can be chosen arbitrarily if $CH_i=0$.} $CH_i=c_i K_i$, where $c_i\in\R$, $i=1,2,3$, and set
\begin{equation}\label{eq:wimVille}
\begin{aligned}
\widehat  w_i^-&:=\langle W_- K_i,K_i\rangle, \quad i=1,2,3,\\
\alpha&:=\max_{1\leq i\leq 3} |\widehat w_i^-|.
\end{aligned}
\end{equation}
We stress that while $w_i^+$ are the eigenvalues of $W_+$ as in \eqref{eq:widef}, the numbers $\widehat  w_i^-$ defined in \eqref{eq:wimVille} in general \emph{do not agree} with the eigenvalues $w_i^-$ of $W_-$. Still, arguing as in \cite[Lemma 4]{ville-negative} and \cite[Lemme 1.4]{ville-positive}, it follows from \eqref{eq:wimVille} that
\begin{equation}\label{eq:villeW-}
|W_-|^2 \geq \tfrac32\alpha^2.
\end{equation}

Since the oriented Grassmannian \eqref{eq:grassmannian} can be written as
\begin{equation}\label{eq:grassmannian-realization}
\Gr^+(\R^4)=\left\{\dfrac{H+K}{\sqrt2}\in\wedge^2\R^4 : H\in\wedge^2_+\R^4,\, K\in\wedge^2_-\R^4,\, \|H\|=\|K\|=1 \right\},
\end{equation}
it follows from $R\in\dpinched$ and \eqref{eq:wipVille} that the quantities
\begin{equation}\label{eq:defvi}
v_i:=u+\frac{w_i^+}{2}=\frac{\langle R \sigma_i, \sigma_i\rangle + \langle R * \sigma_i, * \sigma_i\rangle}{2}, \quad i=1,2,3,
\end{equation}
where $\sigma_i:=\frac{1}{\sqrt2}(H_i+K_0)\in\Gr^+(\R^4)$ and $K_0\in\wedge^2_-\R^4$ is a unit vector chosen so that $\langle W_- K_0,K_0\rangle=0$, satisfy 
\begin{equation}\label{eq:vi}
\delta \leq v_i \leq 1, \quad \text{and} \quad \textstyle\sum\limits_{i=1}^3 v_i=3u.
\end{equation}
Similarly, using that $R\in\dpinched$ and $\frac{1}{\sqrt2}(H_i\pm K_i)\in\Gr^+(\R^4)$, we have
\begin{equation}\label{eq:cibound}
\delta\leq v_i+\frac{\widehat w_i^-}{2} \leq1, \quad \text{and}\quad |c_i|\leq m\!\left(v_i+\frac{\widehat w_i^-}{2}\right), \quad i=1,2,3,
\end{equation}
where $m$ is the piecewise affine function 
\begin{equation}\label{eq:defm}
\begin{aligned}
m &\colon [\delta,1] \longrightarrow\left[0,\tfrac{1-\delta}{2}\right]\\
m(x)&:=\min\left\{1-x,x-\delta \right\}.
\end{aligned}
\end{equation}
Thus, from \eqref{eq:wimVille} and \eqref{eq:cibound},
\begin{equation}\label{eq:boundC}
\begin{aligned}
|C|^2 &= \textstyle \sum\limits_{i=1}^3 c_i^2 \leq  \sum\limits_{i=1}^3 m\!\left(v_i+\frac{\widehat w_i^-}{2}\right)^2 \leq  \sum\limits_{i=1}^3 \left( m(v_i)+\frac{|\widehat w_i^-|}{2}\right)^2\\
&\textstyle\;\leq  \sum\limits_{i=1}^3 \left(m(v_i)^2+\alpha\, m(v_i) +\frac{1}{4}\alpha^2\right)
= \sum\limits_{i=1}^3 m(v_i)^2 +\alpha \sum\limits_{i=1}^3 m(v_i) + \frac{3}{4}\alpha^2.
\end{aligned}
\end{equation}

Combining \eqref{eq:villeW-}, \eqref{eq:defvi}, \eqref{eq:vi}, and \eqref{eq:boundC}, we arrive at Ville's main estimate:
\begin{align*}
\tfrac12\underline{\Delta}_\lambda(R)&=3u^2+\left(\tfrac12-\tfrac{1}{3\lambda}\right)|W_+|^2+\left(\tfrac12+\tfrac{1}{3\lambda}\right)|W_-|^2-|C|^2\\
&=\textstyle\left(\frac{4}{9\lambda}-\frac{1}{3}\right)\left(\sum\limits_{i=1}^3 v_i\right)^2  +\left(2-\frac{4}{3\lambda}\right)\sum\limits_{i=1}^3 v_i^2  +\left(\tfrac12+\tfrac{1}{3\lambda}\right)|W_-|^2-|C|^2\\
&\geq \textstyle\left(\frac{4}{9\lambda}-\frac{1}{3}\right)\left(\sum\limits_{i=1}^3 v_i\right)^2  +\left(2-\frac{4}{3\lambda}\right)\sum\limits_{i=1}^3 v_i^2 
+ \tfrac{1}{2\lambda}\alpha^2 -\alpha \sum\limits_{i=1}^3 m(v_i) -\sum\limits_{i=1}^3 m(v_i)^2  \\
&= \textstyle\left(\frac{4}{9\lambda}-\frac{1}{3}\right)\left(\sum\limits_{i=1}^3 v_i\right)^2  +\left(2-\frac{4}{3\lambda}\right)\sum\limits_{i=1}^3 v_i^2  -\frac{\lambda}{2}\left(\sum\limits_{i=1}^3 m(v_i)\right)^2 
 -\sum\limits_{i=1}^3 m(v_i)^2 \\
&\quad \textstyle + \left(\frac{1}{\sqrt{2\lambda}}\,\alpha -\sqrt\frac{\lambda}{2}\sum\limits_{i=1}^3 m(v_i) \right)^2 \\
&\geq F_\lambda(v_1,v_2,v_3),
\end{align*}
where $F_\lambda\colon V_\delta \to\R$ is the piecewise quadratic function
\begin{equation*}
F_\lambda(v_1,v_2,v_3):=\textstyle\left(\frac{4}{9\lambda}-\frac{1}{3}\right)\left(\sum\limits_{i=1}^3 v_i\right)^2  +\left(2-\frac{4}{3\lambda}\right)\sum\limits_{i=1}^3 v_i^2  -\frac{\lambda}{2}\left(\sum\limits_{i=1}^3 m(v_i)\right)^2 
 -\sum\limits_{i=1}^3 m(v_i)^2
\end{equation*}
on the polyhedron $V_\delta:=\big\{(v_1,v_2,v_3)\in\R^3:\delta\leq v_1\leq v_2\leq v_3\leq 1\big\}$. More precisely, from \eqref{eq:defm}, the restriction $F_\lambda^i:=\big(F_\lambda\big)|_{V_\delta^i}$ of the above to each of the subpolyhedra
\begin{equation}\label{eq:defVdi}
\begin{aligned}
V_\delta^1 &:=\big\{(v_1,v_2,v_3)\in V_\delta : \tfrac{\delta+1}{2}\leq v_1\leq v_2\leq v_3\leq 1\big\},\\
V_\delta^2 &:=\big\{(v_1,v_2,v_3)\in V_\delta : \delta \leq v_1\leq\tfrac{\delta+1}{2}\leq v_2\leq v_3\leq 1 \big\},\\
V_\delta^3 &:=\big\{(v_1,v_2,v_3)\in V_\delta : \delta \leq v_1\leq v_2\leq\tfrac{\delta+1}{2}\leq v_3\leq 1 \big\},\\
V_\delta^4 &:=\big\{(v_1,v_2,v_3)\in V_\delta : \delta\leq v_1\leq v_2\leq v_3\leq \tfrac{\delta+1}{2}\big\},
\end{aligned}
\end{equation}
is a quadratic form $F_\lambda^i\colon V_\delta^i\to\R$; and, clearly, $V_\delta=\bigcup_{i=1}^4 V_\delta^i$. Therefore, in order to show that $\lambda = \lambda^\V(\delta)$ implies \eqref{eq:goalville}, it suffices to show that it implies
\begin{equation}\label{eq:goalville1}
\min_{V_\delta^i} F_\lambda^i \geq0, \quad 1\leq i\leq 4;
\end{equation}
which we shall now prove using Corollary~\ref{cor:smallsig} with $Q=-F_\lambda^i$ and $K=V_\delta^i$.

According to \eqref{eq:defVdi}, we have that the vertices $q_j^i$ of the polyhedron $V_\delta^i$, where $1\leq j\leq 4$ if $i\in \{1,4\}$, and $1\leq j\leq 6$ if $i\in\{2,3\}$, are given as follows:

\medskip

$q_1^1=\left(\frac{\delta+1}{2},\frac{\delta+1}{2},\frac{\delta+1}{2}\right)$, \
$q_2^1=\left(\frac{\delta+1}{2},\frac{\delta+1}{2},1\right)$, \
$q_3^1=\left(\frac{\delta+1}{2},1,1\right)$, \
$q_4^1=\left(1,1,1\right)$,

\smallskip

$q_1^2=\left(\delta,\frac{\delta+1}{2},\frac{\delta+1}{2}\right)$, \
$q_2^2=\left(\delta,\frac{\delta+1}{2},1\right)$, \
$q_3^2=\left(\delta,1,1\right)$, \
$q_4^2=q_1^1$, \
$q_5^2=q_2^1$, \
$q_6^2=q_3^1$,

\smallskip

$q_1^3=\left(\delta,\delta,\frac{\delta+1}{2}\right)$, \
$q_2^3=\left(\delta,\delta,1\right)$, \
$q_3^3=q_1^2$, \
$q_4^3=q_2^2$, \
$q_5^3=q_1^1$, \
$q_6^3=q_2^1$,

\smallskip

$q_1^4=\left(\delta,\delta,\delta\right)$, \
$q_2^4=q_1^3$, \
$q_3^4=q_1^2$, \
$q_4^4=q_1^1$.

\medskip

\noindent
Note that $V_\delta^1$ and $V_\delta^4$ are $3$-simplices (i.e., tetrahedra), but $V_\delta^2$ and $V_\delta^3$ are convex hulls of $6$~distinct points in $\R^3$. However, it is easy to see that $V_\delta^2$ and $V_\delta^3$ are \emph{prisms}, i.e., isometric to a product of a $2$-simplex (the convex hull of $3$ of its vertices) and a $1$-simplex (the convex hull of $2$ of its vertices). Thus, it is straightforward to check which subsets of $q_j^2$ and $q_j^3$ determine faces of $V_\delta^2$ and $V_\delta^3$, respectively.

\begin{table}[!ht]  
\begin{tabular}{|c|l|}
\hline
 $q_j^i$ & $F^i_\lambda(q_j^i)$ $\phantom{\Big|}$ \\[2pt]
\hline
$q_1^1$ & $-\frac{9  \lambda}{8} \delta ^2 +\left(\frac{9 \lambda }{4}+3\right)\delta -\frac{9 \lambda }{8}$ \rule[4pt]{0pt}{8pt} \\[3pt]
$q_2^1$ & $\left(\frac{1}{6}- \frac{\lambda}{2}-\frac{2}{9\lambda} \right) \delta^2 + \left(\frac{5}{3}+\lambda +\frac{4}{9 \lambda }\right)\delta+\frac{7}{6} -\frac{\lambda }{2}-\frac{2}{9 \lambda }$ \rule[4pt]{0pt}{8pt} \\[3pt]
$q_3^1$ & $\left(\frac{1}{6} -\frac{\lambda }{8}-\frac{2}{9 \lambda }\right)\delta^2  +\left(\frac{2}{3}+\frac{\lambda}{4}+\frac{4}{9\lambda}\right) \delta+\frac{13}{6}-\frac{\lambda }{8}-\frac{2}{9 \lambda}$ \rule[4pt]{0pt}{8pt} \\[3pt]
$q_4^1$ & $3$ \rule[4pt]{0pt}{8pt} \\[3pt]
\hline
$q_1^2$ & $ \left(\frac{7}{6}-\frac{\lambda }{2}-\frac{2}{9 \lambda }\right)\delta^2 + \left(\frac{5}{3}+\lambda +\frac{4}{9 \lambda }\right)\delta + \frac{1}{6}-\frac{\lambda}{2}-\frac{2}{9\lambda} $ \rule[4pt]{0pt}{8pt} \\[3pt]
$q_2^2$ & $\left(\frac{3}{2}-\frac{\lambda }{8}-\frac{2}{3 \lambda }\right)\delta^2 +\left(\frac{\lambda }{4}+\frac{4}{3 \lambda }\right)\delta +\frac{3}{2}-\frac{\lambda }{8}-\frac{2}{3 \lambda } $  \rule[4pt]{0pt}{8pt} \\[3pt]
$q_3^2$ & $\left(\frac{5}{3}-\frac{8}{9 \lambda }\right)\delta ^2 +\frac{4}{9} \left(\frac{4}{\lambda }-3\right) \delta +\frac{8}{3}-\frac{8}{9 \lambda } $ \rule[4pt]{0pt}{8pt} \\[3pt]
\hline
$q_1^3$ & $\left(\frac{13}{6}-\frac{\lambda }{8}-\frac{2}{9 \lambda }\right)\delta^2 +\left(\frac{2}{3}+\frac{\lambda}{4}+\frac{4}{9\lambda}\right)\delta+\frac{1}{6}-\frac{\lambda }{8}-\frac{2}{9 \lambda }$ \rule[4pt]{0pt}{8pt} \\[3pt]
$q_2^3$ & $ \left(\frac{8}{3}-\frac{8}{9 \lambda }\right)\delta^2+\frac{4}{9} \left(\frac{4}{\lambda }-3\right)\delta+\frac{5}{3} -\frac{8}{9 \lambda }  $  \rule[4pt]{0pt}{8pt} \\[3pt]
\hline
$q_1^4$ & $ 3\delta^2$ \rule[4pt]{0pt}{8pt} \\[3pt]
$q_2^4$ & $ \left(\frac{13}{6}-\frac{\lambda }{8}-\frac{2}{9 \lambda }\right)\delta^2+\left(\frac{2}{3}+\frac{\lambda}{4}+\frac{4}{9\lambda}\right)\delta+\frac{1}{6} -\frac{\lambda }{8}-\frac{2}{9 \lambda } $  \rule[4pt]{0pt}{8pt} \\[3pt]
\hline
\end{tabular}
\caption{Values of $F_\lambda^i$ on the vertices $q_j^i$ of $V_\delta^i$, $1\leq i\leq 4$. The suppressed entries are equal to some other entry in the table, namely
$F^4(q_2^4)=F^3(q_1^3)$, $F^3(q_4^3)=F^2(q_2^2)$, $F^2(q_6^2)=F^1(q_3^1)$,
$F^4(q_4^4)=F^3(q_5^3)=F^2(q_4^2)=F^1(q_1^1)$,
$F^3(q_6^3)=F^2(q_5^2)=F^1(q_2^1)$, and 
$F^4(q_3^4)=F^3(q_3^3)=F^2(q_1^2)$.}\label{tab:Fdi0dim}
\end{table}

Routine computations show that the Hessian of each $F_\lambda^i$ is negative-definite if $\lambda<\frac{4}{3}$, and has exactly $d=2$ positive eigenvalues if $\lambda>\frac{4}{3}$; so it suffices to inspect faces of $V_\delta^i=\conv(q_j^i)$ that have dimension $\leq 2$, i.e., convex hulls of no more than $3$ different $q_j^i$'s.
By direct inspection, we find that, for all $0<\delta< 1$, $\lambda>0$, and $1\leq i\leq 4$, the Hessian of the restrictions of $F_\lambda^i$ to all $1$- and $2$-dimensional faces of $V_\delta^i$ is \emph{not} positive-definite. Thus, by Corollary~\ref{cor:smallsig}, we have that
$\min_{V_\delta^i} F_\lambda^i=\min F_\lambda^i(q_j^i)$ is the smallest among the values assumed by $F_\lambda^i$ on the vertices $q_j^i$ of $V_\delta^i$, which are given in Table~\ref{tab:Fdi0dim}.

Similarly to the last step in the proof of Theorem~\ref{thm:estimates-general}, applying Algebraic Cylindrical Decomposition to the semialgebraic set of $\R^2$ given by
\begin{equation*}
\mathfrak X^\V:=\left\{(\delta,\lambda)\in H :  F^i_\lambda(q_j^i) \geq 0, \text{ for all } i, j \right\},
\end{equation*}
where $H=(0,1)\times(0,+\infty)$, we obtain that $\mathfrak X^\V=\left\{\delta\in\left[\delta_0^\V,1\right) : \lambda^\V(\delta)\leq \lambda\leq \overline{\lambda}(\delta)\right\}$, where $\lambda^\V\colon \left[\delta_0^\V,1\right]\to\R$ is the piecewise continuous function in the statement, and $\overline{\lambda}(\delta):=\frac{8\delta}{3(1-\delta)^2}$. Note that
$\lambda^\V(\delta_0^\V)=\overline{\lambda}(\delta_0^\V)$, and at least one among $F_\lambda^1(q_1^1)$ and $F_\lambda^2(q_1^2)$ is negative if $\delta<\delta_0^\V$, so \eqref{eq:goalville1} does not hold. However, if $\delta\in\left[\delta_0^\V,1\right)$, then $(\delta,\lambda^\V(\delta))\in\mathfrak X^\V$ implies
that \eqref{eq:goalville1} and hence \eqref{eq:goalville} hold. The desired conclusion also holds at the right endpoint $\delta=1$, where $\lambda^\V(1)=0$, since constant curvature manifolds are locally conformally flat, and thus have zero~signature. 
\end{proof}

\begin{remark}
There are two noteworthy differences between the semialgebraic sets $\mathfrak X$ and $\mathfrak X^\V$ 
of $(\delta,\lambda)$ for which the crucial lower bounds in the proofs of Theorems~\ref{thm:estimates-general} and \ref{thm:ville}, respectively, yield nonnegative quantities (as desired). First, the projection of $\mathfrak X$ onto $0<\delta<1$ is surjective, which does not hold for $\mathfrak X^\V$. Second, for any given $\delta$ in this projection, the interval of $\lambda>0$ for which $(\delta,\lambda)\in\mathfrak X$ is not bounded from above, while it is for $\mathfrak X^\V$. Clearly, both stem from the presence of the upper bound $\overline{\lambda}(\delta)$ in the cylindrical decomposition of $\mathfrak X^\V$, which on $\mathfrak X$ it is simply $+\infty$.
\end{remark}

\begin{remark}
It was uncovered in our communications with Ville that there is a small mistake in \cite[p.~333-334]{ville-negative} and \cite[p.~152]{ville-positive}, where the coefficient $-\frac{1}{7}$, respectively $-\frac{2}{5}$, of the term $\big(\sum_{i=1}^3 m(v_i)\big)^2$, should be replaced by
$-\frac{1}{6}$, respectively $-\frac{1}{4}$; i.e., this coefficient should be equal to $-\frac{\lambda}{2}$, as in the definition of~$F_\lambda$. 
This arises from erroneously assuming that the sum of the numbers $\widehat w_i^-$ in \eqref{eq:wimVille} vanishes, which would imply that the sum of their squares is bounded above by $2\alpha^2$. This need not be the case, and the sum of $(\widehat w_i^-)^2$ is only bounded above by $3\alpha^2$. Indeed, denoting by $K\in\textnormal{Mat}_{3\times 3}(\R)$ the matrix whose columns are the coordinates of $K_i$ with respect to a fixed orthonormal basis of $\wedge^2_-\R^4$, we have that 
\begin{equation*}
\textstyle 0=\sum\limits_{i=1}^3 w_i^-=\tr W_-=\tr\big(K W_- K^\mathrm t\big)=\tr \big(W_- (K^\mathrm tK)\big),
\end{equation*}
since $K^\mathrm t K=\id$; however
\begin{equation*}
\textstyle  \sum\limits_{i=1}^3 \widehat w_i^-=\tr \big(K^\mathrm t W_- K\big)=\tr  \big(W_- (KK^\mathrm t)\big)
\end{equation*}
may not vanish, as $KK^\mathrm t$ may not be equal to $\id$.
Fortunately, the rest of the proofs in \cite{ville-negative,ville-positive} can be modified accordingly, e.g., following the above proof of Theorem~\ref{thm:ville}, without any impact on the main result.
\end{remark}

\providecommand{\bysame}{\leavevmode\hbox to3em{\hrulefill}\thinspace}
\providecommand{\MR}{\relax\ifhmode\unskip\space\fi MR }
\providecommand{\MRhref}[2]{%
  \href{http://www.ams.org/mathscinet-getitem?mr=#1}{#2}
}
\providecommand{\href}[2]{#2}

\end{document}